\documentclass[reqno,10pt,a4paper]{amsart}

\usepackage{amssymb}
\usepackage{mathtools}
\usepackage[shortlabels]{enumitem}
\usepackage[utf8]{inputenc}
\usepackage{color}
\usepackage{makecell}
\usepackage{microtype}
\usepackage{tikz}
\usetikzlibrary{positioning,decorations.pathreplacing,decorations.markings,arrows.meta,calc,fit,quotes,cd,math,arrows,backgrounds,shapes.geometric}
\usepackage[a4paper,top=3cm,bottom=3cm,inner=2.5cm,outer=2.5cm]{geometry}
\usepackage{graphicx}
\definecolor{darkgreen}{rgb}{0,0.45,0}
\usepackage{bm}
\usepackage{cmll}

\usepackage[ocgcolorlinks]{hyperref}
\hypersetup{
	colorlinks = true,
	citecolor=darkgreen,
	urlcolor=gray,
	linkcolor = blue,
	final = true,
	linktoc=page
}

\usepackage[capitalize]{cleveref}

\numberwithin{equation}{section}

\crefname{equation}{}{}
\crefname{lem}{Lemma}{Lemmas}
\crefname{thm}{Theorem}{Theorems}
\crefname{defi}{Definition}{Definitions}
\crefname{conj}{Conjecture}{Conjectures}
\crefname{ex}{Example}{Examples}
\crefname{sec}{Section}{Sections}
\crefname{prop}{Proposition}{Propositions}
\crefname{section}{Section}{Sections}

\newcommand{\note}[3][]{\def\auth{#1}\textcolor{#2}{{\ifx\auth\empty\else\auth: \fi{#3}}}}

\theoremstyle{plain}
\newtheorem{theorem}{Theorem}[section]
\newtheorem{corollary}[theorem]{Corollary}
\newtheorem{lemma}[theorem]{Lemma}
\newtheorem{proposition}[theorem]{Proposition}

\theoremstyle{definition}
\newtheorem{definition}[theorem]{Definition}
\newtheorem{hypothesis}[theorem]{Hypothesis}
\newtheorem{remark}[theorem]{Remark}

\newtheorem{convention}[theorem]{Convention}

\newtheorem*{replemmax}{\reptitle}
 {\end{replemmax}}

\newcommand{\ie}{\text{i.e.\ }}
\newcommand{\eg}{\text{e.g.\ }}
\newcommand{\cf}{\text{cf.\ }}

\newcommand{\myemph}{\textit} 

\newcommand{\defeq}{=_{\operatorname{def}}}
\newcommand{\co}{\colon}

\newcommand{\op}{\mathrm{op}}

\newcommand{\coend}{\int}

\newcommand{\cat}[1]{\mathcal{#1}}

\newcommand{\catE}{\cal{E}}

\newcommand{\catK}{\cat{K}}
\newcommand{\catL}{\cat{L}}
\newcommand{\psh}{\mathsf{Psh}}
\newcommand{\pshA}{\psh(A)}
\newcommand{\pshB}{\psh(B)}

\newcommand{\cal}[1]{\mathcal{#1}}

\newcommand{\Psh}{\mathsf{Psh}}

\newcommand{\id}{1}

\newcommand{\CoMon}[1]{\mathsf{CMon}(#1)}
\newcommand{\BraCoMon}[1]{\mathsf{BraCMon}(#1)}
\newcommand{\SymCoMon}[1]{\mathsf{SymCMon}(#1)}

\newcommand{\Mon}[1]{\mathsf{Mon}(#1)}
\newcommand{\BraMon}[1]{\mathsf{BraCMon}(#1)}
\newcommand{\SymMon}[1]{\mathsf{SymMon}(#1)}

\newcommand{\com}{n}
\newcommand{\cou}{e}

\newcommand{\Em}[1]{{#1}^\bang}
\newcommand{\Kl}[1]{{#1}_\bang}

\newcommand{\unit}{I}
\newcommand{\bra}{\mathsf{r}}
\newcommand{\bang}{\oc}
\newcommand{\bbang}{\oc \oc}
\newcommand{\bbbang}{\oc \oc \oc}

\newcommand{\mon}{\mathsf{m}}
\newcommand{\monn}{\mon^2}
\newcommand{\moni}{\mon^0}

\newcommand{\dig}{\mathsf{p}}
\newcommand{\der}{\mathsf{d}}
\newcommand{\con}{\mathsf{c}}
\newcommand{\weak}{\mathsf{w}}
\newcommand{\cocon}{\bar{\con}}
\newcommand{\coweak}{\bar{\weak}}
\newcommand{\coder}{\bar{\der}}

\newcommand{\seel}{\mathsf{s}}
\newcommand{\seell}{\seel^2}
\newcommand{\seeli}{\seel^0}

\newcommand{\term}{\top}

\newcommand{\linhom}{\multimap}

\newcommand{\Set}{\mathsf{Set}}
\newcommand{\Rel}{\mathsf{Rel}}
\newcommand{\Cat}{\mathsf{Cat}}
\newcommand{\CAT}{\mathsf{CAT}}
\newcommand{\Prof}{\mathsf{Prof}}
\newcommand{\SMonCoc}{\mathsf{SMonCoc}}

\newcommand{\Sym}{\mathsf{CatSym}}
\newcommand{\PP}{\mathsf{P}}

\tikzset{tick/.style={postaction={decorate,decoration={markings,mark=at
position 0.5 with {\draw[-] (0,.4ex) -- (0,-.4ex);}}}}}
\tikzset{bigtick/.style={postaction={decorate,decoration={markings,mark=at
position 0.5 with {\draw[-] (0,.6ex) -- (0,-.6ex);}}}}}

\newcommand{\sbul}{\scriptstyle\bullet}
\tikzset{bul/.style={postaction={decoration={markings,mark=at position 0.5 with
{\node{$\sbul$};}},decorate}}}

\tikzset{Rightarrow/.style={double equal sign distance,>={Implies},->},
triple/.style={-,preaction={draw,Rightarrow}}}

\newcommand{\Two}{\scriptstyle\Downarrow}
\newcommand{\TwoHor}{\scriptstyle\Rightarrow}

\newcommand{\arghole}{-}

\makeatletter
\def\ignorespacesandallpars{%
  \@ifnextchar\par
    {\expandafter\ignorespacesandallpars\@gobble}%
    {}%
}
\makeatother

\def\lam#1{{\lambda}\@lamarg#1:\@endlamarg\@ifnextchar\bgroup{.\,\lam}{.\,}}
\def\@lamarg#1:#2\@endlamarg{\if\relax\detokenize{#2}\relax #1\else\@lamvar{\@lameatcolon#2},#1\@endlamvar\fi}
\def\@lamvar#1,#2\@endlamvar{(#2\,{:}\,#1)}
\def\@lameatcolon#1:{#1}

\def\lamu#1{{\lambda}\@lamuarg#1:\@endlamuarg\@ifnextchar\bgroup{.\,\lamu}{.\,}}
\def\@lamuarg#1:#2\@endlamuarg{#1}

\numberwithin{equation}{section}

\title[Monoidal bicategories, differential linear logic, and analytic functors]{Monoidal bicategories, differential linear logic,\\ and analytic functors }
\begin{document}

\begin{abstract}
We develop further the theory of monoidal bicategories by introducing and studying bicategorical counterparts of the notions of a linear exponential comonad, as considered in the study of linear logic, and of a codereliction transformation, introduced to study differential linear logic via differential categories. As an application, we extend the differential calculus of
Joyal's analytic functors to analytic functors between presheaf categories, just as ordinary calculus extends from a
single variable to many variables. 
\end{abstract}

\author[M. Fiore]{Marcelo Fiore}
\address{Department of Computer Science and Technology, University of Cambridge}
\email{Marcelo.Fiore@cl.cam.ac.uk}

\author[N. Gambino]{Nicola Gambino}
\address{Department of Mathematics, University of Manchester}
\email{nicola.gambino@manchester.ac.uk}

\author[M. Hyland]{Martin Hyland}
\address{Department of Pure Mathematics and Mathematical Statistics, University of Cambridge}
\email{M.Hyland@dpmms.cam.ac.uk}

\date{\today}

\subjclass{18N10, 18M45, 18F40, 18D60, 18M80.}
\keywords{Monoidal bicategory, differential linear logic, profunctor, symmetric sequence, analytic functor}

\maketitle

\tableofcontents

\section{Introduction}
\label{sec:intro}

\subsection*{Context and motivation} 
The aim of this paper is to develop  and connect two apparently distant strands of research: low-dimensional category theory and differential linear logic. Let us begin by providing some context and motivation for our work.

By low-dimensional category theory we mean here the study of two-dimensional and  three-dimensional categories \cf \cite{BenabouJ:intb,KellyG:revetc} and~\cite{GordonR:coht,GurskiN:cohtdc} respectively. The subject has grown enormously in the last decades, with motivation coming both from within category theory itself and from other parts of mathematics. Indeed, just as it is useful to study standard set-based mathematical structures (such as groups and vector spaces) by assembling them into categories, it is natural to investigate category-based mathematical structures (\eg monoidal categories and Grothendieck toposes) by forming appropriate two-dimensional categories (see~\cite{BlackwellR:twodmt} for example), and so on. Furthermore, low-dimensional categorical structures find applications in algebra, algebraic topology, topological quantum field theory, \eg to obtain more informative invariants for mathematical objects (\eg knots), via the research programme known as categorification~\cite{BaezJ:cat}. One of the key advances 
in this area has been the development of the theory of monoidal bicategories, initiated by Breen in~\cite{BreenL:letpds} and by Kapranov and Voevodsky in~\cite{KapranovM:bramcm,KapranovM:2czt}, and continued by many others, \eg in~\cite{BaezJ:higdab,CransS:gencbs,DayStreet,ElguetaJ:cohdtm,GurskiN:looscm,GurskiN:inlsc}.  This promises to serve a role as important as that of monoidal categories, \eg in the study of topological quantum field theories, \cf \cite{KongL:cenmcd,Schommer-PriesC:clatdc}.

The other area involved in this paper is differential linear logic, introduced by Ehrhard and Reignier in~\cite{EhrhardT:intdll,EhrhardRegnier}, as an extension of linear logic, introduced by Girard in~\cite{GirardJ:linl}.
The subject arose from the observation that many models of linear logic, such as those based on categories of topological vector spaces~\cite{EhrhardT:kotssl}, possess a well-behaved notion of differentiation. Remarkably, not only  is it possible to introduce a syntactic counterpart of it, but the idea leads to interesting applications to the $\lambda$-calculus, obtained by introducing a Taylor series expansion for $\lambda$-terms~\cite{EhrhardT:bohtkm}. 

The study of differential linear logic led also to the introduction of differential categories~\cite{BluteR:difcr,SeelyEtAl}, which provide categorical models of differential linear logic.
Since  differential categories axiomatise the essential categorical structure required to have a well-behaved operation of differentiation on the maps of a category, they have a wide range of examples across various areas of mathematics. 
Their axioms are formulated on top of those for a categorical model of linear logic, which is given by a symmetric monoidal category $(\catE, \otimes, \unit)$ equipped with a linear exponential comonad $\bang(-) \co \catE \to \catE$,  \ie a symmetric monoidal comonad satisfying some additional axioms~\cite{MelliesPA:catsll}. In this setting, we think of maps $A \to B$ as being `linear' and of Kleisli maps $\bang A \to B$ as being `non-linear'.  Following the idea that the derivative of a function is a function in two arguments, linear in one and non-linear in the other, the derivative of a map~$f \co \bang A \to B$ in a differential category has the form~$\mathrm{d} f \co A \otimes \bang A \to B$. In the presence of sufficient structure on the ambient category, this operation is equivalent to having a natural transformation, which is referred to in the literature either as a codereliction transformation~\cite{BluteR:difcr,EhrhardT:intdll} or as a creation map~\cite{FioreM:difsmi}, with components of the form~$\coder_A \co A \to \bang A$, and subject to a few axioms. Remarkably, these axioms allow us to derive counterparts of all the
basic results on differentiation~\cite{BluteR:difcr}. 

Here, we apply the idea of categorification to the study of models of linear logic and differential linear logic and contribute to developing a new theory, unavoidably more subtle and complex than the existing one, based on symmetric monoidal bicategories. This is of interest for logic and theoretical computer science since the additional layer of structure present in bicategories, namely 2-cells (\ie morphisms between morphisms), can be used to model computational rewriting steps between terms~\cite{PowerJ:absfr,SeelyRAG:modc2}.

The origin of this line of investigation can be traced back to categorifications of the relational model of linear logic~\cite{CattaniG:proomb}. Recall that the relational model arises by considering the monad on the category of sets and functions whose algebras are commutative monoids and extending it to a monad on the category of sets and relations, written~$\Rel$ here. The duality available on~$\Rel$ allows us to turn this monad into a comonad~$\bang(-) \co \Rel \to \Rel$, which can then be shown to satisfy all the axioms for a linear exponential comonad, so that its Kleisli category is cartesian closed.

\begin{table}[htb]
\begin{tabular}{|c|c|}\hline
\textsf{Standard version} & \textsf{Categorified version} \\ \hline
$\Set$ & $\Cat$ \\  \hline 
$\Rel$ & $\Prof$ \\ \hline
\makecell{
Free commutative monoid \\ monad 
$\wn(-)\co \Set \to \Set$} & \makecell{
Free symmetric strict monoidal category \\ 2-monad 
$\wn(-) \co \Cat \to \Cat$} \\   \hline
\makecell{Linear exponential \\
comonad $\bang(-) \co \Rel \to \Rel$} &
\makecell{Linear exponential \\
pseudocomonad $\bang(-) \co \Prof \to \Prof$}  \\
\hline
\makecell{
Kleisli category $\Kl{\Rel}$} & 
\makecell{
Kleisli bicategory $\Kl{\Prof}$} \\ \hline
\end{tabular} \medskip
\caption{The relational model and its categorification.} 
\label{tab:rel-prof}
\end{table}

These ideas were categorified in~\cite{CattaniG:proomb} by replacing sets with small categories and relations with profunctors, also known as distributors or bimodules~\cite{BenabouJ:intb,BenabouJ:disw}. As typical  in the process of categorification, there is now additional freedom, since there are several 2-monads on the 2-category of small categories and functors, written $\Cat$ here, that may be considered to take the place of the monad for commutative monoids. 

An example of particular interest, illustrated in~\cref{tab:rel-prof}, arises by considering the 2-monad on $\Cat$ whose strict algebras are symmetric strict monoidal categories. By a form of pseudodistributivity \cite{FioreM:klebrp}, this 2-monad extends to a pseudomonad on the bicategory of profunctors, written 
$\Prof$, where one obtains a pseudocomonad by duality, as before.
The associated Kleisli bicategory, called the bicategory of \emph{categorical symmetric sequences} here,
and written~$\Sym$, was originally introduced in~\cite{FioreM:carcbg}, where it was shown to be cartesian closed. This bicategory was investigated further in connection with the theory of operads in~\cite{GambinoN:opebaf}, and shown to admit rich additional structure in~\cite{GalalZ:fixo2c,GambinoN:monkba}.

The distinction between linear and non-linear maps acquires particular significance in this example: a linear map here is a map $F \co A \to B$ in $\Prof$, \ie a functor
$F \co B^\op \times A \to \Set$. Such a functor $F$ determines canonically a functor $F^\dag\co \pshA \to \pshB$ between presheaf categories, defined by a coend formula which is reminiscent of the expression for the linear map associated to a matrix:
\[
F^\dag(X,b) = \int^{a \in A} F[b, a] \times X(a) .
\]
By contrast, a non-linear map is a map $F \co  A \to B$ in $\Sym$, \ie a profunctor $F \co \bang A \to B$.
Such a profunctor~$F$ induces a functor $F^\ddag \co \pshA \to \pshB$ between presheaf categories, 
defined by a formula similar to the one for the 
Taylor series expansion of an analytic function: 
\[
F^\ddag(X,b) = \int^{\alpha \in \bang A} F[b, \alpha] \times X^{\alpha} .
\]
Here, $X^{\alpha} = X(a_1) \times \ldots \times X(a_n)$, for $X \in \pshA$ and $\alpha = \langle a_1, \ldots, a_n \rangle \in \bang A$.
We call such functors \emph{analytic} since they generalise the analytic functors on $\Set$ introduced by Joyal in~\cite{JoyalA:fonaes} as part of his approach to enumerative combinatorics~\cite{BergeronF:comstl,JoyalA:thecsf}.
(Joyal's analytic functors arise when $A = B = \mathsf{1}$.)

These ideas led to  a new line of research, outlined in~\cite{HylandM:somrgd}, aimed at extending theory and examples of categorical models of linear logic to the two-dimensional setting. This provides a so-called `quantitative semantics' for a variety of logical systems which are of interest also in theoretical computer science, \cf \cite{GalalZ:bicocl,FioreM:stapss,JacqC:catcnis,OlimpieriF:inttrc,OngL:quaslc,TsukadaT:gensrr,TsukadaT:spepte}.

The first goal of this paper is to establish the analogy between the relational and profunctor models on a more precise basis, by  developing a bicategorical counterpart of the standard theory of models of linear logic, which recovers the results on profunctors as a special case. The motivation for creating such a theory can readily be seen by observing that, while many facts about the relational model follow from the theory of linear exponential comonads, analogous results in the two-dimensional context have been proved on a case-by-case basis. While some first steps towards a bicategorical theory have been taken recently~\cite{GalalZ:bicocl,JacqC:catcnis,MirandaA:paper-1,OlimpieriF:inttrc}, much foundational work remains yet to be done. Here, we address this issue by considering counterparts of some key notions 
and results on models of linear logic, in particular those related to the notion of a linear exponential comonad considered in~\cite{HylandM:gluoml}.

The second goal of this work is to provide a bicategorical counterpart of a special class of differential categories~\cite{BluteR:difcr,SeelyEtAl} and show that the bicategory of profunctors is an example of this new notion.
The motivation for this is twofold. On the one hand we wish to make precise the analogy with the relational model of differential linear logic; on the other hand, we want to provide a clean approach to extending the differential calculus for analytic functors on $\Set$ developed in~\cite{JoyalA:fonaes} to analytic functors between presheaf categories, analogously to how calculus in a single variable extends to many variables,
or to the theory of differentiation in the distinct context of polynomial functors~\cite{AltenkirchT:delfd}.
 To the best of our knowledge, this is the first example of a genuinely two-dimensional model of differential linear logic.

\subsection*{Main contributions}

Our first main contribution is the definition and study of the notion of a linear exponential pseudocomonad, which we introduce in \cref{thm:linear-exponential-comonad} as a bicategorical counterpart of the notion of a linear exponential comonad  in~\cite{HylandM:gluoml}. We support this definition by generalising several facts about one-dimensional models of linear logic to the two-dimensional setting.  We then prove results (\cref{thm:case-1} and \cref{thm:case-2}) that help us to construct linear exponential pseudocomonads in many cases of interest, including in our application to the bicategory of profunctors. 

Our second main contribution is the exploration of consequences of the assumption of additional structure and properties on the ambient bicategory, offering a modular development of the theory.  In particular, under the additional assumption of existence of finite products, we construct the so-called Seely equivalences and show that they provide the structure of a sylleptic strong monoidal 2-functor (\cref{thm:seely-equivalences-monoidal}). Because of coherence issues, a direct approach seems daunting. We therefore offer a more conceptual treatment, inspired by a passing remark in~\cite{GarnerR:hyplem}, which allows us to handle coherence in an efficient way. Overall, we prove bicategorical counterparts of all the diagrams considered in~\cite{EhrhardT:intdll} and ~\cite{FioreM:difsmi}, starting from our notion of linear exponential pseudocomonad.

The third main contribution of this paper is the definition of a bicategorical
counterpart of the codereliction transformation, given in \cref{thm:codereliction}, which provides one possible way to define a derivation operation. The definition of a codereliction transformation considered here is a two-dimensional version of the one introduced in~\cite{FioreM:difsmi}. While the notion considered therein (under the name of creation map) is equivalent to the one considered in the context of differential categories in the one-dimensional setting, as shown in~\cite{BluteR:difcr},  it offers some advantages when transported to the two-dimensional setting since, as discussed below, it allows us to deal effectively with coherence issues. In order to do this, we work under additional assumptions, including that the ambient bicategory has biproducts and that the induced convolution structure is a coproduct. These hold in our application.

Finally, we show that the bicategory of categorical symmetric sequences admits a codereliction (\cref{thm:sym-has-differentiation}), which allows us to extend the derivative operation to analytic functors between presheaf categories. To achieve this result, we show that the pseudomonad obtained by extending to $\Prof$ the free symmetric strict monoidal category 2-monad on $\Cat$ is a linear exponential pseudocomonad in our sense (\cref{thm:prof-degenerate-model}) and that the additional hypotheses underpinning our definition of a codereliction transformation are satisfied (\cref{thm:sym-has-differentiation}).
This then determines the desired operation of differentiation for analytic functors between presheaf categories.

\subsection*{Technical aspects} The development of the paper involved overcoming a number of conceptual and technical challenges. First of all, the bicategorical definitions underpinning the subject are rather complex, as they involve a significant amount of data and a large number of coherence conditions. Because of this, extending results from the one-dimensional to the two-dimensional setting involves a combination of trivial and non-trivial aspects. While it is generally possible to guess what the desired statements are, their proofs usually  require long calculations. A good illustration of this point that can be found in the existing literature is the statement that a bicategory with finite products admits a canonical symmetric monoidal structure, with tensor product given by binary products~\cite[Theorem~2.15]{CarboniA:carbii}. Thankfully, known strictification theorems help us to reduce the complexity of the notions that we need to use, while maintaining sufficient generality to cover the intended examples.

Issues of coherence had to be faced also when introducing our key notions, that of a linear exponential pseudocomonad and that of a codereliction transformation. For the former, we are able to arrive at a definition for which coherence axioms are completely determined by the fundamental notions in the theory of monoidal bicategories by taking as our starting point the definition formulated  in the one-dimensional setting in~\cite{HylandM:gluoml}.
It has been pleasing to observe how the coherence conditions for these notions are exactly what is required to prove the desired facts, which we hope provides corroboration for the robustness of the coherence conditions in~\cite{DayStreet}. Given the troubled evolution of coherence in monoidal bicategories (see~\cite[Appendix~A]{CarmodyS:cobc} and~\cite[Section~2.1]{Schommer-PriesC:clatdc} for some details), we believe this experience will be helpful for the future development of the theory.

The definition of a codereliction presented additional challenges, in that it is not known whether its axioms can be expressed purely in terms of the basic concepts of the theory of monoidal categories. For this reason, if  we were to require the presence of invertible 2-cells in the diagrams that are part of the definition of a codereliction transformation in~\cite{BluteR:difcr}, we would then be facing the question of what coherence axioms to impose on them, which appears to be a difficult question.  To resolve this problem, we develop a careful analysis  by first showing that the diagrams
expressing the axioms for a codereliction transformation in~\cite{FioreM:difsmi} (under the name of a creation map) are  filled by canonical 2-cells. The definition of a codereliction transformation can then be simply stated as requiring these 2-cells to be invertible, thus sidestepping any issue of coherence. We hope that our results provide guidance for the formulation of coherence conditions in the future. It should also be pointed out that, since our theory uses non-invertible 2-cells in a crucial way, it is not immediately subsumed by the theory of $(\infty, 1)$-categories, although it may be possible to develop an $(\infty, 2)$-categorical counterpart of it.

\subsubsection*{Outline of the paper} \cref{sec:prelim} reviews the basic definitions and theorems of the theory of monoidal bicategories that will be needed in the paper. In \cref{sec:mon-comon-bialg} we study symmetric pseudocomonoids. We introduce and study linear exponential pseudocomonads in \cref{sec:linear-exponential}.
We explore our definitions further  in \cref{sec:products}, under the assumption  that the ambient bicategory has finite products, and in \cref{sec:biproducts}, with the hypothesis that it has finite biproducts. We introduce our bicategorical notion of codereliction in~\cref{sec:codereliction}. We conclude the paper in \cref{sec:prof} by showing how the bicategory of analytic functors can be equipped with a codereliction operator, thereby modelling differential linear logic.

\subsubsection*{References and conventions.}  In order to keep the paper at a reasonable length and avoid duplication of material that is now standard, we refer to~\cite{BaezJ:higdab,DayStreet,GurskiN:looscm} for the coherence conditions of the notions that we use.  Many of our proofs will construct the relevant 2-cells and indicate how to prove the required coherence conditions in text. This is similar to how standard diagram-chasing arguments are outlined in one-dimensional category theory. Readers who wish to fill in the details of the proofs are advised to keep the references above at hand.  For differential linear logic and differential categories, our main references are~\cite{BluteR:difcr,EhrhardT:intdll,FioreM:difsmi}. The notation for differential linear logic and differential categories used here follows closely that in~\cite{EhrhardT:intdll}, as summarised~in \cref{tab:structural-maps}.

\begin{table}[htb]
\fbox{\parbox{\textwidth}{
\begin{align*}
\text{Weakening} & \quad \weak_A \co \bang A \to \unit  & 
\text{Coweakening} &  \quad  \coweak_A \co \unit \to  \bang A \\
\text{Contraction} & \quad  \con_A \co  \bang A  \to  \bang A \otimes \bang A   &  
\text{Cocontraction} & \quad  \cocon_A \co \bang A \otimes \bang A \to \bang A \\ 
\text{Dereliction} & \quad  \der_A \co \bang A \to A   & 
\text{Codereliction} & \quad  \coder_A \co  A \to \bang A \\
\text{Promotion} & \quad \dig_A \co \bang A  \to \bbang A & 
\end{align*}}}
\caption{The structural maps.}
\label{tab:structural-maps}
\end{table}

\section{Preliminaries} 
\label{sec:prelim}

We assume that readers are familiar with the key notions of two-dimensional category theory~\cite{JohnsonN:twodc,LackS:a2cc}. Bicategories will be denoted by letters $\catK, \catL, \ldots$. When working with a bicategory $\catK$, we use upper-case letters~$A, B, C, \ldots$ for objects, lower-case letters $f \co A \to B$, $g \co B \to C, \ldots$ for maps, and lower-case Greek letters $\alpha \co f \Rightarrow f'$ for 2-cells. Composition is written simply by juxtaposition and the identity map on~$A \in \catK$ is written $\id_A \co A \to A$.  If $f \co A \to B$ is an adjoint equivalence in $\catK$, we write 
$f^\bullet \co B \to A$ for its adjoint. 

When we say that a bicategory $\catK$ has finite products (or finite coproducts), this is intended in the bicategorical sense~\cite{StreetR:fibb}, even when $\catK$ is a 2-category. For binary products, this means that for every $A, B \in \catK$, we have an object $A \with B \in \catK$ and projections  $\pi_1 \co A \with B \to A$ and $\pi_2 \co A \with B \to B$ which are universal, \ie for every $X \in \catK$, composition with $\pi_1$ and $\pi_2$ induces an adjoint equivalence
\[
\begin{tikzcd}
\catK[X, A \with B] \ar[r, "(  \pi_1(-) {,} \pi_2 (-))"]  &[8ex] \catK[X,A] \times \catK[A, B] .
\end{tikzcd}
\]
The pairing of $f \co X \to A$ and $g \co X \to B$ is written $(f, g) \co X \to A \with B$ as usual. The diagonal map of an object $A \in \catK$ is written $\Delta_A \co A \to A \with A$.
A terminal object is an object $\top \in \catK$ such that the unique functor
\[
\catK(X, \top) \to \mathsf{1}
\]
is an equivalence for every $X \in \catK$, where $\mathsf{1}$ is the category with a single object 
and only an identity map on it. If $\catK$ has finite coproducts, we write~$A + B$ for the coproduct of~$A, B \in \catK$, and~$\iota_1 \co A \to A + B $ and~$\iota_2 \co B \to A + B$ for the coprojections. The initial object is denoted $0$. The codiagonal map of an object $A \in \catK$ is written~$\nabla_A \co A + A \to A$.

If $\catK$ and $\catL$ are bicategories with finite products and $F \co \catK \to \catL$ is a pseudofunctor,
we say that $F$ preserves finite products if  the canonical maps 
$F(A \with B) \to FA \with FB$, for $A, B \in \catK$,
and $F(\term) \to \term$ are equivalences. We adopt this definition also when $\catK$ and $\catL$
are 2-categories and $F$ is a 2-functor. Preservation of finite coproducts is formulated dually.

\subsection*{Monoidal bicategories and Gray monoids} Recall from~\cite[Theorem~8.4.1]{JohnsonN:twodc} the strictification theorem for bicategories, asserting that every bicategory is biequivalent to a 2-category. 
This result gives as a scholium the strictification theorem for monoidal categories, asserting that every monoidal category is equivalent to a strict one,
since a monoidal category is a bicategory with a single object and a strict monoidal category is a 2-category with a single object.

A similar pattern arises one dimension up, with tricategories and Gray-categories, defined as in~\cite{GordonR:coht,GurskiN:cohtdc}, replacing bicategories and 2-categories, respectively. 
The strictification theorem for tricategories asserts that every tricategory is triequivalent to a Gray-category, \cf \cite[Theorem~8.1]{GordonR:coht} and \cite[Corollary~9.15]{GurskiN:cohtdc}. 
This gives a strictification result for monoidal bicategories since a monoidal bicategory is a tricategory with a single object and a Gray monoid is a Gray-category with a single 
object~\cite[Chapter~12]{JohnsonN:twodc}. Below, we recall some of this material. 
The notion of a Gray monoid is recalled in \cref{def:gray-monoid} using the notion of a cubical pseudofunctor in~\cite[Section~4.1]{GordonR:coht} or \cite[Definition~12.2.12]{JohnsonN:twodc}, and then partially unfolded.

\begin{definition} \label{def:gray-monoid}
A \myemph{Gray monoid} is a 2-category $\catK$ equipped with
\begin{itemize}
\item a cubical pseudofunctor $(\arghole) \otimes (=) \co \catK \times \catK \to \catK$, called the \emph{tensor product},
\item an object $\unit \in \catK$, called the \emph{unit}, 
\end{itemize}
which satisfy the associativity and unit conditions strictly.
\end{definition}

Let $\catK$ be a Gray monoid.\footnote{We refer to a Gray monoid by its underlying 2-category if this does not cause confusion.
An analogous convention is adopted for other kinds of structures throughout the paper.}
Recall that part of the pseudofunctoriality for the tensor product
involves, a natural isomorphism whose components are invertible 2-cells
\[
\begin{tikzcd}[column sep = large] 
A \otimes B \ar[d,"\id_A \otimes g"']\ar[dr,phantom,"\Two\phi_{f,g}"]\ar[r,"f \otimes \id_B"] & A' \otimes B \ar[d,"\id_{A'} \otimes g"] \\
A \otimes B' \ar[r,"f \otimes \id_{B'}"'] & A' \otimes B' ,
\end{tikzcd}
\]
where $f \co A \to A'$ and $g \co B \to B'$, subject to appropriate compatibility conditions~\cite[page~101]{DayStreet}. As in~\cite{DayStreet}, we take $f \otimes g \co A \otimes B \to A' \otimes B'$ to be the composite
\[
\begin{tikzcd} 
A \otimes B \ar[r, "\id_{A} \otimes g" ] &
A \otimes B' \ar[r, "f \otimes \id_{B'}"] &
A' \otimes B' .
\end{tikzcd} 
\]
We adopt an analogous convention to define $\alpha \otimes \beta \co f \otimes g \to f' \otimes g'$ for $\alpha \co f \Rightarrow f'$ and $\beta \co g \Rightarrow g'$. The other choice is canonically isomorphic and we shall often suppress mention of  these isomorphisms for brevity, since they are essentially unique~\cite{DiVittorioN:gracpt,DiVittorioN:somrig}.
Note that the associativity and unit axioms hold strictly and so in particular $A \otimes (B \otimes C) = (A \otimes B) \otimes C$ for every $A, B, C \in \catK$, and~$A \otimes  \unit  = A =  \unit \otimes A$, for all $A \in \catK$.

In preparation for the material on linear exponential pseudocomonads in \cref{sec:linear-exponential}, we recall the notions of a lax monoidal pseudofunctor, monoidal pseudonatural transformation, and monoidal modification. Here, a lax monoidal pseudofunctor is what is called a weak monoidal homomorphism in~\cite[Definition~2]{DayStreet}.

\begin{definition} \label{thm:lax-monoidal-pseudofunctor}
Let $\catK$ and $\catL$ be Gray monoids. A \myemph{lax monoidal pseudofunctor} from $\catK$ to~$\catL$ is a pseudofunctor $F \co \catK \to \catL$ equipped with
\begin{itemize}
\item a pseudonatural transformation with components on objects 
\[
\monn_{A, B} \co FA \otimes FB \to F(A \otimes B),
\] 
for $A, B \in \catK$, 
\item a map $\moni \co  \unit \to F( \unit)$, 
\item an invertible modification with components 
\[
\begin{tikzcd}[column sep=huge]
FA \otimes FB \otimes FC \ar[d,"\monn_{A,B} \otimes \id_{FC}"']\ar[dr,phantom,"\Two\omega_{A,B,C}"]\ar[r,"\id_{FA} \otimes \monn_{B,C}"] & FA \otimes F(B \otimes C) \ar[d,"\monn_{A, B \otimes C}"] \\
F(A  \otimes B) \otimes FC  \ar[r,"\monn_{A \otimes B, C}"'] & F(A \otimes B \otimes C),
\end{tikzcd}
\]
for $A, B, C \in \catK$, which we call the \emph{associativity constraint} of $F$,
\item two invertible modifications with components 
\[
 \begin{tikzcd}[column sep = tiny]
  & 
F  \unit \otimes FA
	\ar[d, phantom, description, "\Two \zeta_A"]
	\ar[dr, "\monn_{\unit, A}"]   & 
   \\
\unit \otimes FA
 	\ar[ur, "\moni \otimes \id_{FA}"] 
 	\ar[rr, "\id_{FA}"']  &
 \phantom{} & 
FA , 
 \end{tikzcd} \qquad
\begin{tikzcd}[column sep = tiny]
  & 
FA \otimes F  \unit
	\ar[d, phantom, description, "\Two \kappa_A"]
	\ar[dr, "\monn_{A,  \unit}"]   & 
   \\
FA \otimes \unit 
 	\ar[ur, "\id_{FA} \otimes \moni"] 
 	\ar[rr, "\id_{FA}"']  &
 \phantom{} & 
FA ,
 \end{tikzcd} 
\]
 for $A, B \in \catK$, which we call the left and right \emph{unitality constraints} of $F$, 
\end{itemize}
which satisfy the two coherence conditions in~\cite[Definition 2]{DayStreet}.

We say that $F$ is \myemph{strong monoidal} if the components of~$\monn$ and  $\moni$ are equivalences in~$\catL$.
\end{definition}

\begin{definition} \label{thm:monoidal-pseudonatural-transformation}
Let $F, G \co \catK \to \catL$ be lax monoidal pseudofunctors. A \myemph{monoidal pseudonatural transformation} from $F$ to $G$ is a 
pseudonatural transformation $p \co F \Rightarrow G$  equipped with
\begin{itemize}
\item an invertible modification with components
\[
\begin{tikzcd}[column sep = large]
FA \otimes FB \ar[d, "p_A \otimes p_B"'] \ar[r, "\monn_{A,B}"] \ar[dr,phantom,"\Two p^2_{A,B}"] & F(A \otimes B) \ar[d, "p_{A\otimes B}"]  \\
GA \otimes GB \ar[r, "\monn_{A, B}"'] & G(A \otimes B)  ,
\end{tikzcd}
\]
for $A, B \in \catK$, 
\item an invertible 2-cell
\[
 \begin{tikzcd}
  & 
F \unit
	\ar[d, phantom, description, "\Two p^0"]
	\ar[dr, "p_\unit"]   & 
   \\
 \unit
 	\ar[ur, "\moni"] 
 	\ar[rr, "\moni"']  &
 \phantom{} & 
G \unit ,
 \end{tikzcd}
\]
\end{itemize} 
which satisfy the three coherence conditions  in~\cite[Definition 3]{DayStreet}.
\end{definition}

\begin{definition} \label{thm:monoidal-modification}
Let $p, q \co F \Rightarrow F'$ be monoidal pseudonatural transformations.
A \myemph{monoidal modification} from $p$ to $q$ is a modification $\phi \co p \Rrightarrow q$ which satisfies the two coherence conditions in~\cite[Definition 3]{DayStreet}.
\end{definition} 

Note that a monoidal modification does not require additional structure on top of that of a modification.

The strictification theorem for monoidal bicategories, stated below for emphasis, follows as a scholium of the strictification theorem for tricategories~\cite{GordonR:coht,GurskiN:cohtdc}.

\begin{theorem}[Gordon, Power and Street; Gurski] \label{thm:strictification-monoidal}
Every monoidal bicategory is biequivalent, as a monoidal bicategory, to a Gray monoid.
\end{theorem}

\subsection*{Symmetric monoidal bicategories and symmetric Gray monoids} We shall now recall the definition of a symmetric Gray monoid and of the relevant counterparts of lax monoidal pseudofunctors, monoidal transformations, and monoidal modifications. For this, it is convenient to 
arrive at a symmetric Gray monoid by progressively introducing layers of structure and properties, 
via the notions of a braided and sylleptic Gray monoid, as illustrated in \cref{tab:mon-bicat}.

\begin{table}[htb]
\begin{tabular}{|c|c|c|c|}
\hline
  0-cells   &  1-cells & 2-cells &   3-cells   \\  \hline
 \makecell{Gray monoid} &  \makecell{Lax monoidal  \\ pseudofunctors}  &  \makecell{Monoidal \\ transformations}   & \makecell{Monoidal \\ modifications} 
  \\  \hline
 \makecell{Braided \\ Gray monoids}   & \makecell{Braided lax monoidal \\ pseudofunctors} &   \makecell{Braided monoidal \\
transformations}   & ''   \\  \hline
 \makecell{Sylleptic \\ Gray monoids} & \makecell{Sylleptic lax monoidal \\
pseudofunctors}  & ''  & ''  \\  \hline
 \makecell{Symmetric \\ Gray monoids}   & ''  & '' &  '' \\ \hline
\end{tabular}
\medskip
\caption{Overview of Gray monoids.}
\label{tab:mon-bicat}
\end{table}

The definition of a braided Gray monoid in \cref{def:symmetric-gray-monoid} is equivalent to the one in~\cite[Definition~12]{DayStreet}, but we state it in terms of the data considered in~\cite{BaezJ:higdab,GurskiN:looscm,GurskiN:inlsc}, \cf \cite[Remark, page~118]{DayStreet}. This is a direct categorification of the notion of a braided monoidal category~\cite[Definition~2.1]{JoyalA:bratc}, in that the two axioms for a braiding in the one-dimensional setting are replaced by two invertible modifications, written $\beta^1$ and $\beta^2$ below.

\begin{definition} \label{def:symmetric-gray-monoid} \leavevmode
\begin{itemize}
\item A \emph{braided Gray monoid} is a Gray monoid $\catK$ equipped with a pseudonatural equivalence with components on objects 
\[
\bra_{A, B} \co A \otimes B \to B \otimes A ,
\]
for $A, B \in \catK$, called the \emph{braiding}, and invertible modifications with components
\begin{gather*} 
\begin{tikzcd}[ampersand replacement=\&, column sep = small]
  \& 
B \otimes A \otimes C 
	\ar[d, phantom, description, "\Two \beta^1_{A,B,C}"]
	\ar[dr, "\id_B \otimes \bra_{A,C}"]   \& 
   \\
 A \otimes B \otimes C  
 	\ar[ur, "\bra_{A, B} \otimes \id_C"] 
 	\ar[rr, "\bra_{A, B \otimes C}"']  \&
 \phantom{} \& 
B \otimes C \otimes A ,
 \end{tikzcd} \\
\begin{tikzcd}[ampersand replacement=\&, column sep = small]
  \& 
A \otimes C \otimes B
	\ar[d, phantom, description, "\Two \beta^2_{A, B, C}"]
	\ar[dr, "\bra_{A, C} \otimes \id_B"]   \& 
   \\
 A \otimes B \otimes C  
 	\ar[ur, "\id_A \otimes \bra_{B, C}"] 
 	\ar[rr, "\bra_{A \otimes B, C}"']  \&
 \phantom{} \& 
C \otimes A \otimes B  ,
 \end{tikzcd}
 \end{gather*}
 for $A, B \in \catK$, called the \emph{braiding constraints}, such that 
 $\bra_{A, \unit} = \bra_{\unit, A} = \id_A$
for all $A \in \catK$, and which satisfy the coherence conditions in~\cite[Definition, pages 4234-4235]{GurskiN:looscm}, which include the four coherence 
 conditions in~\cite[Definition~6]{BaezJ:higdab}.
\item A \myemph{sylleptic Gray monoid}  is a braided Gray monoid $\catK$  equipped with an invertible modification with components
\[
\begin{tikzcd}
  & 
B \otimes A
	\ar[d, phantom, description, "\Two \sigma_{A, B}"]
	\ar[dr, "\bra_{B, A}"]   & 
   \\
 A \otimes B
 	\ar[rr, "\id_{A \otimes B}"'] 
 	\ar[ur, "\bra_{A, B}"]  &
 \phantom{} & 
A \otimes B  ,
 \end{tikzcd}
\]
for $A, B \in \catK$, called the \emph{syllepsis}, which satisfies the two coherence conditions in~\cite[Definition~15]{DayStreet}.
 \item A \myemph{symmetric Gray monoid} is a sylleptic Gray monoid such that the syllepsis $\sigma$ satisfies the additional condition for it to be a symmetry~\cite[Definition~18]{DayStreet}.
\end{itemize} 
\end{definition}

Next, we consider the counterparts of lax monoidal pseudofunctors for braided and sylleptic Gray monoids.

\begin{definition} \label{thm:sym-mon-functor} \leavevmode
\begin{itemize}
\item Let $\catK$ and $\catL$ be braided Gray monoids. A  \myemph{braided lax monoidal pseudofunctor} from $\catK$ to $\catL$ is a lax monoidal pseudofunctor
$F  \co \catK \to \catL$ equipped with 
an invertible modification with components
\[
\begin{tikzcd}[column sep = large]
FA \otimes FB \ar[r, "\bra_{FA,FB}"] \ar[d, "\monn_{A, B}"'] \ar[dr,phantom,"\Two"]   & FB \otimes FA \ar[d, "\monn_{B,A}"] \\
F(A \otimes B) \ar[r, "F(\bra_{A, B})"'] & F(B \otimes A) ,
\end{tikzcd}
\]
 for $A, B \in \catK$, which satisfies the two axioms in~\cite[Definition~14]{DayStreet}.
\item Let $\catK$ and $\catL$ be sylleptic Gray monoids. A \myemph{sylleptic lax monoidal pseudofunctor} from $\catK$ to $\catL$ is a braided lax monoidal pseudofunctor 
 $\catK \to \catL$ which satisfies the additional axiom in~\cite[Definition~16]{DayStreet} or \cite[Definition~1.5]{GurskiN:inlsc}.
\end{itemize}
\end{definition}

Note that the appropriate morphisms between symmetric  Gray monoids are sylleptic lax monoidal pseudofunctors, since a symmetric Gray monoid is merely a sylleptic Gray monoid which satisfies an additional property,  \cf \cite[Definition~1.5]{GurskiN:inlsc}.

\begin{definition} \label{thm:bra-mon-transformation}
Let $F, G \co \catK \to \catL$ be braided lax monoidal pseudofunctors. A \myemph{braided monoidal pseudonatural transformation} from $F$ to $G$ 
is a monoidal pseudonatural transformation $F \Rightarrow G$ which satisfies the additional coherence condition in ~\cite[pages 125-126]{DayStreet}.
\end{definition}

Analogously to what we said above for pseudofunctors, the appropriate morphisms between
sylleptic lax monoidal pseudofunctors are braided monoidal pseudonatural transformations,
since a sylleptic lax monoidal pseudofunctor is merely a  braided lax monoidal
pseudofunctor satisfying an additional property. Moreover, since these are just monoidal transformations which satisfy an additional
property, the appropriate morphisms between them are the monoidal modifications of \cref{thm:monoidal-modification}.

The next strictification result is recalled from~\cite[Theorem 1.29]{GurskiN:inlsc}.

\begin{theorem}[Gurski and Osorno] 
\label{thm:strictification-symmetric-monoidal} 
Every symmetric monoidal bicategory is equivalent, as a symmetric monoidal bicategory, to a symmetric  Gray monoid. 
\end{theorem}

Without loss of generality it is possible to assume that the additional equations for a strict symmetric monoidal 2-category, in the sense of~\cite[Definition~1.3]{GurskiN:inlsc}, hold as well.
These include those for a braided monoidal 2-category in the sense of~\cite[Definition~2.2]{CransS:gencbs}. In the following, we will use \cref{thm:strictification-monoidal} and \cref{thm:strictification-symmetric-monoidal} and  the associated principles of transport of structure to limit ourselves to consider (symmetric) Gray monoids. Provided
that the structure and properties under consideration are invariant under (symmetric) monoidal biequivalence, one obtains results on (symmetric) monoidal bicategories (\cf remarks at the
start of \cref{sec:prof}).

\begin{convention} \label{thm:coh-numb-convention}
In the following, we shall frequently refer to coherence conditions. We will do so using
the order in which they appear in the reference provided in the definition given here. For example, 
the first coherence axiom for a lax monoidal pseudofunctor (\cf \cref{thm:sym-mon-functor}) means the first
coherence axiom in~\cite[Definition~2]{DayStreet}, which involves an
equality between a pasting diagram
involving $\omega$ and $\zeta$ and a pasting diagram involving $\kappa$.
\end{convention}


\section{Pseudocomonoids and pseudobialgebras}
\label{sec:mon-comon-bialg}

\subsection*{Pseudocomonoids} 
In preparation for the material on linear exponential pseudocomonads in \cref{sec:linear-exponential}, we recall some basic material on pseudocomonoids from~\cite{DayStreet} and prove some auxiliary results about them, which we could not find in the existing literature and may be of independent interest.
The definitions of a  pseudocomonoid, braided pseudocomonoid and symmetric pseudocomonoid are analogous to the classical ones of monoidal category, braided monoidal category and symmetric monoidal category, respectively~\cite{JoyalA:bratc}.
Accordingly, the definitions of a pseudocomonoid morphism and of a braided pseudocomonoid morphism are analogous to those of strong monoidal functor and braided monoidal functor. Finally, pseudocomonoid 2-cells are
counterparts of monoidal transformations. 
In particular, an analogous pattern of distinction between structure and property arises, \cf \cref{tab:mon-cat}.

Let $\catK$ be a Gray monoid.

\begin{definition}  \label{thm:comonoid}
A \myemph{pseudocomonoid} in $\catK$ is an object $A \in \catK$ equipped with
\begin{itemize}
\item a map $n \co A \to A \otimes A$, called the \emph{comultiplication},
\item a map $e \co A \to  \unit$, called the \emph{counit},
\item invertible 2-cells
\begin{equation*}
\begin{tikzcd}[column sep = large]
A \ar[r, "n"] \ar[d, "n"']  \ar[dr,phantom,"\Two \alpha"] & A \otimes A \ar[d, "n \otimes \id_A "] \\
A \otimes A \ar[r, "\id_A \otimes n"'] & A \otimes A \otimes A  ,
\end{tikzcd}
\end{equation*}
called the \emph{associativity constraint}, and 
\begin{equation*}
\begin{tikzcd}[column sep = {1.5cm,between origins}]
  & 
A \otimes A
	\ar[d, phantom, description, "\Two \lambda"]
	\ar[dr, "e \otimes \id_A"]   & 
   \\
A
 	\ar[ur, "n"] 
 	\ar[rr, "\id_A"']  &
 \phantom{} & 
A ,
 \end{tikzcd} \qquad
\begin{tikzcd}[column sep = {1.5cm,between origins}]
  & 
A \otimes A
	\ar[d, phantom, description, "\Two \rho"]
	\ar[dr, "\id_A \otimes e"]   & 
   \\
A
 	\ar[ur, "n"] 
 	\ar[rr, "\id_A"']  &
 \phantom{} & 
A  ,
 \end{tikzcd}
 \end{equation*}
called the \emph{left and right unitality constraints}, respectively,
\end{itemize}
which satisfy the duals of the two coherence conditions in~\cite[Section~3]{DayStreet}. (These are analogous to those for a monoidal category.)
\end{definition}

In the following, we use the letters $n$ and $e$ to denote the comultiplication and counit of different pseudocomonoids when this does not cause confusion.

\begin{definition} Let $A$ and $B$ be pseudocomonoids in $\catK$. A  \myemph{pseudocomonoid morphism} from $A$ to $B$ is a
map $f \co A \to B$ equipped with
\begin{itemize}
\item an invertible 2-cell
\[
\begin{tikzcd}
A \ar[r, "f"] \ar[d, "\com"']  \ar[dr,phantom,"\Two \bar{f}"] & B \ar[d, "\com"]  \\
A \otimes A \ar[r, "f \otimes f"'] & B \otimes B ,
\end{tikzcd}
\]
\item an invertible 2-cell
\[
\begin{tikzcd} 
A \ar[d, "\cou"'] \ar[r, "f"]    \ar[dr,phantom,"\Two \tilde{f}"] & B \ar[d, "\cou"] \\
 \unit \ar[r, "\id_\unit"'] &  \unit ,
\end{tikzcd}
\]
\end{itemize}
which satisfy three coherence conditions. (These are analogous to those for a lax monoidal functor~\cite[page~25]{JoyalA:bratc}.)
\end{definition}

\begin{definition} \label{thm:pseudocomonoid-2-cell}
Let $f, g \co A \to B$ be  pseudocomonoid morphisms in $\catK$. A \emph{pseudocomonoid $2$-cell} from~$f$ to~$g$
is a 2-cell $f \Rightarrow g$ which satisfies two coherence conditions. (These are analogous to those for a monoidal natural
transformation~\cite[page~25]{JoyalA:bratc}.)
\end{definition}

We write $\CoMon{\catK}$ for the 2-category of pseudocomonoids, pseudocomonoid morphisms, and pseudocomonoid 2-cells in~$\catK$. 
Let us now assume that  $\catK$ is braided. 

\begin{definition} \label{thm:braided-comonoid}
A \myemph{braided pseudocomonoid} is a pseudocomonoid $A$ equipped with an invertible 2-cell
\begin{equation}
\label{equ:braiding-comonoid}
\begin{tikzcd}[column sep = {1.5cm,between origins}]
  & 
A \otimes A
	\ar[d, phantom, description, "\Two \gamma"]
	\ar[dr, "\bra_{A,A}"]   & 
   \\
A
 	\ar[ur, "\com"] 
 	\ar[rr, "\com"']  &
 \phantom{} & 
A \otimes A
 \end{tikzcd}
 \end{equation}
which satisfies the duals of the two coherence conditions in~\cite[Definition~13]{DayStreet}. (These are analogous to those for a braided monoidal category~\cite[Definition~2.1]{JoyalA:bratc}.)
\end{definition}

\begin{definition} \label{thm:braided-pseudocomonoid-morphism}
Let $A$ and $B$ be braided pseudocomonoids in $\catK$. A \myemph{braided pseudocomonoid morphism} from $A$ to $B$ is a
pseudocomoid morphism  $A \to B$ which satisfies one additional coherence condition. (This is analogous to that for a lax monoidal functor to be braided~\cite[Definition~2.3]{JoyalA:bratc}.)
\end{definition} 

We write $\BraCoMon{\catK}$ for the 2-category of braided pseudocomonoids, braided pseudocomonoid morphisms, and pseudocomonoid 2-cells  in~$\catK$.
Let us now assume that $\catK$ is sylleptic. 

\begin{definition} \label{thm:symmetric-pseudocomonoid}
A \emph{symmetric pseudocomonoid} in $\catK$ is a braided pseudocomonoid in $\catK$ whose braiding satisfies
the additional coherence condition for a symmetry~\cite[Definition~13]{DayStreet}. (This is analogous to the axiom for a braided monoidal category to be symmetric.)
\end{definition} 

We write $\SymCoMon{\catK}$ for the 2-category of symmetric pseudocomonoids, braided pseudocomonoid morphisms, and pseudocomonoid 2-cells in~$\catK$. 
The definitions of the 2-categories $\CoMon{\catK}$, $\BraCoMon{\catK}$, and $\SymCoMon{\catK}$ is summarised in \cref{tab:mon-cat}.
There are evident duals of these notions, as considered in~\cite{DayStreet} giving rise to the 2-categories $\Mon{\catK}$, $\BraMon{\catK}$, and $\SymMon{\catK}$ of 
pseudomonoids, braided pseudomonoids, and symmetric pseudomonoids, respectively, in $\catK$.

\begin{table}[htb]
\begin{tabular}{|c|c|c|c|}
\hline
 & 0-cells   &1-cells   &  2-cells   \\ \hline
  $\CoMon{\catK}$  & \makecell{Pseudocomonoids} &   \makecell{Pseudocomonoid \\ morphisms } & \makecell{Pseudocomonoid \\ 2-cell}  \\ \hline  
$\BraCoMon{\catK}$  & \makecell{Braided \\ pseudocomonoids}  & \makecell{Braided \\ pseudocomonoid morphisms}  & ''     \\ \hline
$\SymCoMon{\catK}$    &  \makecell{Symmetric \\ pseudocomonoid} & '' & ''  \\ \hline
\end{tabular}
\medskip
\caption{Overview of pseudocomonoids.}
\label{tab:mon-cat} 
\end{table}

\cref{thm:ccmon-finprod} below recalls from~\cite[Theorem~5.2]{SchappiD:indacq} the two-dimensional counterpart of the well-known one-dimensional result asserting that the category of commutative comonoids in a symmetric monoidal category has finite products, \cf \cite{Fox}. Recall from \cref{sec:prelim} that we mean finite products in a bicategorical sense.

\begin{theorem}[Sch\"appi] \label{thm:ccmon-finprod}
Let $\catK$ be a symmetric Gray monoid. The 2-category $\SymCoMon{\catK}$ of symmetric pseudocomonoids in $\catK$ has finite products. 
 \end{theorem}

We outline the definition of the binary products and of terminal object of $\SymCoMon{\catK}$. We begin by recalling that if $A$ and $B$ are symmetric pseudocomonoids  in $\catK$, the tensor product of their 
underlying objects $A \otimes B$ in $\catK$ admits the structure of a symmetric pseudocomonoid. Indeed, the tensor product in a symmetric Gray monoid $\catK$ is a sylleptic strong monoidal pseudofunctor~\cite[page~131]{DayStreet},
and therefore it preserves symmetric pseudocomonoids~\cite[Proposition~16]{DayStreet}. Thus, it lifts as follows:
\begin{equation*}
\begin{tikzcd}[column sep = large] 
\SymCoMon{\catK} \times \SymCoMon{\catK} \ar[r, "(-) \otimes (=)"] \ar[d, "U \times U"'] &
\SymCoMon{\catK} \ar[d, "U"] \\
\catK \times \catK \ar[r, "(-) \otimes (=)"']  &
\catK . 
\end{tikzcd}
\end{equation*}
Explicitly, the comultiplication and counit are the following composites:
\begin{gather}
\label{equ:comultiplication-for-tensor}
\begin{tikzcd}[ampersand replacement=\&] 
A \otimes B
	\ar[r, "\com \otimes \com"] \&
A \otimes A \otimes B \otimes B
	\ar[r, "\id_A \otimes \bra_{A,B} \otimes \id_A"] \&[4em] 
A \otimes B \otimes A \otimes B  , 
\end{tikzcd} \\
\label{equ:counit-for-tensor}
\begin{tikzcd}[ampersand replacement=\&] 
A \otimes B 
	\ar[r, "e \otimes e"] \&
\unit \otimes \unit 
	\ar[r, "\id"] \&
\unit . 
\end{tikzcd}
\end{gather}
The associativity constraint is 
\[
\begin{small} 
\begin{tikzcd}
A \otimes B \ar[r, "\com \otimes \com"] \ar[d, "\com \otimes \com"']  \ar[dr,phantom,"\Two \alpha \otimes \alpha"]  &[4em] 
A \otimes A \otimes B \otimes B \ar[r, "\id \otimes \bra_{A,B} \otimes \id"] \ar[d, "\com \otimes \id \otimes \com \otimes \id"'] \ar[dr,phantom,"\cong"]  &[4em]
A \otimes B \otimes A \otimes B \ar[d, "\com \otimes \com \otimes \id \otimes \id"'] \\
A \otimes A \otimes B \otimes B \ar[r, "\id \otimes \com \otimes \id  \otimes \com"']  \ar[d, "\id \otimes \bra_{A,B} \otimes \id"']  \ar[dr,phantom,"\cong"] & 
A \otimes A \otimes A \otimes B \otimes B \otimes B  \ar[r, "\id \otimes \id  \otimes \bra_{A,B \otimes B} \otimes \id"]  \ar[d, "\id \otimes \bra_{A \otimes A,B} \otimes \id \otimes \id"]  \ar[dr,phantom,"\Two \id \otimes \xi \otimes \id"']  & 
A \otimes A \otimes B \otimes B \otimes A \otimes B \ar[d, "\id \otimes \bra_{A,B} \otimes \id \otimes \id \otimes \id"'] \\
A \otimes B \otimes A \otimes B \ar[r, "\id \otimes \id  \otimes \com \otimes \com"']  & 
A \otimes B \otimes A \otimes A \otimes B \otimes B  \ar[r, "\id  \otimes \id \otimes \id \otimes \bra_{A,B} \otimes \id"']   &
A \otimes B \otimes A \otimes B \otimes A \otimes B ,
\end{tikzcd}
\end{small}
\]
where $\xi$ is an invertible 2-cell which can easily be constructed using the two braiding constraints~$\beta_1$ and~$\beta_2$ of \cref{def:symmetric-gray-monoid}.
The left unitality constraint is 
\[
\begin{tikzcd}
A \otimes B 
	\ar[r, "\com \otimes \com"]  
	\ar[dr, bend right = 20, "\id"'] &  
A \otimes A \otimes B \otimes B 
	\ar[r, "\id \otimes \bra_{A, B} \otimes \id"] 
	\ar[d, "e \otimes \id \otimes e \otimes \id"] \ar[dr,phantom,"\cong"] &[5em]  
A \otimes B \otimes A \otimes B 
	\ar[d, "e \otimes e \otimes \id \otimes \id"] \\
\phantom{} \ar[ur, phantom, description, pos=(.6), "\Two \lambda \otimes \lambda"] 
 & \unit \otimes A \otimes \unit \otimes B \ar[r, "\id \otimes \bra_{A, \unit} \otimes \id"'] & \unit \otimes \unit \otimes A \otimes B .
 \end{tikzcd}
 \]
The right unitality constraint is
\[
\begin{tikzcd}
A \otimes B \ar[r, "\com \otimes \com"] \ar[dr,  bend right = 20,  "\id"']  & A \otimes A \otimes B \otimes B \ar[r, "\id \otimes \bra_{A,B} \otimes \id"] \ar[d, "\id \otimes e \otimes \id \otimes e"] \ar[dr,phantom,"\cong"]  &[5em] A \otimes B \otimes A \otimes B \ar[d, "\id \otimes \id \otimes e \otimes e"]  \\
\phantom{} \ar[ur, phantom, description, pos=(.6), "\Two \rho \otimes \rho"] 
 &  A \otimes \unit \otimes B \otimes \unit  \ar[r, "\id \otimes \bra_{\unit, B} \otimes \id"']  & \otimes A \otimes B \otimes \unit \otimes \unit .
 \end{tikzcd}
 \]
The symmetry is 
\[
\begin{tikzcd}
A \otimes B 
	\ar[r, "\com \otimes \com"] 
	\ar[dr,  bend right = 20,  "\com \otimes \com"'] &[3em] 
A \otimes A \otimes B \otimes B 
	\ar[r, "\id_A \otimes \bra_{A,B} \otimes \id_B"] 
	\ar[d,  "\bra_{A,A} \otimes \bra_{B,B}"] \ar[dr,phantom,"\Two \sigma'"]  &[7em]
A \otimes B \otimes A \otimes B 
	\ar[d, "\bra_{A \otimes B, A\otimes B}"]  \\
 	\phantom{} \ar[ur, phantom, description, pos=(.6), "\Two \gamma \otimes \gamma"]  &
 	A \otimes A \otimes B \otimes B 
	\ar[r, "\id_A \otimes \bra_{A,B} \otimes \id_B"'] &
A \otimes B \otimes A \otimes B ,
\end{tikzcd}
\]
where $\sigma'$ is an invertible 2-cell constructed using the braiding constraints $\beta^1$, $\beta^2$ and the symmetry $\sigma$ of $\catK$, \cf \cref{def:symmetric-gray-monoid}.

It is then possible to show that $A \otimes B$ is the product of $A$ and $B$ in $\SymCoMon{\catK}$, when considered equipped with the projections $\pi_1 \co A \otimes B \to A$ and $\pi_2 \co A \otimes B \to B$  given by
the following composites:
\[
\begin{tikzcd}
A \otimes B \ar[r, "\id_A \otimes e"] &
 A \otimes \unit \ar[r, equal] &
 A , 
 \end{tikzcd} \quad
\begin{tikzcd}
A \otimes B \ar[r, "e \otimes \id_B"] &
 \unit \otimes B \ar[r, equal] &
B , 
 \end{tikzcd}
 \]
 respectively, both of which can be shown to be braided pseudocomonoid morphisms.  

The terminal object of $\SymCoMon{\catK}$ is the unit $\unit$ of $\catK$, viewed as a symmetric pseudocomonoid in the evident way. For a symmetric pseudomonoid $A$, the required essentially unique braided pseudocomonoid morphism to $\unit$ is the counit of $A$.

\begin{corollary} \label{thm:sym-comon-is-monoidal}
Let $\catK$ be a symmetric Gray monoid. The 2-category $\SymCoMon{\catK}$ of symmetric pseudocomonoids in $\catK$ admits a symmetric monoidal structure such that the forgetful 2-functor 
\[
U \co \SymCoMon{\catK} \to \catK
\] 
is  strict symmetric monoidal.
\end{corollary}

\begin{proof} Combine~\cref{thm:ccmon-finprod} with \cite[Theorem~2.15]{CarboniA:carbii}.
\end{proof}

Of course, \cref{thm:sym-comon-is-monoidal} could also be established directly (without appealing to \cref{thm:ccmon-finprod}), but we prefer not to do so for brevity. One of the reasons for our interest in \cref{thm:sym-comon-is-monoidal} is that it
allows us to introduce the notion of a symmetric pseudobialgebra in a concise way, 
as done in \cref{thm:bialgebra-def} below, just as it can be done for commutative bialgebras in the one-dimensional setting~\cite{AguiarM:monfsh}. 
The definition is restated in an equivalent way and unfolded explicitly in~\cref{thm:bialgebra-equivalent}.

\begin{definition} \label{thm:bialgebra-def}
A \myemph{symmetric pseudobialgebra} in $\catK$ is a symmetric pseudomonoid in $\SymCoMon{\catK}$.
\end{definition}

\begin{remark}  \label{thm:bialgebra-equivalent}
In analogy with the one-dimensional situation, a symmetric pseudobialgebra can
be defined equivalently as a symmetric pseudocomonoid in the 2-category $\SymMon{\catK}$ of symmetric pseudomonoids, braided pseudomonoid morphisms, and pseudomonoid 2-cells. Indeed, assume that we have
\begin{itemize}
\item an object $A \in \catK$,
\item maps $\com \co A \to A \otimes A$, $e \co A \to I$ in $\catK$, and invertible 2-cells $\alpha$, $\lambda$, $\rho$ and $\gamma$ equipping $A$ with the structure of a braided pseudocomonoid in $\catK$,
as in \cref{thm:comonoid} and \cref{thm:braided-comonoid},
\item maps $m \co A \otimes A \to A$ and $u \co I \to A$ in $\catK$ and invertible 2-cells 
\[
\begin{tikzcd}[column sep = large]
A \otimes A  \otimes A \ar[r, "\id_A \otimes m "] \ar[d, " m \otimes \id_A"']  \ar[dr,phantom,"\Two \beta"] & A \otimes A \ar[d, "m  "] \\
A \otimes A \ar[r, " m"'] & A   ,
\end{tikzcd}
\]
\[
\begin{tikzcd}[column sep = {1.5cm,between origins}]
  & 
A \otimes A
	\ar[d, phantom, description, "\Two \sigma"]
	\ar[dr, "m"]   & 
   \\
A 
 	\ar[ur, "u \otimes \id_A"] 
 	\ar[rr, "\id_A"']  &
 \phantom{} & 
A  ,
 \end{tikzcd} \qquad
\begin{tikzcd}[column sep = {1.5cm,between origins}]
  & 
A \otimes A 
	\ar[d, phantom, description, "\Two \tau"]
	\ar[dr, "m"]   & 
   \\
A 
 	\ar[ur, "\id_A \otimes u"] 
 	\ar[rr, "\id_A"']  &
 \phantom{} & 
A \otimes A ,
 \end{tikzcd}
 \]
\[
\begin{tikzcd}[column sep = {1.5cm,between origins}]
  & 
A \otimes A
	\ar[d, phantom, description, "\Two \delta"]
	\ar[dr, "m"]   & 
   \\
A \otimes A
 	\ar[ur, "\bra_{A,A}"] 
 	\ar[rr, "m"']  &
 \phantom{} & 
A  ,
 \end{tikzcd}
 \]
equipping $A$ with the structure of a braided pseudomonoid in $\catK$.
\end{itemize}
Then, there is a bijection between:
\begin{enumerate}
\item the set of 4-tuples of invertible 2-cells making $m \co A \otimes A \to A$ and $u \co I \to A$ into braided pseucodomonoid morphisms such that $\beta$, $\sigma$, $\tau$ and $\delta$ are pseudocomonoid 2-cells, 
which therefore determine a symmetric pseudomonoid in $\SymCoMon{\catK}$;
\item the set of 4-tuples of invertible 2-cells making $\com \co A \to A \otimes A$ and $e \co A \to I$ into  braided pseudomonoid morphisms such that $\alpha$, $\lambda$, $\rho$ and $\gamma$ are pseudomonoid 2-cells, 
which therefore determine a symmetric pseudocomonoid in $\SymMon{\catK}$.
\end{enumerate}
In this correspondence, the coherence conditions for the braided pseudomonoid morphisms and for pseudocomonoid 2-cells in~(i) imply the coherence conditions for the pseudocomonoid 2-cells and
for the braided pseudomonoid morphisms in (ii), respectively. 

Since \cref{thm:bialgebra-def} is stated in terms of the data in~(i), let us record for reference that the 2-cells therein have the form
 \[
\begin{tikzcd}[column sep = large]
A \otimes A 
	\ar[r, "m"]
	\ar[d, "\com \otimes \com"'] 
	\ar[ddr, phantom, description, "\Two \bar{m}"] &
A \ar[dd, "\com"]  \\
A \otimes A \otimes A \otimes A 
	\ar[d, "\id_A \otimes \bra_{A,A} \otimes \id_A"'] 
	 & \\
 A \otimes A \otimes A \otimes A \ar[r, "m \otimes m"'] & 
A \otimes A ,
\end{tikzcd} \qquad
 \begin{tikzcd}
 A \otimes A 
 	\ar[r, "m"] 
	\ar[dd, "e \otimes e"'] 
	\ar[ddr, phantom, description, "\Two \tilde{m}"] & 
A  \ar[dd, "\cou"]  \\
 & \\
 \unit \otimes \unit \ar[r, "\id"'] & 
\unit ,
 \end{tikzcd} 
 \]
 and
 \[
 \qquad
\begin{tikzcd}
\unit 
	\ar[r, "u"] 
	\ar[d, "\id"'] 
	\ar[dr, phantom, description, "\Two \tilde{u}"] &[4em]
A 
	\ar[d, "\com"] \\
 \unit \otimes \unit
	 \ar[r, "u \otimes u"']  &
 A \otimes A ,
 \end{tikzcd}  \qquad
 \begin{tikzcd}
 \unit 
 	\ar[d, "\id"'] 
 	\ar[r, "u"] 
	\ar[dr, phantom, description, "\Two \bar{u}"] & 
 A \ar[d, "\cou"] \\ 
\unit 
	\ar[r, "\id"']  & 
\unit 
.
 \end{tikzcd} \qquad 
\]
\end{remark}

\subsection*{Products, coproducts and biproducts} 

We  establish some facts regarding cartesian and cocartesian monoidal structures. Recall from \cref{sec:prelim} our conventions and notation regarding products and coproducts in a 2-category.

\begin{theorem} \label{thm:prod-iff-all-ccmon}
Let $\catK$ be a symmetric Gray monoid. The following conditions are equivalent.
\begin{enumerate} 
\item The monoidal structure of $\catK$ is cartesian, \ie the unit $\unit$ is a terminal object and, for all
$A, B \in \catK$, the object $A \otimes B$ is a product of $A$ and $B$, with projections
given by the composites
  \[
  \begin{tikzcd}[column sep = large]
  A \otimes B \ar[r, "\id_A \otimes e_B"] & A \otimes \unit \ar[r, equal] & A , \quad
  \end{tikzcd}
   \begin{tikzcd}[column sep = large]
  A \otimes B \ar[r, "e_A \otimes \id_B"] & \unit \otimes B \ar[r, equal] & B ,
  \end{tikzcd}
  \]
  where, for $A \in \catK$, we write $e_A \co A \to \unit$ for the essentially unique map
  into the terminal object.  
  \item The forgetful symmetric strict monoidal 2-functor $U \co \SymCoMon{\catK} \to \catK$ has a quasi-inverse $V \co \catK \to \SymCoMon{\catK}$ that is
  a  symmetric strong monoidal section of it. In particular, $U$ is a  symmetric strong monoidal biequivalence.
\item Every $A \in \catK$ is a symmetric pseudocomonoid in $\catK$, pseudonaturally and monoidally in $A$, by which we mean the following.
\begin{itemize}
\item Every $A \in \catK$ is a symmetric pseudocomonoid in the sense of \cref{thm:comonoid}, \ie we have maps
\begin{equation*}
\com_A \co A \to  A \otimes A , \qquad
e_A \co A \to \unit ,
\end{equation*}
and invertible 2-cells 
\begin{gather*}
\begin{tikzcd}[column sep = large, ampersand replacement=\&] 
A 
	\ar[r, "\com_A"] \ar[d, "\com_A"']  
	\ar[dr,phantom,"\Two \alpha_A"] \& 
A \otimes A 
	\ar[d, "\id_A \otimes \com_A"] \\
A \otimes A 
	\ar[r, "\com_A \otimes \id_A"'] \& 
A \otimes A \otimes A  
,
\end{tikzcd}  \\
\label{equ:uniform-unitality}
\begin{tikzcd}[column sep = {1.5cm,between origins}, ampersand replacement=\&] 
  \& 
A \otimes A
	\ar[d, phantom, description, "\Two \lambda_A"]
	\ar[dr, "e_A \otimes \id_A"]   \& 
   \\
A
 	\ar[ur, "\com_A"] 
 	\ar[rr, "\id_A"']  \&
 \phantom{} \& 
\unit \otimes A ,
 \end{tikzcd} \qquad
\begin{tikzcd}[column sep = {1.5cm,between origins}, ampersand replacement=\&] 
  \& 
A \otimes A
	\ar[d, phantom, description, "\Two \rho_A"]
	\ar[dr, "\id_A \otimes e_A"]   \& 
   \\
A
 	\ar[ur, "\com_A"] 
 	\ar[rr, "\id_A"']  \&
 \phantom{} \& 
A \otimes \unit ,
 \end{tikzcd} \\
\begin{tikzcd}[column sep = {1.5cm,between origins}, ampersand replacement=\&] 
  \& 
A \otimes A
	\ar[d, phantom, description, "\Two \gamma_A"]
	\ar[dr, "\bra_{A,A}"]   \& 
   \\
A
 	\ar[ur, "\com_A"] 
 	\ar[rr, "\com_A"']  \&
 \phantom{} \& 
A \otimes A ,
 \end{tikzcd}
\end{gather*}
satisfying appropriate coherence conditions.
\item The families of maps $(\com_A)_{A \in \catK}$ and $(e_A)_{A \in \catK}$ are pseudonatural, \ie for every map $f \co A \to  B$ in $\catK$, we have invertible  2-cells
\begin{equation*}
\begin{tikzcd} 
A \ar[r, "\com_A"] \ar[d, "f"'] \ar[dr, phantom, description, "\Two n_f"]  &  A \otimes  A \ar[d, "f \otimes f"] \\
 B \ar[r, "\com_B"'] &  B \otimes  B ,
 \end{tikzcd} \qquad \begin{tikzcd} 
 A \ar[r, "e_A"] \ar[d, "f"']  \ar[dr, phantom, description, "\Two e_f"]  & \unit \ar[d, "\id_\unit"] \\
 B \ar[r, "e_B"'] & \unit ,
 \end{tikzcd} 
 \]
 satisfying pseudonaturality axioms.
\item The families of 2-cells $(\alpha_A)_{A \in \catK}$, $(\lambda_A)_{A \in \catK}$, $(\rho_A)_{A \in \catK}$ and $(\gamma_A)_{A \in \catK}$ are modifications.
\item The pseudonatural transformations $\com$ and $e$ are braided monoidal in the sense of~\cref{thm:bra-mon-transformation}, \ie 
we have an invertible modification $\com^2$ and an invertible 2-cell $\com^0$ as follows:
\begin{equation*}
\begin{tikzcd}[column sep = large]
A \otimes B \ar[r, "\com_{A \otimes B}"] \ar[d, "\id_{A \otimes B}"'] \ar[dr, phantom, description, "\Two \com^2_{A,B}"]  & A \otimes B \otimes A \otimes B \ar[d, "\id_A \otimes \bra_{B,A} \otimes \id_B"] \\
A \otimes B \ar[r, "\com_A \otimes \com_B"'] & A \otimes A \otimes B \otimes B ,
 \end{tikzcd} 
  \qquad
\begin{tikzcd} 
\unit \ar[d, "\id_\unit"'] \ar[r, "\com_\unit"] \ar[dr, phantom, description, "\Two \com^0"] & \unit  \otimes \unit  \ar[d, "\id_\unit"] \\
\unit \ar[r, "\id_\unit"'] &  \unit  ,
 \end{tikzcd} 
\end{equation*}
and an invertible modification $e^2$ and an invertible 2-cell  $e^0$  as follows:
\begin{equation*}
\begin{tikzcd}[column sep = large]
A \otimes B \ar[r, "e_A \otimes e_B"] \ar[d, "\id_{A \otimes B}"'] \ar[dr,
phantom, description, "\Two e^2_{A,B}"]  & \unit \otimes \unit \ar[d, "\id"] \\
A \otimes B \ar[r, "e_{A \otimes B}"'] & \unit ,
\end{tikzcd} \qquad
\begin{tikzcd}
\unit \ar[r, "\id_\unit"] \ar[d, "\id_\unit"'] \ar[dr, phantom, description, "\Two e^0"] & \unit \ar[d, "\id_\unit"] \\
\unit \ar[r, "e_\unit"'] & \unit ,
\end{tikzcd}
\end{equation*}
satisfying the appropriate coherence conditions.
\item The modifications $\alpha$, $\lambda, \rho$ and $\gamma$  are monoidal, in the sense of \cref{thm:monoidal-modification}.
\end{itemize}
\end{enumerate}
\end{theorem} 

\begin{proof} For $\text{(i)} \Rightarrow \text{(ii)}$, assume that the monoidal structure on $\catK$ is cartesian. We exhibit the required section
$V \co \catK \to \SymCoMon{\catK}$ as follows. On objects, we send $A \in \catK$ to the symmetric pseudocomonoid with underlying object
$A$ determined by finite products, with the diagonal $\Delta_A \co A \to A \otimes A$ as comultiplication and the essentially unique map into the terminal object~$t_A \co A \to \term$ 
as counit. The rest of the structure and the coherence conditions can be constructed and verified using the universal property of binary products and terminal objects. 
This definition extends to maps and 2-cells in an evident way. In 
particular, every map $f \co A \to B$ can be equipped with the structure of a braided pseudocomonoid morphism $(f, \bar{f}_\textup{can}, \tilde{f}_\textup{can}) \co (A, \Delta_A, t_A) \to (B, \Delta_B, t_B)$, in the sense of  \cref{thm:braided-pseudocomonoid-morphism}.

We claim that $V$ is essentially surjective. For this, let $(A, \com, e)$ be  symmetric pseudocomonoid  as in \cref{thm:symmetric-pseudocomonoid}. We claim that the identity
$\id_A \co A \to A$ is a braided pseudocomonoid morphism between $(A, n, e)$ and $(A, \Delta_A, t_A)$. We thus need invertible 2-cells~$\com \cong \Delta_A$ and~$e \cong t_A$. The first is determined, via the universal property of binary products, by the 2-cells $ \pi_1 \com \cong \id_A$ and $\pi_2 \com \cong \id_A$
in
\[
\begin{tikzcd}
A \ar[r, "\com"] \ar[dr, bend right = 30, "\id_A"']  \ar[dr, phantom, description, "\quad \Two \rho"]  & A \otimes A \ar[d, "\id_A \otimes e"] \\
 & A \otimes \unit ,
    \end{tikzcd} \qquad
    \begin{tikzcd}
 A \ar[r, "\com"] \ar[dr, bend right = 30, "\id_A"']  \ar[dr, phantom, description, "\quad \Two \lambda"]  & A \otimes A \ar[d, "\id_A \otimes e"] \\
 & A \otimes \unit .
    \end{tikzcd}
  \]
The second is determined uniquely by the universal property of terminal objects. It remains to show that, for all $A, B \in \catK$, the functor
  \[
  V_{A,B} \co \catK[A,B] \to \mathsf{SymCMon} \big[ (A, \Delta_A, t_A), (B, \Delta_B, t_B) \big]
  \]
  is an equivalence. For essential surjectivity, let  $(f, \bar{f}, \tilde{f}) \co (A, \Delta_A, t_A) \to (B, \Delta_B, t_B)$ be a general braided pseudocomonoid morphism in the sense of
 \cref{thm:braided-pseudocomonoid-morphism} (but note that the symmetric pseudocomonoid structures involved here are those determined by the finite products). One then
 shows  that the identity
 \mbox{2-cell} $1_f \co f \Rightarrow f$ is a pseudocomonoid 2-cell between $(f, \bar{f}, \tilde{f})$ and $(f, \bar{f}_\textup{can}, \tilde{f}_\textup{can})$ in the sense of \cref{thm:pseudocomonoid-2-cell}. This can be done using the universal property of binary products. Finally,  fullness and faithfulness are straightforward. The required monoidality of $V \co \catK \to \SymCoMon{\catK}$ follows from 
 long calculations, which we omit.

The proof that $\text{(ii)} \Rightarrow \text{(iii)}$ essentially consists in unravelling the relevant definitions. The section $V$ not only provides the required symmetric pseudocomonoid structure
on every object, but also makes every map in $\catK$ into a braided pseudocomonoid morphism and every 2-cell into a pseudocomonoid 2-cell (with respect to the assigned structure). The pseudonaturality of $(\com_A)_{A \in \catK}$ and $(e_A)_{A \in \catK}$ and  the fact that $(\alpha_A)_{A \in \catK}$, $(\lambda_A)_{A \in \catK}$, $(\rho_A)_{A \in \catK}$ and $(\gamma_A)_{A\in \catK}$ are modifications come
from the fact that every $f \co A \to B$ in~$\catK$ is a braided pseudocomonoid morphism. The monoidality of $(n_A)_{A\in \catK}$ and $(e_A)_{A \in \catK}$ comes from the 
monoidality of $V$. Finally, the monoidality of $(\alpha_A)_{A \in \catK}$, $(\lambda_A)_{A \in \catK}$, $(\rho_A)_{A \in \catK}$, and $(\gamma_A)_{A\in \catK}$  comes from
the fact that every 2-cell is a pseudocomonoid 2-cell.

  For $\text{(iii)} \Rightarrow \text{(i)}$, assume that $A \in \catK$ admits a symmetric pseudocomonoid structure in $\catK$, with comultiplication $n_A \co A \to A \otimes A$ and counit $e_A \co A \to \unit$, 
  pseudonaturally and monoidally in $A$.  First, one shows that $\unit$ is terminal. We thus need to show that, for $A \in \catK$, the unique functor $\catK[A, \unit] \to \mathsf{1}$  is an equivalence. A quasi-inverse to it is defined by mapping
 the unique object of $\mathsf{1}$ to~$e_A \co A \to \unit$, the counit of $A$. To prove that this
 is an  equivalence we need to show that there are invertible 2-cells $\eta_f 
 \co f \Rightarrow e_A$, for $f \co A \to \unit$, natural in $f$, and that these 2-cells are unique.
 These 2-cells are given by the pasting diagrams
 \[
 \begin{tikzcd}
 A \ar[r, "e_A"] \ar[d, "f"']  \ar[dr, phantom, description, "\Two e_f"]  & \unit \ar[d, equal] \\
 \unit \ar[r, "e_\unit"] \ar[d, equal] \ar[dr, phantom, description, "\Two e^0"]  & \unit \ar[d, equal] \\
 \unit \ar[r, "\id_\unit"'] & \unit .
 \end{tikzcd}
 \] 
 From the pseudo-naturality of the counits $e_A$ one readily deduces the naturality
of $\eta_f$ and also the fact that $\eta_{e_A}$ is the identity 2-cell which gives the
uniqueness.
  
      For binary products, we wish to show that $A \otimes B$ is the product of $A, B \in \catK$. First, we
  define projections $\pi_1 \co A \otimes B \to A$ and $\pi_2 \co A \otimes B \to B$ as the composites
  \[
  \begin{tikzcd}
  A \otimes B \ar[r, "\id_A \otimes e_B"] & A \otimes \unit \ar[r, equal] & B , \quad
  \end{tikzcd}
   \begin{tikzcd}
  A \otimes B \ar[r, "e_A \otimes \id_B"] & A \otimes \unit \ar[r, equal] & A .
  \end{tikzcd}
  \]
We wish to show that the functor
 \begin{equation}
 \label{equ:monoidal-cartesian}
\begin{tikzcd}
\catK[X, A \otimes B] \ar[r, "(  \pi_1(-) {,} \pi_2 (-))"]  &[8ex] \catK[X,A] \times \catK[A, B] 
\end{tikzcd}
\end{equation}
is an equivalence. We do this by constructing a quasi-inverse for it. For $f \co X \to A$ and $g \co X \to B$, we define $(f, g) \co X \to A \otimes B$ as the map
\[
  \begin{tikzcd}
  X  \ar[r, "\com_X"] & X \otimes X \ar[r, "f \otimes g"] & A \otimes B .
  \end{tikzcd}
\]
For 2-cells $\alpha \co f \Rightarrow f'$ and $\beta \co g \Rightarrow g'$, there is an evident
2-cell $(\alpha, \beta) \co (f, g) \Rightarrow (f', g')$ and the functoriality of this construction is easy to show.
Next, we need to exhibit a natural family of invertible 2-cells 
\[
\eta_h \co h \Rightarrow ( \pi_1 h, \pi_2 h) ,
\]
for $h \co X \to A \otimes B$. These are given by the pasting diagram:
\[
\begin{tikzcd}[column sep = huge] 
X \ar[r, "n_X"] \ar[d, "h"'] \ar[dr, phantom, description, "\cong"] & X \otimes X \ar[d, "h \otimes h"] &  \\
A \otimes B \ar[r, "n_{A \otimes B}"]  \ar[d, equal] \ar[dr, phantom, description, "\Two n^2_{A,B}"]
 & A \otimes B \otimes A \otimes B \ar[d, "\id \otimes \bra_{B,A} \otimes \id"] \ar[r, "\id \otimes e_B \otimes e_A \otimes \id"] 
 \ar[dr, phantom, description, "\cong"]
 & A \otimes \unit \otimes \unit \otimes B \ar[d, "\id \otimes \bra_{\unit, \unit} \otimes \id" ]  \\
 A \otimes B \ar[r,  "n_A \otimes n_B"] \ar[bend right = 20, dr, "\id_{A \otimes B}"']  
 \ar[dr, bend left = 5, phantom, description, "\Two  \rho_A \otimes \lambda_B"]
 & A \otimes A \otimes B \otimes B \ar[d, "\id \otimes e_A \otimes e_B \otimes \id"] \ar[r,  "\id \otimes e_A \otimes e_B \otimes \id"] & A \otimes \unit \otimes \unit \otimes B \ar[d, equal]  \\
 & A \otimes B \ar[r, equal] & A \otimes B  ,
 \end{tikzcd}
\]
using that $\bra_{\unit, \unit} = \id_\unit$. Finally, we need a natural family of invertible 2-cells 
\[
\varepsilon^1_{f,g} \co \pi_1(f, g) \Rightarrow f , \qquad
\varepsilon^2_{f,g}  \co \pi_2(f,g) \Rightarrow g ,
\]
for $f \co X \to A$ and $g \co X \to B$. These are given by the pasting diagram
\[
\begin{tikzcd}
X \ar[r, "\com_X"] \ar[dr, bend right = 20, "\id_X"'] 
 \ar[dr, bend left = 8, phantom, description, "\Two  \rho_X"]
& X \otimes  X \ar[r, "f \otimes g"]  \ar[d, "\id_X \otimes e_X"]  \ar[dr, phantom, description, "\quad \cong"] & A \otimes B \ar[d, "\id_A \otimes e_B"]  \\
 & X \ar[r, "f"']  & A .
 \end{tikzcd}
 \]
The naturality conditions and then the equivalence \eqref{equ:monoidal-cartesian} follow by routine calculations which we omit.
\end{proof}
  
 The force of \cref{thm:prod-iff-all-ccmon} is in the implication $\textup{(iii)} \Rightarrow \textup{(ii)}$, since $\textup{(ii)}$ expresses not only the existence, but also the 
 essential uniqueness of the symmetric pseudocomonoid structure on objects in $\catK$, which is not evident in $\textup{(iii)}$. We have a preliminary
 analysis of this fact in terms of a two-dimensional version of Fox's classical result~\cite{Fox}, which we hope to develop further and present elsewhere.

For the next lemma, recall what we mean by preservation of finite products from the start of \cref{sec:prelim}.

\begin{lemma} \label{thm:canonical-preserved}
Let $\catK$ and $\catL$ be 2-categories with finite products and $F \co \catK \to \catL$ be a 2-functor that preserves finite products.
Then, for every $A \in \catK$, the following
symmetric pseudocomonoids in $\catL$ are equivalent as symmetric pseudocomonoids:
\begin{enumerate}
\item the symmetric pseudomonoid on $FA$ obtained by applying $F$ to the canonical pseudocomonoid on $A$ determined by finite products in $\catK$, with comultiplication and counit 
\[
\begin{tikzcd} 
FA \ar[r, "F(\Delta_A)"] &[1em]  F(A \with A) \ar[r, "\simeq"] &  FA \with FA  ,
\end{tikzcd} \quad
\begin{tikzcd}
FA \ar[r, "F(\textup{can})"]  &[1em] F(\term) \ar[r, "\simeq"] & \term ,
\end{tikzcd}
\]
\item the symmetric pseudocomonoid on $FA$ determined by finite products in $\catL$, with comultiplication and counit 
\[
\begin{tikzcd} 
FA \ar[r, "\Delta_{FA}"] &[1em] FA \with FA  ,
\end{tikzcd} \qquad
\begin{tikzcd}
FA \ar[r, "\textup{can}"]  &[1em]  \term . 
\end{tikzcd}
\]
\end{enumerate}
\end{lemma}

\begin{proof} Direct calculation.
\end{proof}

We conclude this section with our definition of a biproduct and a basic fact about it, which will be useful for  \cref{sec:biproducts}. In order to state it, we need some auxiliary notation.
Let $\catK$ be a 2-category with finite products and finite coproducts and assume that the evident map $f^0 \co 0 \to \term$ is an equivalence. Then, for every $A, B \in \catK$, we define 
$0_{A,B} \co A \to B$ as the composite $A \to \term \to 0 \to B$, where $\term \to 0$ is the pseudo-inverse of $f^0 \co 0 \to \term$. Then, for $A, B \in \catK$, the universal property of coproducts determines a map 
\[
f^2_{A,B} \co A + B \to A \with B
\] 
such that $f^2_{A,B} \circ \iota_1 \cong (\id_A, 0_{A,B})$ and $f^2_{A,B} \circ \iota_2 \cong (0_{B,A}, \id_B)$.

\begin{definition} \label{thm:biproducts}
Let $\catK$ be a 2-category with finite products and finite coproducts. We say that $\catK$ has \myemph{biproducts} if 
\begin{enumerate}
\item the map $f^0 \co 0 \to \term$ is an equivalence;
\item the map $f^2_{A, B} \co A + B \to A \with B$ is an equivalence, for all $A, B \in \catK$.
\end{enumerate}
\end{definition}

If a 2-category $\catK$ has biproducts, we identify product and coproducts,
call them \emph{biproducts}, and write~$A \oplus B$ for the biproduct of $A$ and $B$. We also identify the initial and terminal object, write $0$ for it, and call it the \emph{zero object} of $\catK$. The next result builds on \cref{thm:prod-iff-all-ccmon}.

\begin{proposition} \label{thm:biprod-bialgebra}
Let $\catK$ be a 2-category with biproducts. For every object $A \in \catK$, the symmetric pseudomonoid structure on $A$ determined by  products,
\[
\Delta_A \co A \to A \oplus A , \quad A \to 0  , 
\]
and the symmetric pseudocomonoid structure on $A$, determined by coproducts,
\[
\nabla_A \co A \oplus A \to A , \quad 0 \to A ,
\]
determine a symmetric pseudobialgebra structure on $A$.
\end{proposition}

\begin{proof} The symmetric monoidal structure determined by biproducts is both cartesian and cocartesian and therefore every object $A \in \catK$ has both the structure of a  symmetric pseudocomonoid
and of a  symmetric pseudomonoid by~\cref{thm:prod-iff-all-ccmon} and its dual. It remains to check that the multiplication (\ie the codiagonal $\nabla_A \co A \oplus A \to A$) and the counit (\ie the essentially unique map $A \to 0$) are braided pseudocomonoid morphisms and that the associativity and unitality constraints are pseudocomonoid 2-cells, but this is true for any map and any 2-cell in a cocartesian monoidal structure.
\end{proof}

\section{Linear exponential pseudocomonads}
\label{sec:linear-exponential}

The goal of this section is to introduce our bicategorical counterpart of the notion of a linear exponential comonad and prove some basic facts about it. Our formulation is based on the
definition of linear exponential comonad in~\cite{HylandM:gluoml}. We compare our notion with the one in~\cite{JacqC:catcnis} in \cref{thm:compare-with-jacq} below.

\subsection*{Pseudocomonads}
As a first step, we recall the definition of a pseudocomonad, which is dual to that of a pseudomonad~\cite{LackS:cohap}. 
We consider pseudocomonads whose underlying pseudofunctor is a 2-functor, as justified by the 
strictification theorems in~\cite{CampbellA:howss,LackS:cohap}, \cf \cite[§1.3.1]{MirandaA:enrkop}.

 \begin{definition} \label{def:psd-comonad} Let $\catK$ be a 2-category.  A \myemph{pseudocomonad} on $\catK$ is a 2-functor $\bang(-) \co \catK \to \catK$ equipped with:
 \begin{itemize}
 \item a pseudonatural transformation with components on objects $\dig_A \co \bang A \to \bbang A$, for $A \in \catK$, called the \emph{comultiplication},
 \item a pseudonatural transformation with components on objects  $\der_A \co \bang A \to A$, for $A \in \catK$, called the \emph{counit},
 \item invertible modifications with components
\[
\begin{tikzcd}[column sep = large]
\bang A \ar[r, "\dig_A"] \ar[d, "\dig_A"']  \ar[dr,phantom,"\Two \pi_A"] & \bbang A \ar[d, "\bang \dig_A"] \\
\bbang A \ar[r, "\dig_{\bang A} "'] & \bbbang A  ,
\end{tikzcd} 
\]
for $A \in \catK$, called the \emph{associativity constraint}, 
\item two invertible modifcations
\[
\begin{tikzcd}[column sep = {1.5cm,between origins}]
  & 
\bbang A 
	\ar[d, phantom, description, "\Two \mu_A"]
	\ar[dr, "\der_{\bang A}"]   & 
   \\
\bang A 
 	\ar[ur, "\dig_A"] 
 	\ar[rr, "\id_{\bang A}"']  &
 \phantom{} & 
\bang A ,
 \end{tikzcd} \qquad
\begin{tikzcd}[column sep = {1.5cm,between origins}]
  & 
\bbang A 
	\ar[d, phantom, description, "\Two \nu_A"]
	\ar[dr, "\bang \der_A"]   & 
   \\
\bang A 
 	\ar[ur, "\dig_A"] 
 	\ar[rr, "\id_{\bang A}"']  &
 \phantom{} & 
\bang A ,
 \end{tikzcd}
 \end{equation*}
for $A \in \catK$, called the  (left and right) \emph{unitality constraints}, 
\end{itemize}
which satisfy two coherence conditions. (These are dual to those for a pseudomonad~\cite[Section~1]{LackS:cohap}.)
\end{definition}

The notation for a pseudocomonad adopted here\footnote{The slightly unorthodox choice of the Greek letters $\pi$, $\mu$, $\nu$ for the pseudocomonad
constraints is intended to avoid a clash with the notation for the constraints of a pseudomonoid, \cf \cref{thm:comonoid}. This is especially useful in view of the formulation of a linear exponential pseudocomonad
in \cref{thm:linear-exponential-comonad}, which will involve both notions.} is inspired by the literature on
categorical models of Linear Logic, where the underlying 2-functor of the pseudocomonad corresponds to the exponential modality, 
the comultiplication to the promotion rule, and the comultiplication to the dereliction rule~\cite{MelliesPA:catsll}.

Let us now fix a 2-category $\catK$ and a pseudocomonad $(\bang, \dig, \der)$ on it as in \cref{def:psd-comonad}. 

\begin{definition} A \emph{pseudocoalgebra} for the pseudocomonad is an object $A \in \catK$ equipped with:
\begin{itemize}
\item a map $a \co A \to \bang A$, called the \emph{structure map} of the pseudocoalgebra,
\item two invertible 2-cells
\[
\begin{tikzcd}[column sep = large]
A \ar[r, "a"] \ar[d, "a"'] \ar[dr,phantom,"\Two \bar{a}"] & \bang A \ar[d, "\dig_A"] \\
\bang A \ar[r, "\bang a"'] &  \bbang A ,
\end{tikzcd} \qquad
\begin{tikzcd}[column sep = large]
A \ar[r, "a"] \ar[dr, bend right = 20,  "\id_A"'] & \bang A \ar[d, "\der_A"] \\
\phantom{} \ar[ur, phantom, {pos=.7},  "\Two \tilde{a}"] & A ,
\end{tikzcd} 
\]
called the \emph{associativity} and \emph{unitality} constraints of the pseudocoalgebra, 
\end{itemize}
which satisfy  two coherence conditions. (These are dual to those for a pseudoalgebra~\cite[Section~4.2]{LackS:cohap}.)
\end{definition}

\begin{definition} Let $A$ and $B$ be pseudocoalgebras in $\catK$. A \myemph{pseudocoalgebra morphism} from $A$ to $B$ is 
a map $f \co A \to B$ equipped with an invertible 2-cell
\[
\begin{tikzcd}
A \ar[r, "f"] \ar[d, "a"'] \ar[dr,phantom,"\Two \bar{f}"] & B \ar[d, "b"] \\
\bang A \ar[r, "\bang f"'] & \bang B
\end{tikzcd}
\]
which satisfy  two coherence conditions. (These are dual to those for a pseudoalgebra morphism~\cite[Section~1.2]{BlackwellR:twodmt}.)
\end{definition}

\begin{definition} Let $f, g \co A \to B$ be pseudocoalgebra morphisms in $\catK$. A \myemph{coalgebra $2$-cell} from $f$ to $g$ is a 
2-cell $\phi \co f \Rightarrow g$ which satisfies one coherence condition. (This is dual to that for a pseudoalgebra 2-cell~\cite[Section~1.2]{BlackwellR:twodmt}.)
\end{definition}

We write $\Em{\catK}$ for the 2-category of  pseudocoalgebras, pseudocoalgebra
morphisms, and pseudocoalgebra 2-cells. This is connected to $\catK$ by
a biadjunction 
\begin{equation}
\label{equ:em-adjunction}
\begin{tikzcd}[column sep = large]
\catK  
	\ar[r, shift right =1, bend right = 10, "F"']  
	\ar[r, description, phantom, "\scriptstyle \bot"] &  
 \Em{\catK}   , 
	\ar[l, shift right = 1, bend right = 10, "U"'] 
\end{tikzcd}
\end{equation}
where the left biadjoint $U$ is the evident forgetful 2-functor, sending a pseudocoalgebra to its underlying object, and the right biadjoint $F$ sends
an object $B \in \catK$ to the cofree pseudocoalgebra on $B$, given by $\bang B$, viewed as a pseudocoalgebra with structure map
 $\dig_B \co \bang B \to \bbang B$. The associativity and unit constraints of the pseudocoalgebra  are obtained from those of the pseudocomonad. 
 
In view of an application in \cref{sec:products}, let us unfold more explicitly the biadjunction in~\eqref{equ:em-adjunction} in terms of universal properties.
For $B \in \catK$, the universal property of the pseudocoalgebra $\bang B$ means that for every pseudocoalgebra~$A$, with structure map $a \co A \to \bang A$, 
composition with $\der_B \co \bang B \to B$  gives an adjoint equivalence of categories
\begin{equation}
\label{equ:aux-equiv}
\begin{tikzcd}[column sep = large]
 \Em{\catK}[ A, \bang B]   \ar[r, "(-) \der_B"] & 
\catK[A, B]  .
\end{tikzcd} 
\end{equation}
In particular, for every $f \co A \to B$ in $\catK$, there is an essentially unique  pseudocoalgebra morphism $f^\sharp \co A \to \bang B$  such that $\der_B \circ f^\sharp \cong f$ in $\catK$. Here, essential
uniqueness refers to the fact that \eqref{equ:aux-equiv} is an equivalence, rather than an isomorphism.
Explicitly, $f^\sharp \co A \to \bang B$ is obtained as the composite
\begin{equation}
\label{equ:unfold-f-sharp}
\begin{tikzcd}
A \ar[r, "a"] &
\bang A \ar[r, "\bang f"] &
\bang B .
\end{tikzcd}
\end{equation}

We shall also make use of the Kleisli bicategory of the pseudocomonad, written $\Kl{\catK}$, whose objects
are the objects of~$\catK$ and whose hom-categories are defined by letting $\Kl{\catK}[A, B] \defeq \catK[\bang A, B]$,
for $A, B \in \catK$. As observed in~\cite{ChengE:psedl}, in general $\Kl{\catK}$ is a  bicategory, not a 2-category. There is a biadjunction 
\begin{equation}
\label{equ:cokleisliadjunction}
\begin{tikzcd}[column sep = large]
 \Kl{\catK} 
	\ar[r, shift left =1, bend left =10, "K"]  
	\ar[r, description, phantom, "\scriptstyle \bot"] &  
\catK .
	\ar[l, shift left = 1, bend left =10, "J"] 
\end{tikzcd}
\end{equation}
Here, $K$ sends $A \in \Kl{\catK} $ to $\bang A \in \catK$ and $f  \co \bang A \to B$ to the composite $\bang f \circ \dig_A \co \bang A \to \bang B$, 
while~$J$ is the identity on objects and sends $f \co A \to B$ to the composite $f \circ \der_A \co \bang A \to B$.

It is immediate 
that if $\catK$ has finite products, then so does $\Kl \catK$. 
We record this as a proposition since we will use it in the proof of
\cref{thm:coKleisli-cartesian-closed}. 

\begin{proposition} \label{thm:coKleisli-cartesian}
Let $\catK$ be a 2-category  and $(\bang, \dig, \der)$ a pseudocomonad on it. Assume $\catK$ has finite products.
Then the  Kleisli bicategory $\Kl \catK$ has finite products.
\end{proposition} 

\begin{proof} 
Follows from the facts that the right biadjoint $J \co \catK \to \Kl{\catK}$
preserves bicategorical limits and is the identity on objects.
\end{proof}

\subsection*{Symmetric lax monoidal pseudocomonads} As a second step towards our notion of a linear exponential pseudocomonad, we give the definition of a
symmetric lax monoidal pseudocomonad, assuming that $\catK$ has the structure of a symmetric Gray monoid. The definition involves not only 
the underlying 2-functor, comultiplication, and counit interacting suitably with the symmetric monoidal structure, but also the additional modifications
for associativity and unit laws. 

\begin{definition} \label{thm:sym-lax-monoidal-pseudocomonad}
Let $\catK$ be a symmetric Gray monoid. A \myemph{symmetric lax monoidal pseudocomonad} on $\catK$ is  a pseudocomonad
$(\bang, \dig, \der)$ on $\catK$ as in \cref{def:psd-comonad} equipped with additional structure and satisfying extra properties as follows:
\begin{itemize}
\item the underlying 2-functor of the pseudocomonad is sylleptic lax monoidal
  in the sense of \cref{thm:sym-mon-functor},  \ie we have a pseudonatural transformation with object components
\begin{equation}
\label{equ:monn}
\monn_{A, B} \co \bang A \otimes \bang B \to \bang (A \otimes B) ,
\end{equation}
for $A, B \in \catK$, a map 
\begin{equation}
\label{equ:moni}
\moni \co  \unit \to \bang  \unit ,
\end{equation}
and invertible modifications with components
\begin{equation}
\label{equ:lax-monoidal-associativity}
\begin{tikzcd}[column sep = huge, ampersand replacement=\&] 
\bang A \otimes \bang B \otimes \bang C \ar[r, "\id_A \otimes \monn_{B,C}"] \ar[d, "\monn_{A,B} \otimes \id_C"'] \ar[dr,phantom,"\Two\omega_{A,B,C}"] \& 
\bang A \otimes \bang (B \otimes C) \ar[d, "\monn_{A, B \otimes C}"]  \\
\bang (A \otimes B) \otimes \bang C \ar[r, "\monn_{A \otimes B, C}"'] \& \bang (A \otimes B \otimes C) ,
\end{tikzcd} 
\end{equation}
for $A, B, C \in \catK$, 
\begin{equation} 
\label{equ:lax-monoidal-unitality}
\begin{tikzcd}[column sep = large, ampersand replacement=\&] 
  \& 
\bang  \unit \otimes   \bang A 
	\ar[d, phantom, description, "\Two \kappa_A"]
	\ar[dr, "\monn_{ \unit, A}"]   \& 
   \\
\unit \otimes \bang A 
 	\ar[rr, "\id_{\bang A}"'] 
 	\ar[ur, "\moni \otimes\id_{\bang A}"]  \&
 \phantom{} \& 
\bang A  ,
 \end{tikzcd}   \qquad
\begin{tikzcd}[column sep = large, ampersand replacement=\&] 
  \& 
\bang A \otimes \bang  \unit
	\ar[d, phantom, description, "\Two \zeta_A"]
	\ar[dr, "\monn_{A,  \unit}"]   \& 
   \\
\bang A \otimes \unit
 	\ar[rr, "\id_{\bang A}"'] 
 	\ar[ur, "\id_{\bang A} \otimes \moni"]  \&
 \phantom{} \& 
\bang A  ,
 \end{tikzcd} 
 \end{equation}
 for $A \in \catK$, and 
 \begin{equation}
 \label{equ:lax-monoidal-braiding}
 \begin{tikzcd}[ampersand replacement=\&] 
 \bang A \otimes \bang B \ar[r, "\mathsf{r}_{\bang A, \bang B}"] \ar[d, "\monn_{A,B}"'] \ar[dr,phantom,"\Two \theta_{A,B}"] \& \bang B \otimes \bang A \ar[d, "\monn_{B,A}"] \\
 \bang (A \otimes B) \ar[r, "\bang \mathsf{r}_{A,B}"'] \& \bang (B \otimes A) ,
 \end{tikzcd}
 \end{equation}
 for $A, B \in \catK$, which satisfy the coherence conditions for a sylleptic lax monoidal pseudofunctor;
\item the comultiplication $\dig$  is  a braided  monoidal pseudonatural transformation in the sense of \cref{thm:bra-mon-transformation}, \ie we have an invertible modification with components
\begin{equation}
\label{equ:comultiplication-monoidal-2}
\begin{tikzcd}[column sep = large]
\bang A \otimes \bang B \ar[rr, "\monn_{A,B}"]  \ar[d, "\dig_A \otimes \dig_B"'] 
\ar[drr,phantom,"\Two \dig^2_{A,B}"]
& & \bang (A \otimes B) \ar[d, "\dig_{A \otimes B}"] \\
\bbang A \otimes \bbang B \ar[r, "\monn_{\bang A, \bang B}"']  & \bang ( \bang A \otimes \bang B) \ar[r, "\bang \monn_{A, B}"']  & \bbang (A \otimes B)  ,
\end{tikzcd}
\end{equation}
and an invertible 2-cell
\begin{equation}
\label{equ:comultiplication-monoidal-0}
\begin{tikzcd}[column sep = large]
 \unit \ar[r, "\moni"] \ar[d, "\moni"'] \ar[dr,phantom,"\Two \dig^0"] & \bang  \unit \ar[d, "\dig_ \unit"] \\
\bang  \unit \ar[r, "\bang \moni"'] & \bbang  \unit   ,
\end{tikzcd}
\end{equation}
which satisfy the coherence conditions for a braided monoidal transformation;
\item the counit $\der$ is a braided monoidal pseudonatural transformation in the sense of \cref{thm:bra-mon-transformation}, \ie we have an invertible modification with components
\begin{equation}
\label{equ:counit-monoidal-2}
\begin{tikzcd}[column sep = large]
\bang A \otimes \bang B \ar[r, "\monn_{A,B}"]  \ar[d, "\der_A \otimes \der_B"']  \ar[dr,phantom,"\Two \der^2_{A,B}"]
& \bang (A \otimes B) \ar[d, "\der_{A \otimes B}"]  \\
 A \otimes B \ar[r, "\id_{A \otimes B}"'] & A \otimes B  ,
 \end{tikzcd} 
 \end{equation}
for $A, B \in \catK$, and an invertible 2-cell
\begin{equation}
\label{equ:counit-monoidal-0}
 \begin{tikzcd} 
 \unit \ar[r, "\moni"] \ar[d, "\id_\unit"'] \ar[dr,phantom,"\Two \der^0"] & \bang \unit \ar[d, "\der_\unit"]  \\
\unit  \ar[r, "\id_\unit"'] & \unit  ,
  \end{tikzcd}
  \end{equation} 
which satisfy the coherence conditions for a braided monoidal transformation;
 \item the associativity and unitality constraints $\pi$, $\mu$, $\nu$  of the pseudocomonad are monoidal modifications in the sense of \cref{thm:monoidal-modification}.
\end{itemize}
\end{definition}

In the one-dimensional setting, it is well-known that the category of coalgebras for a symmetric monoidal comonad admits a symmetric
monoidal structure. The 2-dimensional counterpart of this
result requires some care, in that one does not get back a structure that is as strict as the one from which one starts.  Let us recall the precise statement from~\cite[Theorem~6.3.2 and Corollary~6.3.3]{MirandaA:paper-1}.

\begin{theorem}[Miranda] \label{thm:coalg-is-monoidal} Let $\catK$ be a symmetric Gray monoid and 
let $(\bang, \dig, \der)$ be a symmetric lax monoidal pseudocomonad on $\catK$. Then 
the 2-category of pseudocoalgebras $\Em{\catK}$ admits the structure of a symmetric monoidal bicategory so that
the forgetful 2-functor $U \co \Em{\catK} \to \catK$ is a strict symmetric monoidal 2-functor. 
Furthermore, the associativity and unit constraints for this monoidal bicategory are 2-natural and they satisfy the pentagon and triangle axioms strictly.
\end{theorem}

We outline some of the structure in the statement of \cref{thm:coalg-is-monoidal} for later use. Let $\catK$ be a symmetric Gray monoid and 
let $(\bang, \dig, \der)$ be a symmetric lax monoidal pseudocomonad on $\catK$, with data as in \cref{thm:sym-lax-monoidal-pseudocomonad}.
Let $A$ and $B$ be coalgebras with structure maps~$a \co A \to \bang A$
and~$b \co B \to \bang B$, respectively. Their tensor product in $\Em{\catK}$ is the pseudocoalgebra with underlying object
$A \otimes B$ and structure map the following composite:
\[
\begin{tikzcd}
A \otimes B \ar[r, "a \otimes b"] & \bang A \otimes \bang B \ar[r, "\monn_{A,B}"] & \bang (A\otimes B) . 
\end{tikzcd}
\]
The associativity constraint is the invertible 2-cell obtained by the following pasting diagram:
\[
\begin{tikzcd}[column sep = large]
A \otimes B 
	\ar[r, "a \otimes b"] 
	\ar[d, "a \otimes b"'] 
	\ar[dr,phantom,"\Two \bar{a} \otimes \bar{b}"] &[3em]
\bang A \otimes \bang B 
	\ar[r, "\monn_{A,B}"] \ar[d, "\dig_A \otimes \dig_B"]  
	\ar[ddr,phantom,"\Two \dig^2_{A,B}"]   &[3em]
\bang (A \otimes B) 
	\ar[dd, "\dig_{A \otimes B}"] \\
\bang A \otimes \bang B 
	\ar[r, "\bang a \otimes \bang b"] 
	\ar[d, "\monn_{A, B}"'] 
	\ar[dr,phantom,"\cong"] & 
\bbang A \otimes \bbang B 
	\ar[d, "\monn_{\bang A, \bang B}"] &  \\
  \bang (A \otimes  B)
  	\ar[r, "\bang (a \otimes b)"'] & 
 \bang (\bang A \otimes \bang B) 
 	\ar[r, "\bang \monn_{A, B}"'] & 
\bbang (A \otimes B) .
\end{tikzcd}
\]
The unitality constraint is the invertible 2-cell below:
\[
\begin{tikzcd}[column sep = huge]
A \otimes B \ar[r, "a \otimes b"] \ar[dr,  bend right = 20, "\id_{A \otimes B}"']  & 
\bang A \otimes \bang B \ar[d, "\der_A \otimes \der_B"]  \ar[r, "\monn_{A,B}"]  \ar[dr,phantom,"\Two \der^2_{A,B}"] &
\bang (A \otimes B) \ar[d, "\der_{A \otimes B}"] \\
 \phantom{} 
 \ar[ur, phantom, pos = (0.7), "\Two \tilde{a} \otimes \tilde{b}"]
 & A \otimes B \ar[r, "\id_{A \otimes B}"']  &  A \otimes B  .
 \end{tikzcd} 
 \]
The unit of the monoidal structure is the pseudocoalgebra with underlying object $\unit$ and coalgebra structure map $\moni \co \unit \to \bang \unit$. Its associativity
and unitality constraints are given by the invertible 2-cells $\dig^0$ and $\der^0$ in~\eqref{equ:comultiplication-monoidal-0} and~\eqref{equ:counit-monoidal-0}, respectively.

A subtle point is that, even if the tensor product of $\catK$ is strictly unital and associative, the tensor product in $\Em{\catK}$ need not be, even if the underlying objects of the
relevant pseudocoalgebras are equal. The key point here is that, for pseudocoalgebras $A, B, C$, the pseudocoalgebra structures on $(A \otimes B) \otimes C$ and $A \otimes (B \otimes C)$ are 
isomorphic in $\Em{\catK}$. Indeed, the identity map $\id_{A \otimes B \otimes C}$ is a pseudocoalgebra morphism between them.

\begin{remark} \label{equ:lift-braiding-to-em}
The content of \cref{thm:coalg-is-monoidal} goes much further than merely lifting tensor product and unit from $\catK$ to $\Em{\catK}$, in that it asserts that all the data for the symmetric Gray monoid $\catK$,
as in \cref{def:symmetric-gray-monoid}, lifts to $\Em{\catK}$. This means in particular that:
\begin{enumerate}
\item  the braiding $\bra_{A, B} \co A \otimes B \to B \otimes A$ is a pseudocoalgebra morphism for all pseudocoalgebras $A, B$;
\item the components $\beta^1_{A,B,C}$ and $\beta^2_{A,B,C}$ of the braiding constraints are pseudocoalgebra 2-cells for all pseudocoalgebras  $A, B, C$;
\item the component $\sigma_{A,B}$ of the syllepsis is a pseudocoalgebra 2-cell for all pseudocoalgebras $A$ and $B$.
\end{enumerate}
\end{remark}

\begin{corollary} Let $\catK$ be a symmetric Gray monoid and 
$(\bang, \dig, \der)$  a symmetric lax monoidal pseudocomonad on $\catK$.
The forgetful-cofree biadjunction between $\catK$ and $\Em{\catK}$ in~\eqref{equ:em-adjunction} lifts to a symmetric lax monoidal biadjunction, as in
\[
\begin{tikzcd}[column sep = large]
(\catK, \otimes, \unit) 
	\ar[r, shift right =1, bend right = 10, "F"']  
	\ar[r, description, phantom, "\scriptstyle \bot"] &  
(\Em{\catK}, \otimes, \unit)   . 
	\ar[l, shift right = 1, bend right = 10, "U"'] 
\end{tikzcd}
\]
\end{corollary}

\begin{proof} Since the left biadjoint~$U$ is a sylleptic strict monoidal pseudofunctor, the claim 
follows by a two-dimensional version of Kelly's doctrinal adjunction, \cf \cite[pages~62-63]{GarnerR:enrcfc}. In particular, the right biadjoint~$F$ becomes a sylleptic lax monoidal 
pseudofunctor by~\cite[Proposition~15]{DayStreet}. 
\end{proof}

For a 2-category $\catK$ and pseudocomonad $(\bang, \dig, \der)$ on it, it is immediate to observe that the components of the comultiplication $\dig_A \co \bang A \to \bbang A$ are
pseudocomonad morphisms, the pseudonaturality 2-cells 
\[
\begin{tikzcd}
\bang A \ar[r, "\dig_A"] \ar[d, "\bang f"'] \ar[dr,phantom,"\Two \bar{\dig}_f"] & \bbang A \ar[d, "\bbang f"] \\
\bang B \ar[r, "\dig_B"'] & \bbang B ,
\end{tikzcd}
\]
for $f \co A \to B$ in $\catK$, are pseudocoalgebra 2-cells (since $\pi$ is a modification), and the components $\pi_A$ of its associativity constraints, for $A \in \catK$,  are pseudocoalgebra 2-cells (by the
coherence axioms for a pseudocomonad). When $\catK$ is a Gray
monoid and the pseudocomonad is sylleptic lax monoidal, as in \cref{thm:sym-lax-monoidal-pseudocomonad}, other parts of the structure can also be lifted to
the 2-category of pseudocoalgebras. This is described in the next two lemmas. The second one concerns pseudocoalgebra 2-cells and therefore does not have a one-dimensional counterpart.

Let us fix a symmetric Gray monoid $\catK$ and a symmetric lax monoidal pseudocomonad $(\bang, \dig, \der)$ on it as in \cref{thm:sym-lax-monoidal-pseudocomonad}.

\begin{lemma} \label{thm:mu-are-coalgebra-morphisms} \leavevmode
\begin{enumerate}[(i)]
\item For every $A, B \in \catK$, the map $\monn_{A, B} \co \bang A \otimes \bang B \to \bang (A \otimes B)$ is a pseudocoalgebra morphism, \ie we have an invertible 2-cell
\[
\begin{tikzcd}[column sep = large]
\bang A \otimes \bang B \ar[d, "\dig_A \otimes \dig_B"'] \ar[r, "\monn_{A,B}"] 
 \ar[ddr, phantom, description, "\Two"]
&  \bang (A \otimes B) \ar[dd, "\dig_{A \otimes B}"]  \\
\bbang A \otimes \bbang B \ar[d, "\monn_{\bang A, \bang B}"'] & \\
\bang (\bang A \otimes \bang B) \ar[r, "\bang \monn_{A,B}"'] & \bbang (A \otimes B) 
\end{tikzcd}
\]
such that the appropriate coherence conditions hold.
\item The map $\moni \co \unit \to \bang \unit$ is a pseudocoalgebra morphism, \ie we have an invertible 2-cell
\[
\begin{tikzcd}[column sep = large]
\unit \ar[r, "\moni"] \ar[dd, "\moni"']  \ar[ddr, phantom, description, "\Two"] & \bang \unit \ar[dd, "\dig_\unit"] \\ 
 & \\ 
\bang \unit \ar[r, "\bang \moni"'] & \bbang \unit
\end{tikzcd}
\]
such that the appropriate coherence conditions hold.
\end{enumerate} 
\end{lemma}

\begin{proof} For part~(i), the required invertible 2-cell is given by the 2-cell $\dig^2_{A,B}$ in~\eqref{equ:comultiplication-monoidal-2}. The two coherence conditions for a pseudoalgebra morphism follow from the first\footnote{{Cf.}~\cref{thm:coh-numb-convention}.}
coherence axiom of a monoidal modification for $\pi$ (the associativity constraint of the pseudocomonad) and for $\nu$ (the right unitality constraint of the pseudomonad), 
respectively. 

For part (ii), it suffices to recall that the map $\moni \co \unit \to \bang \unit$ is the structure map of a pseudocoalgebra.
\end{proof} 

Next, we show that the 2-cells that are part of the structure of a symmetric lax monoidal pseudomonad can also be lifted. For this statement to make sense, recall that the components of the 
comultiplication $\dig_A \co \bang A \to \bbang A$, for $A \in \catK$ are pseudoalgebra morphisms, the statement of \cref{thm:mu-are-coalgebra-morphisms}, and that
the braiding $\bra_{A, B} \co A \otimes B \to B \otimes A$ is a pseudoalgebra morphism for all pseudoalgebras $A, B$ by \cref{equ:lift-braiding-to-em}.

\begin{lemma} \label{thm:to-be-added-1} \leavevmode
\begin{enumerate}
\item The 2-cell $\omega_{A, B, C}$   in \eqref{equ:lax-monoidal-associativity} is a pseudocoalgebra 2-cell for all  pseudocoalgebras~$A, B, C$.
\item The 2-cells $\kappa_A$ and $\zeta_A$ in \eqref{equ:lax-monoidal-unitality} are  pseudocoalgebra 2-cells for every  pseudocoalgebra~$A$.
\item The 2-cell $\theta_{A,B}$ in \eqref{equ:lax-monoidal-braiding} is a pseudocoalgebra 2-cell for all pseudoalgebras $A, B$.
\item The 2-cell $\dig^2_{A,B}$ in \eqref{equ:comultiplication-monoidal-2}  is a pseudocoalgebra 2-cell   for all pseudoalgebras $A, B$.
\item The 2-cell $\dig^0$ in \eqref{equ:comultiplication-monoidal-0} is a  pseudocoalgebra 2-cell.
\end{enumerate}
\end{lemma}

\begin{proof} For parts (i) and (ii), use the first, second and third coherence conditions for the comultiplication~$\dig$ to be a monoidal transformation to prove the claim for~$\omega_{A, B, C}$, $\kappa_A$ and~$\zeta_A$,
respectively. For part~(iii), use the condition expressing that the monoidal transformation~$\dig$ is braided. For parts~(iv) and (v), use the first and second coherence conditions for the modification $\pi$ (the associativity constraint of the pseudocomonad) to be monoidal, respectively.
\end{proof}

\subsection*{Linear exponential pseudocomonads} 
We are now ready to introduce our bicategorical counterpart of  linear exponential comonads, which is based on the one-dimensional axiomatisation of~\cite{HylandM:gluoml}.
Note that, since the 2-category of pseudocoalgebras $\Em \catK$  for a symmetric monoidal pseudocomonad $(\bang, \dig, \der)$ is a symmetric monoidal bicategory by~\cref{thm:coalg-is-monoidal}, 
we can consider  symmetric pseudocomonoids therein, in the sense of \cref{thm:comonoid} and \cref{thm:braided-comonoid}.

\begin{definition} \label{thm:linear-exponential-comonad}
Let $\catK$ be a symmetric Gray monoid. A \myemph{linear exponential pseudocomonad} on $\catK$ is a symmetric lax monoidal pseudocomonad $(\bang, \dig, \der)$ on $\catK$
as in \cref{thm:sym-lax-monoidal-pseudocomonad}, 
such that every pseudocoalgebra $A \in \Em \catK$ is a symmetric pseudocomonoid
in $\Em \catK$, pseudonaturally and monoidally in $A$. 
\end{definition}

\begin{remark} \label{thm:unfold-linear-exponential} For later use, 
we unfold  what it means that every pseudocoalgebra $A \in \Em \catK$ is a symmetric pseudocomonoid in $\Em \catK$, pseudonaturally and monoidally in $A$.
While this is the counterpart in $\Em \catK$ of part~(ii) of \cref{thm:prod-iff-all-ccmon}, it is convenient 
to record explicitly that all the maps and 2-cells under consideration are in~$\Em{\catK}$.

\begin{itemize}
\item Every pseudocoalgebra~$A$ is a symmetric pseudocomonoid, in the sense of of \cref{thm:symmetric-pseudocomonoid}, in $\Em{\catK}$, \ie 
we have pseudocoalgebra morphisms 
\begin{equation}
\label{equ:lin-exp-com-and-cou}
\com_A \co A \to  A \otimes A , \qquad
e_A \co A \to \unit ,
\end{equation}
which come equipped with invertible 2-cells 
 \[
\begin{tikzcd}
 A \ar[r, "\com_A"]  \ar[dd, "a"'] \ar[ddr, phantom, description, "\Two \bar{n}_A"] &  A \otimes  A \ar[d, "a \otimes a"]  \\
 &  \bang A \otimes \bang A \ar[d, "\monn_{A, A}"]  \\
\bang A \ar[r, "\bang \com_A"'] &  \bang (A \otimes A) ,
\end{tikzcd} \qquad
\begin{tikzcd}
A \ar[r, "e_A"] \ar[d, "a"'] \ar[dr, phantom, description, "\Two \bar{e}_A"] & \unit \ar[d, "\moni"] \\
\bang A \ar[r, "\bang e_\unit"']   & \bang \unit 
\end{tikzcd}
\]
of $\catK$, and invertible pseudocoalgebra 2-cells 
\begin{gather}
\label{equ:lin-exp-associativity}
\begin{tikzcd}[column sep = large, ampersand replacement=\&] 
A 
	\ar[r, "\com_A"] \ar[d, "\com_A"']  
	\ar[dr,phantom,"\Two \alpha_A"] \& 
A \otimes A 
	\ar[d, "\id_A \otimes \com_A"] \\
A \otimes A 
	\ar[r, "\com_A \otimes \id_A"'] \& 
A \otimes A \otimes A  
,
\end{tikzcd}  \\
\label{equ:lin-exp-unitality}
\begin{tikzcd}[column sep = {1.5cm,between origins}, ampersand replacement=\&] 
  \& 
A \otimes A
	\ar[d, phantom, description, "\Two \lambda_A"]
	\ar[dr, "e_A \otimes \id_A"]   \& 
   \\
A
 	\ar[ur, "\com_A"] 
 	\ar[rr, "\id_A"']  \&
 \phantom{} \& 
\unit \otimes A ,
 \end{tikzcd} \qquad
\begin{tikzcd}[column sep = {1.5cm,between origins}, ampersand replacement=\&] 
  \& 
A \otimes A
	\ar[d, phantom, description, "\Two \rho_A"]
	\ar[dr, "\id_A \otimes e_A"]   \& 
   \\
A
 	\ar[ur, "\com_A"] 
 	\ar[rr, "\id_A"']  \&
 \phantom{} \& 
A \otimes \unit ,
 \end{tikzcd} \\
\label{equ:lin-exp-braiding} 
\begin{tikzcd}[column sep = {1.5cm,between origins}, ampersand replacement=\&] 
  \& 
A \otimes A
	\ar[d, phantom, description, "\Two \gamma_A"]
	\ar[dr, "\bra_{A,A}"]   \& 
   \\
A
 	\ar[ur, "\com_A"] 
 	\ar[rr, "\com_A"']  \&
 \phantom{} \& 
A \otimes A ,
 \end{tikzcd}
\end{gather}
satisfying the appropriate coherence axioms.
\item The families $(\com_A)_{A \in \catK}$ and $(e_A)_{A \in \catK}$ 
 are pseudonatural transformations with respect to pseudocoalgebra morphisms, \ie for every pseudocoalgebra morphism $f \co A \to  B$, we have 
  invertible pseudocoalgebra   2-cells
\begin{equation}
\label{equ:d-psdnat}
\begin{tikzcd} 
A \ar[r, "\com_A"] \ar[d, "f"'] \ar[dr, phantom, description, "\Two n_f"]  &  A \otimes  A \ar[d, "f \otimes f"] \\
 B \ar[r, "\com_B"'] &  B \otimes  B ,
\end{tikzcd} 
\end{equation}
and
\begin{equation}
\label{equ:e-psdnat}
\begin{tikzcd} 
 A \ar[r, "e_A"] \ar[d, "f"']  \ar[dr, phantom, description, "\Two e_f"]  & \unit \ar[d, "\id_\unit"] \\
 B \ar[r, "e_B"'] & \unit ,
 \end{tikzcd} 
\end{equation}
satisfying the pseudonaturality axioms. 
\item The families of 2-cells $(\alpha_A)_{A \in \catK}$, $(\lambda_A)_{A \in \catK}$, $(\rho_A)_{A \in \catK}$ and $(\gamma_A)_{A \in \catK}$ are modifications.
\item The pseudonatural transformations $\com$ and $e$ are  braided monoidal in the sense of \cref{thm:bra-mon-transformation}, 
\ie we have a modification $\com^2$, with components invertible pseudocoalgebra 2-cells, 
and an invertible pseudocoalgebra 2-cell $\com^0$, as in 
\begin{equation}
\label{equ:d-monoidal}
\begin{tikzcd}[column sep = large]
A \otimes B \ar[r, "\com_{A \otimes B}"] \ar[d, "\id_{A \otimes B}"'] \ar[dr, phantom, description, "\Two \com^2_{A,B}"]  & A \otimes B \otimes A \otimes B \ar[d, "\id_A \otimes \bra_{B,A} \otimes \id_B"] \\
A \otimes B \ar[r, "\com_A \otimes \com_B"'] & A \otimes A \otimes B \otimes B ,
 \end{tikzcd} 
  \qquad
\begin{tikzcd} 
\unit \ar[d, "\id_\unit"'] \ar[r, "\com_\unit"] \ar[dr, phantom, description, "\Two \com^0"] & \unit  \otimes \unit  \ar[d, "\id_\unit"] \\
\unit \ar[r, "\id_\unit"'] &  \unit ,
 \end{tikzcd} 
\end{equation}
 an invertible modification $e^2$, with components invertible pseudocoalgebra 2-cells, and an invertible pseudocoalgebra 2-cell $e^0$, as in 
\begin{equation}
\label{equ:e-monoidal}
\begin{tikzcd}[column sep = large]
A \otimes B \ar[r, "e_A \otimes e_B"] \ar[d, "\id_{A \otimes B}"'] \ar[dr,
phantom, description, "\Two e^2_{A,B}"]  & \unit \otimes \unit \ar[d, "\id_{\unit}"] \\
A \otimes B \ar[r, "e_{A \otimes B}"'] & \unit ,
\end{tikzcd} \qquad
\begin{tikzcd}
\unit \ar[r, "\id_\unit"] \ar[d, "\id_\unit"'] \ar[dr, phantom, description, "\Two e^0"] & \unit \ar[d, "\id_\unit"] \\
\unit \ar[r, "e_\unit"'] & \unit ,
\end{tikzcd}
\end{equation}
satisfying the coherence axioms of \cref{thm:monoidal-pseudonatural-transformation}.
\item The modifications $\alpha$, $\lambda$, $\rho$, $\gamma$ are monoidal, in the sense of \cref{thm:monoidal-modification}.
\end{itemize}
\end{remark}

\begin{remark} \label{thm:retract-cofree}
 In the one-dimensional setting, a possible axiomatisation of the notion of a linear
exponential comonad requires  the existence of a commutative comonoid structure only
on cofree coalgebras. From this one can derive the presence of 
a commutative comonoid structure on every coalgebra, using the fact that every
coalgebra lies in an equaliser diagram of
free ones which splits at the level of underlying objects. We expect that an analogous argument applies in the two-dimensional
setting, \cf also \cref{thm:compare-with-jacq} below, but do not pursue the details here.
\end{remark} 

Let us now fix a symmetric Gray monoid $\catK$ and a linear exponential pseudocomonad~$(\bang, \dig, \der)$ on it, with data as in \cref{def:psd-comonad,thm:sym-lax-monoidal-pseudocomonad,thm:linear-exponential-comonad} and \cref{thm:unfold-linear-exponential}.

\begin{proposition} \label{thm:coalgebras-is-cartesian} Let $\catK$ be a symmetric Gray monoid and $(\bang, \dig, \der)$ be a linear exponential pseudocomonad.
The  symmetric monoidal structure on the 2-category of pseudocoalgebras $\Em{\catK}$ of \cref{thm:coalg-is-monoidal} is a cartesian monoidal structure.
In particular, $\Em{\catK}$ has finite products.
\end{proposition}

\begin{proof}   \cref{thm:prod-iff-all-ccmon} implies the claim. 
\end{proof} 

\begin{remark} \label{thm:canonical-is-linear-exponential}
In the setting of~\cref{thm:coalgebras-is-cartesian}, the canonical symmetric pseudocomonoid structure on a pseudoalgebra determined by products  in $\Em{\catK}$ is the one given by the linear exponential pseudocomonad. Also note that, by \cref{thm:prod-iff-all-ccmon}, a linear exponential pseudocomonad could be 
defined equivalently as a symmetric lax monoidal pseudocomonad such that the monoidal
structure on its 2-category of pseudocoalgebras is cartesian.
\end{remark}

We derive some more consequences of the axioms for a linear exponential pseudocomonad. We give explicit proofs even if 
some of the statements follow from~\cref{thm:prod-iff-all-ccmon} (\cf \cref{thm:canonical-is-linear-exponential}).

For the next lemma, recall that if $A$ and $B$ are pseudocoalgebras, then $A \otimes B$ admits the structure of a pseudocoalgebra by \cref{thm:coalg-is-monoidal} and thus acquires the structure of a symmetric pseudocomonoid by the axioms for a linear exponential pseudocomonad.
Recall also that the 2-category $\SymCoMon{\catK}$ has a symmetric monoidal structure by \cref{thm:sym-comon-is-monoidal}.

\begin{lemma} 
\label{thm:id-is-comonoid-morphism}
\leavevmode
\begin{enumerate}[(i)] 
\item Let $A$ and $B$ be pseudocoalgebras. Then the symmetric pseudocomonoid  $(A \otimes B, \com_{A \otimes B}, e_{A \otimes B})$, determined by the linear 
exponential pseudocomonad, is equivalent to the tensor product of the symmetric pseudocomonoids
$(A, \com_A, e_A)$ and $(B, \com_B, e_B)$ in $\SymCoMon{\catK}$.
\item The symmetric pseudocomonoid $(\unit, \com_\unit, e_\unit)$ determined by the linear exponential pseudocomonad 
is equivalent to the unit of $\SymCoMon \catK$. 
\end{enumerate}
\end{lemma}

\begin{proof} For~(i), recall the definition of the tensor product $(A, \com_A, e_B) \otimes (B, \com_B, e_B)$ from the comments below~\cref{thm:coalg-is-monoidal}. In particular, 
its underlying object is $A \otimes B$, with comultiplication as in~\eqref{equ:comultiplication-for-tensor} and counit as in~ \eqref{equ:counit-for-tensor}.  

We claim that the identity map $\id_{A \otimes B} \co A \otimes B \to A \otimes B$ is a braided pseudocomonoid morphism
between $(A \otimes B, \com_{A \otimes B}, e_{A \otimes B})$ and $(A, \com_A, e_B) \otimes (B, \com_B, e_B)$. For this, we need 2-cells fitting in
the diagrams
\[
\begin{tikzcd}[column sep = huge]
A \otimes B \ar[r, "\id_{A \otimes B}"] \ar[dd, "\com_{A \otimes B}"'] 
 \ar[ddr, phantom, description, "\Two"]
& A \otimes B  \ar[d, "\com_A \otimes \com_B"] \\
 & A \otimes A \otimes B \otimes B \ar[d, "\id_A \otimes \mathsf{r}_{A,B} \otimes \id_B"] \\
 A \otimes B \otimes A \otimes B \ar[r, "\id_{A \otimes B \otimes A \otimes B}"'] & A \otimes B \otimes A \otimes B ,
 \end{tikzcd}
 \qquad 
 \begin{tikzcd}
 A \otimes B \ar[r, "\id"] \ar[dd, "e_{A \otimes B}"'] 
 \ar[ddr, phantom, description, "\Two"]
 & A \otimes B \ar[d, "e_A \otimes e_B"] \\
 & \unit \otimes \unit \ar[d, "\id"] \\
 \unit \ar[r, "\id"'] & \unit .
 \end{tikzcd}
 \]
 These are given by the 2-cells $\com^2_{A,B}$ and $e^2_{A,B}$ in~\eqref{equ:d-monoidal} and~\eqref{equ:e-monoidal}, which are part of the data making~$\com$ and~$e$ into braided monoidal
 pseudonatural transformations, respectively. We need to check the four coherence conditions for a braided pseudocomonoid morphism. The four proofs use 
 the first\footnote{{Cf.} \cref{thm:coh-numb-convention}.} coherence condition for the four modifications $\alpha$, $\lambda, \rho$ and $\gamma$ in \eqref{equ:lin-exp-associativity}, \eqref{equ:lin-exp-unitality}, and \eqref{equ:lin-exp-braiding} to be monoidal, respectively. 
  
 For part~(ii), recall that the unit $\unit$ is a braided pseudocomonoid with identity maps for both comultiplication and unit, and identity 2-cells for all the constraints. We claim that the identity 
 map $\id_\unit \co \unit \to \unit$ is a braided pseudocomonoid morphism from $(\unit, \com_\unit, e_\unit)$ to~$(\unit, \id_\unit, \id_\unit)$. Thus, we need invertible 2-cells:
 \[
 \begin{tikzcd}[column sep = large]
 \unit \ar[r, "\id_\unit"] \ar[d, "\com_\unit"'] \ar[dr,phantom,"\Two"] & \unit \ar[d, "\id_\unit"] \\
 \unit \otimes \unit \ar[r, "\id_\unit \otimes \id_\unit"'] & \unit \otimes \unit ,
 \end{tikzcd} \qquad
  \begin{tikzcd}[column sep = large]
 \unit \ar[r, "\id_\unit"] \ar[d, "e_\unit"'] \ar[dr,phantom,"\Two"] & \unit \ar[d, "\id_\unit"] \\
 \unit \ar[r, "\id_\unit"'] & \unit .
 \end{tikzcd} 
 \]
 These are given by the 2-cells $\com^0$ and $e^0$ in ~\eqref{equ:d-monoidal} and~\eqref{equ:e-monoidal}, respectively. To verify the four coherence conditions for a braided pseudocomonoid morphism, use
 the second coherence condition for the modifications $\alpha$, $\lambda, \rho$, and $\gamma$ to be monoidal, respectively.
\end{proof}

\begin{corollary} Let $A$ and $B$ be pseudocoalgebras. The symmetric pseudocomonoid  $(A \otimes B, \com_{A \otimes B}, e_{A \otimes B})$ is the product of $A$ and $B$ in $\SymCoMon{\catK}$. 
\end{corollary} 

\begin{proof} Immediate from~\cref{thm:id-is-comonoid-morphism} and \cref{thm:ccmon-finprod}.
\end{proof}

We now define explicitly two pseudonatural transformations, called \emph{contraction} and \emph{weakening} following the terminology used in Linear Logic,  which will be important
for our development.
For $A \in \catK$, the cofree pseudocoalgebra $\bang A$ admits the structure of a symmetric pseudocomonoid 
by the axioms for a linear exponential comonad, like any pseudocoalgebra. 
We define the map~$\con_A$, called \emph{contraction}, to be the comultiplication of the symmetric pseudocomonoid $\bang A$, \ie we let
\begin{equation}
\label{equ:con}
\begin{tikzcd}
 \bang A  \ar[r, "\con_A"] &  \bang A \otimes \bang A
 \end{tikzcd}  \quad \defeq \quad
 \begin{tikzcd}
 \bang A  \ar[r, "\com_{\bang A}"] &  \bang A \otimes \bang A .
 \end{tikzcd}
 \end{equation} 
 Similarly, we define the map $\weak_A$, called \emph{weakening}, to be the counit of the symmetric pseudocomonoid $\bang A$, \ie we let
 \begin{equation}
 \label{equ:weak}
\begin{tikzcd} 
\bang A \ar[r, "\weak_A"] &  \unit
 \end{tikzcd} 
\quad \defeq \quad
 \begin{tikzcd} 
\bang A \ar[r, "e_{\bang A}"] &  \unit .
 \end{tikzcd} 
 \end{equation}  
By definition, $\con_A$ and $\weak_A$ are the components on objects of pseudonatural transformations.

\begin{samepage}
\begin{lemma} \label{thm:con-and-weak-coalgebra-maps}
\leavevmode
\begin{enumerate}
\item For every $A \in \catK$,
the map $\con_A \co \bang A \to \bang A \otimes \bang A$ is a pseudocoalgebra morphism, \ie we have an invertible 2-cell
\[
\begin{tikzcd}[column sep = large] 
 \bang A \ar[r, "\con_A"]  \ar[dd, "\dig_A"'] 
 \ar[ddr, phantom, description, "\Two \bar{\con}_A"]
 &  \bang A \otimes  \bang A \ar[d, "\dig_A \otimes \dig_A"]  \\
 &  \bbang A \otimes \bbang A \ar[d, "\monn_{\bang A, \bang A}"]  \\
\bbang A \ar[r, "\bang \con_A"']  &  \bang (\bang A \otimes \bang A) ,
\end{tikzcd}
\]
such that  the appropriate coherence conditions hold.
\item For every $A \in \catK$,
the  map $\weak_A \co \bang A \to \unit$ is a pseudocoalgebra morphism, \ie we have an invertible 2-cell
\begin{equation}
 \label{thm:dig-and-weak}
\begin{tikzcd}[column sep = large]
\bang A \ar[r, "\weak_A"] \ar[d, "\dig_A"'] \ar[dr, phantom, description, "\Two \bar{\weak}_A"] & \unit \ar[d, "\moni"]  \\
\bbang A \ar[r, "\bang{\weak}_A"']  & \bang \unit ,
\end{tikzcd}
\end{equation}
such that  the appropriate coherence conditions hold.
\end{enumerate}
\end{lemma}
\end{samepage}

\begin{proof} By the definitions of  $\con_A$ and $\weak_A$ in \eqref{equ:con} and \eqref{equ:weak}, the claims are instances of the assertions that the components of the pseudonatural transformations~$\com$ and $e$ 
in~\eqref{equ:lin-exp-com-and-cou} are pseudocoalgebra morphisms, which are part of the axioms for a linear exponential pseudocomonad, \cf \cref{thm:unfold-linear-exponential}.
\end{proof}

The next proposition is a useful result since it allows us to leverage our earlier results on sylleptic lax monoidal pseudomonads to
exhibit a number of braided pseudocomonoid morphism and pseudocomonoid 2-cells.

\begin{proposition} \leavevmode
\label{thm:coalg-mor-comon-mor}
\begin{enumerate}
\item  Let $(A,a), (B, b)$ be pseudocoalgebras. Every pseudocoalgebra morphism $f \co (A,a) \to (B, b)$ is a braided pseudocomonoid morphism $f \co (A, \com_A, e_A) \to (B, \com_B, e_B)$.
\item Let $f \co A \to B$, $g \co A \to B$ be pseudocoalgebra morphisms. Every pseudocoalgebra 2-cell $f \Rightarrow g$ is a pseudocomonoid 2-cell.
\end{enumerate}
\end{proposition} 

\begin{proof} For (i), let $f \co (A,a) \to (B, b)$ be a pseudocoalgebra morphism. The 2-cells making $f$ into a braided pseudocomonoid morphism
are exactly the 2-cells in~\eqref{equ:d-psdnat} and \eqref{equ:e-psdnat} above. The four coherence conditions for a braided pseudocomonoid morphism
follow from the assumption that $\alpha$, $\lambda$, $\rho$, and $\gamma$ in~\eqref{equ:lin-exp-associativity}, \eqref{equ:lin-exp-unitality}, and \eqref{equ:lin-exp-braiding} are modifications, respectively.

For part~(ii), the two coherence conditions follow from the axiom describing the interaction between pseudonatural transformations and 2-cells for $\com$ and $e$, respectively.
Note that one has to use that the 2-cell under consideration is a pseudocoalgebra 2-cell to apply this axiom.
\end{proof}

\begin{corollary}  \label{thm:dig-comonoid-morphism}
\leavevmode
 \begin{enumerate}
 \item For every $A \in \catK$, the  map $\dig_A \co \bang A \to \bbang A$ is 
a braided pseudocomonoid morphism, \ie we have invertible 2-cells
\[
\begin{tikzcd}[column sep = large]
\bang A 
	\ar[d, "\con_A"'] 
	\ar[r, "\dig_A"]  
	\ar[dr, phantom, description, "\Two \bar{\dig}_A"] & 
\bbang A 
	\ar[d, "\con_{\bang A}"]  \\
\bang A \otimes \bang A 
	\ar[r, "\dig_A \otimes \dig_A"'] & 
\bbang A \otimes \bbang A ,
\end{tikzcd} \qquad 
\begin{tikzcd}
\bang A 
	\ar[d, "\weak_A"'] 
	\ar[r, "\dig_A"] 
	\ar[dr, phantom, description, "\Two \tilde{\dig}_A"]  & 
\bbang A  
	\ar[d, "\weak_{\bang A}"] \\
\unit 
	\ar[r, "\id_\unit"'] 	& 
\unit ,
\end{tikzcd}
\]
such that the appropriate coherence conditions hold.
\item For every $A \in \catK$, the associativity constraint 2-cell $\pi_A$ of the pseudocomonad is a pseudocomonoid 2-cell.
\end{enumerate}
\end{corollary}

\begin{proof} For part~(i), $\dig_A \co \bang A \to \bbang A $ is the structure map of a pseudocoalgebra and so it is a pseudocoalgebra morphism.
The claim then follows by part~(i) of \cref{thm:coalg-mor-comon-mor}. 
For part~(ii), $\pi_A$ is a pseudocoalgebra 2-cell and so the claim follows by part~(ii) of  \cref{thm:coalg-mor-comon-mor}.
\end{proof} 

\begin{corollary} \label{thm:mu-comonoid-morphism} \leavevmode
\begin{enumerate}[(i)] 
\item For every $A, B \in \catK$, the map $\monn_{A, B} \co \bang A \otimes \bang B \to \bang (A \otimes B)$ is a braided pseudocomonoid morphism, \ie we have invertible 2-cells
\[
\begin{gathered} 
\begin{tikzcd}[column sep =huge]
\bang A \otimes \bang B 
	\ar[r, "\monn_{A,B}"] 
	\ar[d, "\con_A \otimes \con_B"']
	\ar[ddr, phantom, description, "\Two \bar{\mon}_{A,B}"] & 
\bang (A \otimes B) 
	\ar[dd, "\con_{A \otimes B}"] \\
\bang A \otimes \bang A \otimes \bang B \otimes \bang B
	\ar[d, "\id_{\bang A} \otimes \bra_{\bang A,\bang B} \otimes \id_{\bang B}"'] &  \\
\bang A \otimes \bang B \otimes \bang A \otimes \bang B
	\ar[r, "\monn_{A,B} \otimes \monn_{A, B}"'] &
\bang (A \otimes B) \otimes \bang ( A \otimes B)  ,
\end{tikzcd} \\[2ex]
\begin{tikzcd}[column sep = huge]
\bang A \otimes \bang B 
	\ar[r, "\monn_{A,B}"] 
	\ar[d, "\weak_A \otimes \weak_B"'] 
	\ar[dr, phantom, description, "\Two \tilde{\mon}_{A,B}"] & 
\bang (A \otimes B) 
	\ar[d, "\weak_{A \otimes B}"] \\
\unit \otimes \unit \ar[r, "\id"'] & \unit  ,
\end{tikzcd}
\end{gathered}
\]
such that the appropriate coherence conditions hold.
\item The map $\moni \co \unit \to \bang \unit$ is a braided pseudocomonoid morphism, \ie we have invertible 2-cells
\[
\begin{tikzcd}[column sep = large]
\unit 
	\ar[r, "\moni"] 
	\ar[d, "\id"'] 
	\ar[dr, phantom, description, "\Two \bar{\mon}_\unit"] & 
\bang \unit 
	\ar[d, "\con_\unit"]  \\
\unit \otimes \unit 
	\ar[r, "\moni \otimes \moni"'] & 
	\bang \unit \otimes \bang \unit  ,
\end{tikzcd}
\qquad
\begin{tikzcd} 
\unit 
	\ar[r, "\moni"]  
	\ar[d, "\id"'] 
	\ar[dr, phantom, description, "\Two \tilde{\mon}_\unit"] & 
\bang \unit 
	\ar[d, "\weak_\unit"] \\
\unit 
	\ar[r, "\id"'] &
\unit ,
\end{tikzcd}
\]
such that the appropriate coherence conditions hold.
\end{enumerate}
\end{corollary}

\begin{proof} By \cref{thm:mu-are-coalgebra-morphisms} and \cref{thm:coalg-mor-comon-mor}.
\end{proof}

\begin{corollary} \label{thm:to-be-added-2} \leavevmode
\begin{enumerate}
\item The 2-cell $\omega_{A, B, C}$  in \eqref{equ:lax-monoidal-associativity} is a pseudocomonoid 2-cell for all  pseudocoalgebras $A, B, C$.
\item The 2-cells $\kappa_A$ and $\zeta_A$ in \eqref{equ:lax-monoidal-unitality} are  pseudocomonoid 2-cells for every  pseudocoalgebra $A$.
\item The 2-cell $\theta_{A,B}$ in \eqref{equ:lax-monoidal-braiding} is a pseudocomonoid 2-cell for all pseudoalgebras $A, B$.
\item The 2-cell $\dig^2_{A,B}$ in \eqref{equ:comultiplication-monoidal-2} is a pseudocomonoid 2-cell for all pseudoalgebras $A, B$.
\item The 2-cell $\dig^0$ in \eqref{equ:comultiplication-monoidal-0}  is a  pseudocomonoid 2-cell.
\end{enumerate}
\end{corollary}

\begin{proof} By \cref{thm:to-be-added-1} and part~(ii) of \cref{thm:coalg-mor-comon-mor}.
\end{proof}

Next, we establish that contraction and weakening are braided monoidal pseudonatural transformations.

  \begin{proposition}
  \label{thm:contraction-is-monoidal} 
  The contraction pseudonatural transformation, with components on objects $\con_A \co \bang A \to \bang A \otimes \bang A$, is  braided monoidal, \ie we have invertible modifications with components
 \[
 \begin{gathered}
\begin{tikzcd}[column sep = 1.7cm]
 \bang A \otimes \bang B 
 	\ar[rr, "\monn_{A, B}"] 
	 \ar[d, "\con_A \otimes \con_B"'] 
	 \ar[drr, phantom, description, "\Two \con^2_{A,B}"]  & 
	 	& 
\bang (A \otimes B) 
	\ar[d, "\con_{A \otimes B}"] \\
 \bang A \otimes \bang A \otimes \bang B \otimes \bang B 
 	\ar[r, "\id_{\bang A} \otimes \bra_{\bang A, \bang B} \otimes \id_{\bang B}"']  & 
 \bang A \otimes \bang B \otimes \bang A \otimes \bang B 
 	\ar[r, "\monn_{A,B} \otimes \monn_{A,B}"'] & 
 \bang (A \otimes B) \otimes \bang (A \otimes B) ,
 \end{tikzcd} \\[1ex]
 \begin{tikzcd}[column sep = large]
 \unit 
 	\ar[r, "\moni"]  
 	\ar[d, "\id"] 
	\ar[dr, phantom, description, "\Two \con^0"] & 
\bang \unit\ar[d, "\con_{\unit}"]  \\
 \unit \otimes \unit  
 	\ar[r, "\moni \otimes \moni"'] & 
\bang \unit \otimes \bang \unit ,
\end{tikzcd}
\end{gathered}
 \] 
such that the appropriate coherence conditions hold.
\end{proposition} 

\begin{proof} By definition, $\con$ is obtained by whiskering the sylleptic lax monoidal pseudofunctor $\bang$ and the braided monoidal natural transformation $\com$ and therefore it is a braided monoidal natural transformation.
Explicitly, the 2-cells $\con^2_{A,B}$ are part of the data making $\monn_{A,B}$ into a comonoid morphism in part~(i) of \cref{thm:mu-comonoid-morphism} and the 2-cell $\con^0$ is part
of the data making $\moni$ into a comonoid morphism in part~(ii) of \cref{thm:mu-comonoid-morphism}.
\end{proof}

 \begin{proposition}
 \label{thm:weakening-is-monoidal}
 The weakening pseudonatural transformation, with components on objects $\weak_A \co \bang A \to \unit$ is 
braided monoidal, \ie we have invertible modifications with components
 \[
\begin{tikzcd}
 \bang A \otimes \bang B 
 	\ar[r, "\monn_{A,B}"] 
 	\ar[d, "\weak_A \otimes \weak_B"'] 
	\ar[dr, phantom, description, "\Two \weak^2_{A,B}"] & 
\bang (A \otimes B) 
	\ar[d, "\weak_{A \otimes B}"] \\
 \unit \otimes \unit 
 	\ar[r, "\id_\unit"'] &
\unit  ,
\end{tikzcd} \qquad
 \begin{tikzcd} 
 \unit 
 	\ar[d, "\id_\unit"']
	\ar[r, "\moni"] 
	\ar[dr, phantom, description, "\Two \weak^0"]& 
 \bang \unit 
 	\ar[d, "\weak_I"] \\
\unit 
	\ar[r, "\id_\unit"']  &  
\unit  ,
\end{tikzcd} 
 \]
 such that the appropriate coherence conditions hold.
 \end{proposition}
 
 \begin{proof} By definition, the pseudonatural transformation $\weak$ is obtained by whiskering the sylleptic lax monoidal pseudofunctor $\bang$ and the braided monoidal pseudonatural transformation $e$. Therefore, it is 
 braided monoidal. The 2-cells can also be given explicitly using \cref{thm:mu-comonoid-morphism}.
 \end{proof}
 
 \begin{remark} \label{thm:compare-with-jacq} Our axioms for a linear exponential pseudocomonad imply those in \cite[Definition~90]{JacqC:catcnis}.
Indeed, for each object $A \in \catK$, the object $\bang A$ is a symmetric pseudocomonoid, the maps 
$\con_A \co \bang A \to \bang A \otimes \bang A$ and $\weak_A \co \bang A \to \unit$ are pseudocoalgebra morphisms by \cref{thm:con-and-weak-coalgebra-maps}, 
and $\dig_A \co \bang A \to \bbang A$ is a morphism of braided pseudocomonoids by \cref{thm:dig-comonoid-morphism}. We have not checked yet whether the
converse implication holds.
\end{remark}

\subsection*{Cofree symmetric pseudocomonoids and linear exponential pseudocomonads} We now provide a method to construct linear exponential pseudocomonads
in the sense of \cref{thm:linear-exponential-comonad}, which we will use in our application in \cref{sec:prof}. We begin by considering the setting of a symmetric Gray monoid~$\catK$ and assume that it 
admits the construction of cofree symmetric pseudocomonoids, in the sense that the forgetful 2-functor $U \co \SymCoMon{\catK} \to \catK$ has a right biadjoint~$F$, mapping
an object~$A \in \catK$ to the cofree symmetric pseudocomonoid on it, as in
\begin{equation}
\label{equ:cofree-ccmon}
\begin{tikzcd}[column sep = large]
\catK
 \ar[r, shift left =2, "F"] 
  \ar[r, description, phantom, "\scriptstyle \top"] 
	&  
\SymCoMon{\catK}   .
	\ar[l, shift left = 2, "U"] 
\end{tikzcd}
\end{equation}
Assuming further that the biadjunction in \eqref{equ:cofree-ccmon} is comonadic, in the sense that the canonical comparison 2-functor $K$ in 
\[
\begin{tikzcd}[column sep = {2cm,between origins}]
\SymCoMon{\catK} \ar[rr, pos = (.4), "K"] \ar[dr, "U"'] & & \Em{\catK} \ar[dl, "U"] \\
 & \catK & 
 \end{tikzcd}
 \]
 is a biequivalence, we obtain a linear exponential pseudocomonad, as stated in the next result.

\begin{theorem} 
\label{thm:case-1}
Let $\catK$ be a symmetric Gray monoid. Assume that~$\catK$ admits the construction of cofree symmetric pseudocomonoids and that $\SymCoMon{\catK}$
is comonadic over $\catK$. Then the pseudocomonad on~$\catK$ determined by the forgetful-cofree biadjunction  in \eqref{equ:cofree-ccmon} is a linear exponential pseudocomonad.
\end{theorem} 

\begin{proof} By \cref{thm:ccmon-finprod}, the symmetric
monoidal structure of $\catK$ lifts to $\SymCoMon{\catK}$ in such a way that the forgetful 2-functor $U \co \SymCoMon{\catK} \to \catK$ is  a strict sylleptic monoidal 2-functor.
Therefore its right biadjoint is automatically sylleptic lax monoidal by~\cite[Proposition~15]{DayStreet} and we obtain a symmetic lax monoidal pseudocomonad $\bang (-) \co \catK \to \catK$ by 
a 2-categorical version of doctrinal adjunction~\cite[pages~62-63]{GarnerR:enrcfc}.
By comonadicity, $\Em{\catK}$ inherits finite products from $\SymCoMon{\catK}$. Thus, every pseudocoalgebra admits a canonical symmetric pseudocomonoid structure and the required 
pseudonaturality axioms hold by \cref{thm:prod-iff-all-ccmon}.
\end{proof}

Our applications in \cref{sec:prof} involve a special case of \cref{thm:case-1} which we isolate for later reference. Let us now consider a symmetric Gray monoid $\catK$ that is compact closed, in the sense of~\cite{StayM:comcb}, and assume that it admits the construction of free symmetric pseudomonoids, in the sense that the forgetful 2-functor $U \co \SymCoMon{\catK} \to \catK$ has a left biadjoint $F$, as in
\begin{equation}
\label{equ:free-mon}
\begin{tikzcd}[column sep = large]
 \catK 
 	\ar[r, shift left =2,  "F"] 
 	\ar[r, description, phantom, "\scriptstyle \bot"] 
	&  
\SymCoMon{\catK} . 
	\ar[l,shift left = 2,  "U"] 
\end{tikzcd}
\end{equation}
Assuming further that the biadjunction in~\eqref{equ:free-mon} is monadic, we obtain again a linear exponential pseudocomonad, as stated in the final result of this section.

\begin{theorem} 
\label{thm:case-2}
Let $\catK$ be a compact closed symmetric Gray monoid. Assume that $\catK$ admits the construction of free symmetric pseudomonoids and that $\SymCoMon{\catK}$
is monadic over $\catK$, and let $\wn(-) \co \catK~\to~\catK$ be the pseudomonad associated to the forgetful-free biadjunction in~\eqref{equ:free-mon}.
Then the pseudocomonad $\bang(-) \co \catK \to \catK$ determined by duality from $\wn(-) \co \catK \to \catK$ is a linear exponential pseudocomonad.
\end{theorem}

\begin{proof} By the dual of \cref{thm:case-1}, the pseudomonad $\wn(-) \co \catK \to \catK$ is what we may call a linear exponential pseudomonad, \ie the dual notion of that of
a linear exponential pseudocomonad.

We can then exploit the duality that is available in compact closed bicategories to turn this linear exponential pseudomonad $\wn(-) \co \catK \to \catK$ into a linear exponential 
pseudocomonad $\oc (-) \co \catK \to \catK$, as desired. On objects, for example, we define 
\begin{equation}
\label{equ:wn-to-bang}
\bang A \defeq \big( \wn (A^\bot) \big)^\bot ,
\end{equation}
for $A \in \catK$.
\end{proof}

The linear exponential comonad constructed in \cref{thm:case-2} can be seen as
the counterpart of the notion of a bi-exponential in \cite[Section~2.2]{HylandM:gluoml}.
Note that the application of \cref{thm:case-1} and \cref{thm:case-2} requires to check the appropriate comonadicity or monadicity assumptions. For this, one can appeal 
to the existing  monadicity results for pseudomonads~\cite{CreurerI:bectpm,HermidaC:desfsr} or, as we shall do in \cref{sec:prof}, proceed by direct inspection.

\section{Products and the Seely equivalences} 
\label{sec:products}

We continue to work with our fixed symmetric Gray monoid $\catK$, but now we assume also that $\catK$ has finite products (in a bicategorical sense, following the convention set in \cref{sec:prelim}).
Recall that we write $A \with B$ for the binary product of $A, B \in \catK$, with projections $\pi_1 \co A \with B \to A$ and $\pi_2 \co A \with B \to B$,
and $\term$ for the terminal object.  Let us also fix a linear exponential pseudocomonad $(\bang, \dig, \der)$ on $\catK$ as in~\cref{thm:linear-exponential-comonad}.
The next theorem relates the cartesian and monoidal structures on~$\catK$. Its proof was suggested by the observation in~\cite{GarnerR:hyplem}
that  the Seely equivalences arise in a canonical way.

\begin{theorem}  \label{thm:seely-equivalences-monoidal}
The 2-functor $\bang(\arghole) \co \catK \to \catK$ admits the structure of a sylleptic strong monoidal 2-functor 
\[
(\bang, \seell, \seeli)  \co (\catK, \with, \term)   \to  (\catK, \otimes, \unit) . 
\]
\end{theorem}

\begin{proof} Since the pseudocomonad is lax monoidal, the 2-category of pseudocoalgebras $\Em{\catK}$ admits a symmetric monoidal structure by~\cref{thm:coalg-is-monoidal}. Since the pseudocomonad is a linear exponential pseudocomonad, the monoidal structure on $\Em{\catK}$ is cartesian by~\cref{thm:coalgebras-is-cartesian}. Therefore the product  of two cofree pseudocoalgebras is their tensor 
product, with its evident induced pseudocoalgebra structure.

Let us now consider the biadjunction~in~\eqref{equ:em-adjunction} between $\catK$ and $\Em{\catK}$, where the left biadjoint is the forgetful 2-functor and the right biadjoint $F$ sends an object of $\catK$ to the cofree pseudocoalgebra on it. Since right biadjoints preserve products, the evident maps
\begin{equation}
\label{equ:seely-key}
f_{A,B} \co F(A \with B) \to FA \otimes FB , \qquad f \co F(\term) \to \unit ,
\end{equation} 
are equivalences. Since $F \co (\catK, \with, \term) \to (\Em{\catK}, \otimes, \unit)$ preserves finite products, it is a sylleptic strong monoidal pseudofunctor
between the cartesian monoidal structures, \cf \cite{CarboniA:carbii}. The forgetful pseudofunctor $U$ is a sylleptic strict monoidal pseudofunctor $U \co (\Em{\catK}, \otimes, \unit) \to (\catK, \otimes, \unit)$ 
by~\cref{thm:coalg-is-monoidal} and therefore the composite, 
\[
\begin{tikzcd}
(\catK, \with, \top) \ar[r, "F"] &
(\Em{\catK}, \otimes, \unit) \ar[r, "U"] &
(\catK, \otimes, \unit) ,
\end{tikzcd}
\]
 is a sylleptic strong monoidal pseudofunctor. But this is exactly the underlying pseudofunctor of the pseudocomonad, as desired.
\end{proof}

We can unfold explicitly the definition of the constraints of the sylleptic strong monoidal 2-functor of \cref{thm:seely-equivalences-monoidal}, which we call 
\emph{Seely equivalences}, since they are the 2-categorical counterparts of the Seely 
isomorphisms~\cite{Seely}. 
These have the form
\begin{equation}
\label{equ:seely-maps}
\seell_{A, B} \co \bang A \otimes \bang B \to \bang (A \with B) , \qquad \seeli \co I \to \bang \top .
\end{equation}
The map $\seell_{A,B}$ is the transpose of 
\[
\begin{tikzcd}
\bang A \otimes \bang B \ar[r, "(\der_A \otimes \weak_B {,} \weak_A \otimes \der_B)"]  &[5em] 
(A \otimes \unit) \with (\unit \otimes B) \ar[r, "\id"] &
A \with B 
\end{tikzcd}
\]
across the biadjunction in~\eqref{equ:em-adjunction}, \ie the image of this map under the adjoint equivalence 
\[
\catK[\bang A \otimes \bang B, A \with B] \simeq \Em{\catK}[ \bang A \otimes \bang B,  \bang (A \with B)] .
\]
Thus, it is the essentially unique\footnote{{Cf.}~comments below~\eqref{equ:aux-equiv}.}~pseudocoalgebra morphism from $\bang A \otimes \bang B$ to $\bang (A \with B)$ with an invertible 2-cell
\[
\begin{tikzcd}[column sep = large]
\bang A \otimes \bang B \ar[r, "\seell_{A,B}"] \ar[dr, bend right = 20,  {pos=.4}, "(\der_A \otimes \weak_B {,} \weak_A \otimes \der_B)"'] & \bang (A \with B) 
\ar[d, "\der_{A \with B}"] \\
\phantom{} \ar[ur, phantom, {pos=.7},  "\Two"] 	& A \with B .
\end{tikzcd}
\]
Using the definition in~\eqref{equ:unfold-f-sharp}, we obtain that $\seell_{A,B}$ is the composite
 \[
 \begin{tikzcd}[scale=0.3]
 \bang A \otimes \bang B
 	\ar[r, "\dig_A \otimes \dig_B"] &[2em]
\bbang A \otimes \bbang B
	\ar[r, "\monn_{\bang A, \bang B}"] &
\bang (\bang A \otimes \bang B)
	\ar[d, "\bang (\der_A \otimes \weak_B {,} \weak_A \otimes \der_B)"] \\ 
	& &  \bang ( (A \otimes \unit) \with (\unit \otimes B) ) \ar[d, "\id"]  \\
	& &  \bang (A \with B) .
\end{tikzcd}
\]
Furthermore, the adjoint quasi-inverse  ${\seell}^\bullet_{A,B}$ of $\seell_{A,B}$ (\cf the notation introduced at the start of \cref{sec:prelim}) is the map $f_{A,B}$ in~\eqref{equ:seely-key}, \ie 
\[
\begin{tikzcd}[column sep = large] 
\bang (A \with B) \ar[r, "\con_{A \with B}"] &
\bang (A \with B) \otimes \bang (A \with B) \ar[r, "\bang \pi_1 \otimes \bang \pi_2"] & 
\bang A \otimes \bang B .
\end{tikzcd} 
\]
Similarly, $\seeli \co \unit \to \bang \term$ is the essentially unique pseudocoalgebra morphism from $I$ to $\bang \term$ with an invertible 2-cell
\[
\begin{tikzcd}[column sep = large] 
\unit \ar[r, "\seeli"] \ar[dr, bend right = 20] & \bang \top
\ar[d, "\der_\top "] \\
\phantom{} \ar[ur, phantom, {pos=.7},  "\Two"] 	& \top .
\end{tikzcd}
\]
Thus, it is the composite
\[
\begin{tikzcd}
\unit 
	\ar[r, "\moni"] &
\bang \unit 
	\ar[r] &
\bang \term .
\end{tikzcd}
\]
and its adjoint quasi-inverse ${\seel}^\bullet$ is  $\weak_\term \co \bang \term \to \unit$.

\begin{corollary}  \label{thm:lift-bang-to-comonoids}  \leavevmode
The 2-functor $\bang(-) \co \catK \to \catK$ lifts to 2-categories of symmetric pseudocomonoids, as in 
\[
\begin{tikzcd} 
\SymCoMon{\catK, \with, \term} \ar[r, "\bang"] \ar[d] & \SymCoMon{\catK, \otimes, \unit} \ar[d] \\
\catK \ar[r, "\bang"'] & 
\catK .
\end{tikzcd}
\]
\end{corollary}

\begin{proof} The 2-functor is sylleptic strong monoidal by \cref{thm:seely-equivalences-monoidal}
and therefore it lifts to 2-categories of symmetric pseudocomonoids by \cite[Proposition~16]{DayStreet}.
\end{proof}

We turn to study the lifted 2-functor of \cref{thm:lift-bang-to-comonoids}. Our goal is to show  that it preserves finite products, which will be achieved in \cref{thm:lift-preserves-products}.
Let us begin with a couple of lemmas.

\begin{lemma} \leavevmode \label{thm:coh-digging-with-m}
\begin{enumerate}[(i)] 
\item For every $A, B \in \catK$, the map $\seell_{A,B} \co \bang A \otimes \bang B \to \bang(A \with B)$ is a pseudocoalgebra morphism, \ie we have an invertible 2-cell
\[
\begin{tikzcd}[column sep = large]
\bang A \otimes \bang B 
	\ar[r, "\seell_{A,B}"]
	\ar[d, "\dig_A \otimes \dig_B"'] 
	\ar[ddr, phantom, description, "\Two \bar{\seel}^2_{A,B}"] &
\bang (A \with B) 
	\ar[dd, "\dig_{A \with B}"] \\
\bbang A \otimes \bbang B 
	\ar[d, "\monn_{\bang A, \bang B}"'] &
	\\
\bang (\bang A \otimes \bang B)
	\ar[r, "\bang \seell_{A,B}"'] &
\bbang (A \with B) 
\end{tikzcd}
\]
which satisfies the appropriate coherence conditions.
\item The map $\seeli \co \unit \to \bang \term$ is a pseudocoalgebra morphism,  \ie we have an invertible 2-cell
\[
\begin{tikzcd} 
\unit 
	\ar[r, "\seeli"] 
	\ar[d, "\moni"'] 
	\ar[dr, phantom, description, "\Two \bar{\seel}^0"] & 
\bang \term 
	\ar[d, "\dig_{\term}"] \\
\bang \unit 
	\ar[r, "\bang \seeli"'] & 
\bbang \term 
\end{tikzcd} 
\]
which satisfies the appropriate coherence conditions.
\end{enumerate}
\end{lemma}

\begin{proof}  The Seely equivalences are, by construction, pseudocoalgebra morphisms. However, we can also provide an explicit definition of the required invertible 2-cells. 
For part~(i), this is constructed as follows:
\[
\begin{tikzcd}[column sep = large]
\bang A \otimes \bang B
	\ar[r, "\dig_A \otimes \dig_B"]
	\ar[d, "\dig_A \otimes \dig_B"'] 
	\ar[dr, phantom, description, "\Two \pi_A \otimes \pi_B"] &
\bbang A \otimes \bbang B
	\ar[r, "\monn_{\bang A, \bang B}"] 
	\ar[d, "\dig_{\bang A} \otimes \dig_{\bang B}"] 
	\ar[ddr, phantom, description, "\Two \dig^2_{\bang A, \bang B}"] &
\bang  ( \bang A \otimes \bang B)
	\ar[rr, " \bang (\der_A \otimes \weak_B {,} \weak_A \otimes \der_B) "]
	\ar[dd, "\dig_{\bang A \otimes \bang B}"] 
	\ar[ddrr, phantom, description, "\cong"]&
 	&
\bang (A \with B)
	\ar[dd, "\dig_{A \with B}"] \\
\bbang A \otimes \bbang B
	\ar[r, "\bang \dig_A \otimes \bang \dig_B"]
	\ar[d,  "\monn_{\bang A, \bang B}"'] 
	\ar[dr, phantom, description, "\cong"] &
\bbbang A \otimes \bbbang B		
	\ar[d, "\monn_{\bbang A, \bbang B}"] &
		& 
		\\
\bang (\bang A \otimes \bang B)
	\ar[r, "\bang (\bang \dig_A \otimes \bang \dig_B)"'] &
\bang (\bbang A \otimes \bbang B)
	\ar[r, "\bang \monn_{\bang A, \bang B}"'] &
\bbang (\bang A \otimes \bang B)
	\ar[rr, "\bbang (\der_A \otimes \weak_B {,} \weak_A \otimes \der_B)"'] &
	&
\bbang (A \with B) .
\end{tikzcd}
\]
For part~(ii), the required invertible 2-cells is given by
\[
\begin{tikzcd}[column sep = large]
\unit \ar[r, "\moni"] \ar[d, "\moni"'] 
\ar[dr, phantom, description, "\Two \dig^0"]
& \bang \unit \ar[d, "\dig_\unit"] \ar[r] \ar[dr, phantom, description, "\cong"]& \bang \term \ar[d, "\dig_\term"] \\
\bang \unit \ar[r, "\bang \moni"'] & \bbang \unit \ar[r] & \bbang \term .
\end{tikzcd}
\]
Here, $\dig^0$ is the 2-cell that is part of the data making $\mathsf{p}$ into a symmetric monoidal pseudocomonad, as
in part~(ii) of \cref{thm:sym-lax-monoidal-pseudocomonad}, and the other 2-cell is part of the pseudonaturality of $\dig$.
\end{proof}

\begin{lemma} \label{thm:seely-equiv-comonoid} \leavevmode
\begin{enumerate}
\item For every $A, B \in \catK$, the Seely equivalence $\seell_{A,B}  \co \bang A \otimes \bang B \to \bang(A \with B)  $ is a braided pseudocomonoid morphism.
\item The Seely equivalence $\seeli \co \unit \to \bang \term$ is a braided pseudocomonoid morphism.
\end{enumerate}
\end{lemma} 

\begin{proof} Both claims follow by combining \cref{thm:id-is-comonoid-morphism}, \cref{thm:coalg-mor-comon-mor}, and \cref{thm:coh-digging-with-m}.
\end{proof}

\begin{proposition} \label{thm:lift-preserves-products}
The lifted 2-functor $\bang (-) \co \SymCoMon{\catK, \with, \term} \to   \SymCoMon{\catK, \otimes, \unit}$ preserves
finite products.
\end{proposition} 

\begin{proof}
Recall that finite products in 2-categories of  symmetric pseudocomonoids are given by the tensor product of their underlying objects. Therefore, we need to show that if~$A$ and~$B$
are symmetric pseudocomonoids with respect to $\with$, then the symmetric pseudocomonoid~$\bang (A \with B)$ is equivalent to the symmetric pseudocomonoid $\bang A \otimes \bang B$
in~$\SymCoMon{\catK}$ and that the symmetric pseudocomonoid $\bang \term$ is equivalent to the symmetric pseudocomonoid $\unit$.
 These claims follow from  \cref{thm:seely-equiv-comonoid}.
\end{proof}

The next proposition, which will be needed for \cref{thm:bang-A-bialgebra}, describes the action of the lifted pseudofunctor of \cref{thm:lift-bang-to-comonoids} in terms of the linear exponential pseudocomonad. 
For this, recall from \cref{thm:prod-iff-all-ccmon} that  the finite products of $\catK$ determine a symmetric pseudocomonoid  structure on every object of $\catK$.

\begin{proposition} \label{thm:comm-comonoids-coincide}
Let $A \in \catK$. The following are equivalent as  symmetric pseudocomonoids in $(\catK, \otimes, \unit)$:
\begin{enumerate}
\item the symmetric pseudocomonoid  on $\bang A$ obtained by 
applying the sylleptic strong monoidal 2-functor of \cref{thm:seely-equivalences-monoidal} to the symmetric pseudocomonoid on~$A$ determined by the finite products in~$\catK$,
\item  the symmetric pseudocomonoid on $\bang A$ determined by the linear exponential pseudocomonad. 
\end{enumerate}
Specifically, 
the identity is a braided pseudocomonoid morphism between them, \ie we have invertible 2-cells
\[
\begin{tikzcd}
\bang A
	\ar[r, "\id"]
	\ar[d, "\bang \Delta_A"'] 
	\ar[ddr, phantom, description, "\Two"]
	 &
\bang A
	\ar[dd, "\con_A"]
	 \\
\bang (A \with A)
	\ar[d, "(\seell_{A,A})^{\bullet}"'] 
	\\
\bang A \otimes \bang A
	\ar[r, "\id"'] &
\bang A \otimes \bang A 
\end{tikzcd} \qquad
\begin{tikzcd}
\bang A
	\ar[r, "\id"]
	\ar[d] 
	\ar[ddr, phantom, description, "\Two"] &
\bang A 
	\ar[dd, "\weak_A"] \\
\bang \term
	\ar[d, "(\seeli)^\bullet"'] & 
	\\
\unit
	\ar[r, "\id"'] &
\unit  ,
\end{tikzcd}
\]
satisfying the appropriate coherence conditions.
\end{proposition}

\begin{proof} Since the cofree pseudocoalgebra 2-functor $F \co \catK \to \Em{\catK}$ is a right biadjoint, it preserves finite products. Thus, the claim follows from~\cref{thm:canonical-preserved} 
via \cref{thm:canonical-is-linear-exponential} and the definition of $\con_A$ and $\weak_A$. Explicitly, the required 2-cell in the  left-hand diagram can be obtained as the pasting diagram
\[
\begin{tikzcd}[column sep = huge]
\bang A 
	\ar[r, "\bang \Delta_A"]
	\ar[d, "\con_A"'] 
	\ar[dr, phantom, description, "\Two \bar{\con}_{\Delta_A}"] &[3em]
\bang (A \with A)
	\ar[d, "\con_{A \with A}"] \\
\bang A \otimes \bang A
	\ar[r, "\bang \Delta_A \otimes \bang \Delta_A"] 
	\ar[dr, bend right = 10, "\id_{\bang A \otimes \bang A}"']&
\bang (A \with A) \otimes \bang (A \with A) 
	\ar[d, "\bang \pi_1 \otimes \bang \pi_2"] \\ 
\phantom{} \ar[ur, phantom, {pos=.7},  "\Two"] 	& 
\bang A \otimes \bang A  , 
\end{tikzcd}
\]
where the invertible 2-cell in the rectangle is given by pseudonaturality of $\con$ and the 
invertible 2-cell
in the triangle is obtained by the cartesian structure. For the right-hand diagram, 
unfolding the definitions, we need an invertible 2-cell
\[
\begin{tikzcd}
\bang A
	\ar[r] 
	\ar[d, "\weak_A"'] 
	\ar[dr, phantom, description, "\Two"]&
\bang \term
	\ar[d, "\weak_\term"] \\
\unit 
	\ar[r, "\id_\unit"'] &
\unit ,
\end{tikzcd} 
\]
which  is simply given by the pseudonaturality of $\weak$. 
\end{proof}

The final result of this section, \cref{thm:coKleisli-cartesian-closed} below,  is the two-dimensional counterpart of a fundamental
theorem on linear exponential comonads on symmetric monoidal categories with finite 
products~\cite{MelliesPA:catsll}. In its statement, by a closed symmetric Gray monoid $\catK$, we mean
that, for every $A \in \catK$, the 2-functor $(-) \otimes A \co \catK \to \catK$ has a right
biadjoint. The action of such a right biadjoint on $B \in \catK$ will be written $A \linhom B$,
using again notation inspired by Linear Logic.

\begin{theorem} \label{thm:coKleisli-cartesian-closed} Assume that $\catK$ is closed and has finite products.
Then the Kleisli bicategory~$\Kl{\catK}$ is a cartesian closed bicategory.
\end{theorem}

\begin{proof} The existence of finite products in $\Kl{\catK}$ was recalled in \cref{thm:coKleisli-cartesian}. For the closed structure, 
the claim follows  by~\cite[Theorem~14]{JacqC:catcnis} since the axioms for a linear exponential pseudocomonad in \cite[Definition~90]{JacqC:catcnis}  are consequences
of the axioms for a linear exponential pseudocomonad used here, as noted in \cref{thm:compare-with-jacq}. Explicitly, 
the exponential of $A, B$ in $\Kl \catK$ is defined by $A \Rightarrow B = \bang A \multimap B$. Indeed, we have the following chain of equivalences:
\begin{align*}
\Kl \catK[ X \with A, B] & = \catK[ \bang ( X \with A), B ] \\
 & \simeq \catK[ \bang X \otimes \bang A, B] \\
 & \simeq \catK[ \bang X,  \bang A \multimap B ] \\
 & = \Kl \catK[ X, A \Rightarrow B] ,
 \end{align*}
 which can be proved to be pseudonatural.
\end{proof}

\begin{remark} \label{thm:compact-closed-hom}
 In preparation for our applications in \cref{sec:prof}, we consider the special case where $\catK$ is compact closed, in the sense of~\cite{StayM:comcb}. We write  $A^\perp$ for the dual of an object $A \in \catK$. The internal hom of the closed structure is defined by letting $A \multimap B \defeq A^\perp \otimes B$. Therefore, the exponential objects in $\Kl \catK$ are given
by $A \Rightarrow B \defeq (\bang A)^\perp \otimes B$. 
\end{remark}

\begin{remark}  \label{thm:mellies-2-cell} We conjecture that the Kleisli biadjunction in~\eqref{equ:cokleisliadjunction} becomes a symmetric lax monoidal biadjunction
when $\catK$ is considered as a symmetric Gray monoid with respect to $\otimes$, and $\Kl{\catK}$ is considered as a symmetric monoidal bicategory with respect
to $\with$, as in 
\begin{equation*}
\begin{tikzcd}[column sep = large]
 (\Kl{\catK}, \with, \term) 
	\ar[r, shift left =1, bend left =10, "K"]  
	\ar[r, description, phantom, "\scriptstyle \bot"] &  
(\catK, \otimes, \unit) .
	\ar[l, shift left = 1, bend left =10, "J"] 
\end{tikzcd}
\end{equation*}
By the 2-categorical counterpart of Kelly's doctrinal adjunction theorem (see ~\cite[pages~62-63]{GarnerR:enrcfc} and~\cite[Propositions~2, 12 and~15]{DayStreet}), it is sufficient to show that 
the left biadjoint $K$ is symmetric strong monoidal. This involves showing that the Seely equivalences $\seell$ are pseudonatural with respect to Kleisli maps. The candidate pseudonaturality 2-cells can be constructed as in \cite[Section~7.3.0.6]{MelliesPA:catsll}, making
essential use of invertible 2-cells of the form
 \begin{equation}
 \label{equ:mellies-2-cell}
\begin{tikzcd}[column sep = large]
\bang A \otimes \bang B 
	\ar[r, "\seell_{A,B}"]  
	\ar[dd, "\dig_A \otimes \dig_B"'] 
	\ar[ddr, phantom, description, "\Two \sigma_{A, B}"] 
	& 
\bang (A \with B) 
	\ar[d, "\dig_{A \with B}"] 
	\\ 
 	& 
\bbang (A \with B) 
	\ar[d, "\bang {(\bang \pi_1, \bang \pi_2)}"]
	\\
 \bbang A \otimes \bbang B 
 	\ar[r, "\seell_{\bang A, \bang B}"'] 
  & 
\bang (\bang A \with \bang B) .
\end{tikzcd}
\end{equation}
In turn, these can be constructed using part~(i) of \cref{thm:coh-digging-with-m}.
As yet we have not fully verified the pseudonaturality conditions.
\end{remark} 


\section{Biproducts, cocontraction, and coweakening}
\label{sec:biproducts}

Let $\catK$ be a symmetric Gray monoid. We now assume that $\catK$ has biproducts, in the sense of \cref{thm:biproducts}. Recall that we write $A \oplus B$ for the 
biproduct of $A, B \in \catK$ and $0$ for the zero object.  As a consequence of having biproducts, $\catK$  admits a form of enrichment 
over the category of symmetric monoidal categories. In particular, for every $A, B \in \catK$, 
the hom-category $\catK[A, B]$ admits a symmetric monoidal structure, which we call
\emph{convolution},  in analogy with the one-dimensional case. 
For $f \co A \to B$ and $g \co A \to B$, we define their convolution $f + g \co A \to B$ as the composite
\begin{equation}
\label{equ:convolution}
\begin{tikzcd}
A 
	\ar[r, "\Delta_A"] & 
A \oplus A 
	\ar[r, "f \oplus g"] &
B \oplus B 
	\ar[r, "\nabla_B"] &
B .
\end{tikzcd}
\end{equation}
The unit of the symmetric monoidal structure is the \emph{zero morphism} $0_{A,B} \co A \to B$
defined as the composite 
\begin{equation}
\label{equ:zero-map}
\begin{tikzcd}
A 
	\ar[r] & 
0
 	\ar[r] &
B .
\end{tikzcd} 
\end{equation}

Let us now assume again that we have a linear exponential pseudocomonad $(\bang, \dig, \der)$ on $\catK$. As an instance of \cref{thm:seely-equivalences-monoidal}, the underlying pseudofunctor of the linear exponential pseudocomonad
acquires the structure of a sylleptic strong monoidal functor $\bang (-) \co (\catK, \oplus, 0) \to (\catK, \otimes, \unit)$ and the Seely equivalences have the form
\begin{equation}
\label{equ:seely-with-biproducts}
\seel^{2}_{A,B} \co \bang A \otimes \bang B \to \bang (A \oplus B)  , \qquad
\seeli \co \unit \to \bang 0 .
\end{equation}

In categorical models of linear logic where the product is a coproduct, the Seely isomorphisms
can be used to define cocontraction and coweakening maps. Here we do the same using the
Seely equivalences. For $A \in \catK$, we define the cocontraction map $\cocon_A$ by letting
\begin{equation}
\label{equ:def-cocon}
\begin{tikzcd}[ampersand replacement=\&] 
\bang A \otimes \bang A 
	\ar[r, "\cocon_A"] \& 
\bang A 
\end{tikzcd}
\quad \defeq \quad 
 \begin{tikzcd}[ampersand replacement=\&] 
 \bang A \otimes \bang A 
 	\ar[r, "\seel_{A,A}"] \& 
	\bang (A \oplus A) \ar[r, "\bang \nabla_A"] \& 
	\bang A ,
 \end{tikzcd}
 \end{equation}
 and the coweakening map $\coweak_A$ by letting
 \begin{equation}
 \label{equ:def-coweak}
 \begin{tikzcd}[ampersand replacement=\&] 
\unit \ar[r, "\coweak_A"] \& \bang A 
\end{tikzcd}
\quad \defeq \quad
 \begin{tikzcd}[ampersand replacement=\&]
 \unit \ar[r, "\seeli"] \&
 \bang 0 \ar[r] \&
 \bang A .
 \end{tikzcd}
\end{equation}

For the next proposition, recall the definition of a symmetric pseudobialgebra from \cref{thm:bialgebra-def} and its equivalent rephrasing in \cref{thm:bialgebra-equivalent}.

\begin{proposition} \label{thm:bang-A-bialgebra} Let $A \in \catK$. The object $\bang A$ is not only a symmetric pseudocomonoid, with comultiplication
and counit given by contraction and weakening, 
\[
\con_A \co \bang A \to \bang A \otimes \bang A , \qquad
\weak_A \co \bang A \to \unit , 
\]
as implied by the axioms for a linear exponential pseudocomonad, 
but also a symmetric pseudomonoid, with multiplication and unit given by cocontraction and coweakening, 
\[
\cocon_A \co  \bang A \otimes \bang A \to \bang A ,  \qquad
\coweak_A \co \unit \to \bang A . 
\] 
Furthermore, the symmetric pseudocomonoid and the symmetric pseudomonoid of $\bang A$ are part of a symmetric pseudobialgebra structure, \ie we have additional invertible 2-cells
\[
\begin{tikzcd}[column sep = huge]
 \bang A \otimes \bang A 
 	\ar[r, "\con_A \otimes \con_A"] 
	\ar[d, "\cocon_A"'] 
	\ar[drr, phantom, description, "\Two"] &
 \bang A \otimes \bang A \otimes \bang A \otimes \bang A 
 	\ar[r, "\id \otimes \bra_{\bang A, \bang A} \otimes \id"] &
 \bang A \otimes \bang A \otimes \bang A \otimes \bang A 
 	\ar[d, "\cocon_A \otimes \cocon_A"] \\
 \bang A 
 	\ar[rr, "\con_A"']  & &
 \bang A \otimes \bang A ,
\end{tikzcd}
\]
 
\[
\begin{gathered}
\begin{tikzcd}
\bang A \otimes \bang A 
	\ar[d, "\weak_A  \otimes \weak_A"'] 
	\ar[r, "\cocon_A"] 
	\ar[dr, phantom, description, "\Two"] & 
\bang A 
	\ar[d, "\weak_A"] \\
 \unit \otimes \unit \ar[r, "\id"'] 
 	& 
\unit ,
\end{tikzcd}
 \qquad
\begin{tikzcd}[column sep = large] 
\unit 
	\ar[r, "\coweak_A"]
	\ar[d, "\id"'] 
	\ar[dr, phantom, description, "\Two"] & 
\bang A 
	\ar[d, "\con_A"] \\
\unit \otimes \unit
	\ar[r, "\coweak_A \otimes \coweak_A"']  & 
\bang A \otimes \bang A ,
\end{tikzcd}
 \qquad
\begin{tikzcd}[column sep = large]
\unit
	\ar[r, "\coweak_A"]
	\ar[dr, bend right = 20, "\id_\unit"'] & 
\bang A 
	\ar[d, "\weak_A"] \\
\phantom{} \ar[ur, phantom, {pos=.7},  "\Two"]   & \unit ,
\end{tikzcd}  
 \end{gathered}
\]
satisfying the appropriate coherence conditions.
\end{proposition}

\begin{proof} Let $A \in \catK$. Since $\catK$ has biproducts, $A$ admits a canonical structure of 
symmetric pseudobialgebra by \cref{thm:biprod-bialgebra}. As before, we write 
$\nabla_A \co A \oplus A \to A$ and $0 \to A$ for the symmetric pseudomonoid structure, $\Delta_A \co A \to A \oplus A$ and $A \to 0$ for the symmetric pseudocomonoid
structure. Since~$\bang(-) \co (\catK, \oplus, 0) \to (\catK, \otimes, \unit)$ is a sylleptic 
strong monoidal 2-functor by \cref{thm:seely-equivalences-monoidal}, it sends symmetric pseudobialgebras 
in~$(\catK, \oplus, 0)$ to symmetric pseudobialgebras in~$(\catK, \otimes, \unit)$. Explicitly, it sends $A$, with the symmetric pseudobialgebra structure just described, to $\bang A$,
with the symmetric pseudobialgebra structure defined as follows. The symmetric pseudomonoid structure on $\bang A$ is given exactly by the cocontraction and coweakening maps in the claim,
while the symmetric pseudocomonoid structure on $\bang A$ is given by 
\[
\begin{tikzcd}
\bang A 
	\ar[r, "\bang \Delta_A"] &
\bang (A \oplus A)
	\ar[r, "\seel^\bullet_{A,A}"] &
\bang A \otimes \bang A ,
\end{tikzcd} \qquad
\begin{tikzcd}
\bang A \ar[r, "\bang u"] &
\bang 0 \ar[r, "(\seeli)^\bullet"] &
\unit .
\end{tikzcd}
\]
This is equivalent to the symmetric pseudocomonoid structure given by contraction and weakening by \cref{thm:comm-comonoids-coincide}.
Therefore, we have the symmetric pseudobialgebra structure in the claim.
\end{proof}

\begin{proposition} \label{thm:cocontraction-coweakeneing-with-digging}
\leavevmode
\begin{enumerate}[(i)]
\item For every $A \in \catK$, the cocontraction map $\cocon_A \co \bang A \otimes \bang A \to \bang A$ is a pseudocoalgebra morphism,  \ie we have an invertible 2-cell
\[
\begin{tikzcd} 
\bang A \otimes \bang A 
	 \ar[r, "\cocon_A"] 
	 \ar[d, "\dig_A \otimes \dig_A"'] 
	 \ar[ddr, phantom, "\Two"] &
\bang A 
	\ar[dd, "\dig_A"] \\
\bbang A \otimes \bbang A 
	\ar[d, "\monn_{\bang A, \bang A}"'] &
	\\
  \bang (\bang A \otimes \bang A)
	\ar[r, "\bang \cocon_A"'] &
\bbang A
\end{tikzcd}
\]
satisfying  the appropriate  coherence conditions.
\item For every $A \in \catK$, the coweakening map $\coweak_A \co \unit \to \bang A$ is a pseudocoalgebra morphism, \ie we have an invertible 2-cell
\[
\begin{tikzcd} 
\unit
	\ar[r, "\coweak_A"] 
	\ar[d, "\moni"'] 
	 \ar[dr, phantom, "\Two"] & 
\bang A 
	\ar[d, "\dig_A"] \\
\bang \unit 
	\ar[r, "\bang \coweak_A"'] & 
\bbang A 
\end{tikzcd}
\]
satisfying  the appropriate coherence conditions.
\end{enumerate}
\end{proposition}

\begin{proof} For part~(i), the definition of $\cocon_A$ in \eqref{equ:def-cocon} exhibits $\cocon_A$ as the composite of two pseudocoalgebra morphisms, as $\seell_{A,A}$
is so by part~(i) of \cref{thm:coh-digging-with-m}. 
Similarly, part~(ii) follows by unfolding the definition of $\coweak_A$ in \eqref{equ:def-coweak} and then using part~(ii) of \cref{thm:coh-digging-with-m}.
\end{proof} 

The combination of \cref{thm:cocontraction-coweakeneing-with-digging} and \cref{thm:coalg-mor-comon-mor} gives another proof that $\cocon_A$ 
and~$\coweak_A$ are braided pseudocomonoid morphisms, as required
for a symmetric pseudobialgebra structure.

\cref{thm:cocontraction-coweakeneing-with-digging} above is the last statement in this section the proof of which involves the verification of coherence conditions. The statements after it will involve the existence of invertible 2-cells, but
we have not investigated what coherence conditions these satisfy. Of course, the construction of these 2-cells in the proofs provides some guidance as to what these may be.  
The first proposition in this series, \cref{thm:first-without-coh}, describes the interaction of the maps $\monn_{A, B} \co \bang A \otimes \bang B \to \bang (A \otimes B)$, for $A, B \in \catK$, 
that are part of the lax monoidal structure of the 2-functor $\bang (-) \co \catK  \to \catK$, with cocontraction  and coweakening.

\begin{proposition}  \label{thm:first-without-coh}
\leavevmode
\begin{enumerate}[(i)]
\item For every $A, B \in \catK$, we have an invertible 2-cell
\[
\begin{tikzcd}[column sep = huge]
\bang A \otimes \bang A \otimes \bang B 
	\ar[r, "\cocon_A \otimes \id"]
	\ar[d, "\id_{\bang A} \otimes \id_{\bang A} \otimes \con_B"'] 
	 \ar[dddr, phantom, description, "\Two"] & 
\bang A \otimes \bang B 
	\ar[ddd, "\monn_{A,B}"] \\
\bang A \otimes \bang A \otimes \bang B \otimes \bang B 
	\ar[d, "\id \otimes \bra_{\bang A, \bang B} \otimes \id"'] & \\
\bang A \otimes \bang B \otimes \bang A \otimes \bang B 
	\ar[d, "\monn_{A,B} \otimes \monn_{A, B}"'] & \\
\bang (A \otimes B) \otimes \bang (A \otimes B)
	 \ar[r, "\cocon_{A\otimes B}"'] & 
\bang (A \otimes B) .
\end{tikzcd}
\]
\item For every $A, B \in \catK$, we have an invertible 2-cell
\[
\begin{tikzcd}[column sep = large]
\unit \otimes \bang B 
	\ar[r, "\coweak_A \otimes \id"] 
	\ar[d, "\id \otimes \weak_B"']
	\ar[ddr, phantom, description, "\Two"] & 
\bang A \otimes \bang B 
	\ar[dd, "\monn_{A, B}"] \\
\unit \otimes \unit 
	\ar[d, "\id"] & \\
\unit 
	\ar[r, "\coweak_{A \otimes B}"'] & 
\bang ( A \otimes B)  .
\end{tikzcd}
\]
\end{enumerate}
\end{proposition} 

\begin{proof} For part~(i), we limit ourselves to outlining a strategy to construct the required 
invertible 2-cell. First, one precomposes both composite maps with the adjoint equivalence
\[
\begin{tikzcd}[column sep = huge]
\bang (A \oplus A) \otimes \bang B 
	\ar[r, "\seel^\bullet_{A,A} \otimes \id_{\bang B}"] &
\bang A \otimes \bang A \otimes \bang B .
\end{tikzcd}
\]
Secondly, one shows that both of the resulting maps are isomorphic to the composite
\[
\begin{tikzcd}
 \bang (A \oplus A) \otimes \bang B
 	\ar[r, "u"] & 
\bang \big( ( ( A \oplus A) \otimes B) \oplus ((A \oplus A) \otimes B) \big)
	\ar[r, "v"] & 
\bang (A \otimes B) ,
\end{tikzcd}
\]
where the maps $u$ and $v$ are 
\[
\begin{tikzcd}[column sep=large]
 \bang (A \oplus A) \otimes \bang B
 	\ar[r, "\monn_{A \oplus A, B}"] &
\bang \big( (A \oplus A) \otimes B \big)
	\ar[r, "\bang \Delta_{(A \oplus A) \otimes B}"] &[0.2cm]
\bang \Big( \big( ( A \oplus A) \otimes B \big) \oplus \big( (A \oplus A) \otimes B \big) \Big) ,
\end{tikzcd} 
\]
and
\[
\begin{tikzcd}
\bang \big( \big( ( A \oplus A) \otimes B \big) \oplus \big( (A \oplus A) \otimes B \big) \big)
	\ar[r, "\bang ( (\pi_1 \otimes \id) \oplus (\pi_2 \otimes \id) )"] &[4em]
\bang ( (A \otimes B) \oplus (A \otimes B) )
	\ar[r, "\bang \nabla_{A \otimes B}"] & 
\bang (A \otimes B) ,
\end{tikzcd}
\]
respectively. For part~(ii), it is not difficult to show that both composites are isomorphic to 
\[
\begin{tikzcd}
\unit \otimes \bang B 
	\ar[r, "\id"] &
\bang B
	\ar[r, "\bang 0"] &
\bang (A \otimes B) ,
\end{tikzcd}
\]
where $0$ is the zero map defined in~\cref{equ:zero-map}, by unfolding the definition of $\coweak_A$ and $\coweak_{A \otimes B}$.
\end{proof}

The next proposition describes the interaction of coweakening and cocontraction with the convolution structure defined
in \cref{equ:convolution}.

 \begin{lemma} For every $A \in \catK$, there are invertible 2-cells 
\[
\begin{tikzcd}
  \bang A \otimes \bang A  
  	\ar[r, "\cocon_A"]
	\ar[dr, bend right = 20, "\der_A \otimes \weak_A   + \weak_A \otimes \der_A"'] &   
  \bang A 
  	\ar[d, "\der_A"]  \\
\phantom{} \ar[ur, phantom, {pos=.7},  "\Two"]  & A  
 \end{tikzcd}  \qquad
 \begin{tikzcd}
 \unit 
 	\ar[r, "\coweak_A"] 
 	\ar[dr, bend right = 20, "0"'] &  
 \bang A 
 	\ar[d, "\der_A"]  \\
\phantom{} \ar[ur, phantom, {pos=.7},  "\Two"]   & A 
\end{tikzcd}
\]
\end{lemma}

\begin{proof} For the 2-cell on the left-hand side, we exhibit invertible 2-cells between both morphisms in the boundary of the diagram
and the composite morphism
\[
\begin{tikzcd}
\bang A \otimes \bang A \ar[r, "\seell_{A,A}"] &
\bang (A \oplus A) \ar[r, "\bang \nabla_A"] &
\bang A \ar[r, "\der_A"] &
A .
\end{tikzcd}
\]
For  $\der_A  \cocon_A$, unfolding the definition of $\cocon_A$, the required invertible 2-cell is obtained
by whiskering $\seell_{A,A}$ with the pseudonaturality 2-cell
\[
\begin{tikzcd}
\bang (A \oplus A) 
	\ar[r, "\bang \nabla_A "] 
	\ar[d, "\der_{A \oplus A}"'] 
	\ar[dr, phantom, description, "\Two"] &
\bang A 
	\ar[d, "\der_A"] \\
A \oplus A
	\ar[r, "\nabla_A"'] &
A .
\end{tikzcd}
\]
For $\der_A \otimes \weak_A   + \weak_A \otimes \der_A$, we have
\[
\begin{small}
\begin{tikzcd}
\bang A \otimes \bang A 
	\ar[r, "\dig_A \otimes \dig_A"]
	\ar[dr, bend right = 20, "\id_{\bang A} \otimes \id_{\bang A}"']   &[1.2cm]
\bbang A \otimes \bbang A 
	\ar[r,  "\monn_{\bang A, \bang A}"]
	\ar[d,  "\der_{\bang A} \otimes \der_{\bang A}"'] 
	\ar[dr, phantom, description, "\Two \der^2_{\bang A, \bang A}"] &[0.6cm]
\bang (\bang A \otimes \bang A)
	\ar[r, "\bang ( \der_A \otimes \weak_A {,} \weak_A \otimes \der_A)"]
	\ar[d, "\der_{\bang A \otimes \bang A}"] 
	\ar[dr, phantom, description, "\cong"] &[2cm] 
\bang (A \oplus A)
	\ar[d, "\der_{A \oplus A}"] \\
 \phantom{} \ar[ur, phantom, {pos=.6},  "\Two  \mu_A \otimes \mu_A"]  &
 \bang A \otimes \bang A  \ar[r, "\id"'] & 
 \bang A \otimes \bang A 
 	\ar[r, "( \der_A \otimes \weak_A {,} \weak_A \otimes \der_A)"] 
	\ar[dr, bend right = 20, "\der_A \otimes \weak_A + \weak_A \otimes \der_A"'] &
A \oplus A
	\ar[d, "\nabla_A"] \\
& &  \phantom{} \ar[ur, phantom, {pos=.6},  "\cong"] & A ,
\end{tikzcd}
\end{small}
\]
where we are using the left unitality constraints of the pseudomonad, the monoidality of $\der$ as in~\eqref{equ:counit-monoidal-2},
pseudonaturality of $\der$, and the definition of the convolution.
	
For the 2-cell on the right-hand side, the required 2-cell is given by
\[
\begin{tikzcd}
\unit 
	\ar[r, "\seeli"] 
	\ar[dr, bend right = 20]  &
\bang 0
	\ar[r]  
	\ar[d, "\der_0"] 
	\ar[dr, phantom, description, "\cong "] &
\bang A 
	\ar[d, "\der_A"] \\
\phantom{} \ar[ur, phantom, {pos=.7},  "\cong"] 	&
0
	\ar[r] &
A  ,
\end{tikzcd}
\]
where the 2-cell in the triangle is given by the universal property of the terminal object
and the 2-cell in the square is given by the pseudonaturality of $\der$.
\end{proof} 

\pagebreak

\begin{proposition} \leavevmode \nopagebreak
\begin{enumerate}[(i)]
\item  For every
$f \co A \to B$ and $g \co A \to B$, we have an invertible 2-cell
\[
\begin{tikzcd}
\bang A 
	\ar[r, "\bang (f + g)"]  
	\ar[d, "\con_A"'] 
	\ar[dr, phantom, description, "\Two"] & 
\bang B  \\
\bang A \otimes \bang A 
	\ar[r, "\bang f \otimes \bang g"'] & 
\bang B \otimes \bang B  .
	\ar[u, "\cocon_B"']
\end{tikzcd}
\]
\item  For every $A \in \catK$, we have an invertible 2-cell
\[
\begin{tikzcd}
\bang A 
	\ar[r, "\weak_A"]
	\ar[dr, bend right = 20, "\bang 0_{A,A}"'] & 
\unit 
	\ar[d, "\coweak_A"] \\
 \phantom{} \ar[ur, phantom, description, pos=(0.7), "\Two"] & \bang A .
 \end{tikzcd}
 \]
\end{enumerate}
\end{proposition} 

\begin{proof} For part~(i), unfolding the definition of $f+g \co A \to B$ in \eqref{equ:convolution}, the required 2-cell is the pasting
\[
\begin{tikzcd}[column sep = large]
\bang A 
	\ar[r, "\bang \Delta_A"] 
	\ar[dr, bend right = 20, "\con_A"'] & 
\bang (A \oplus A) 
	\ar[r, "\bang (f \oplus g)"] 
	\ar[dr, phantom, description, "\Two"] & 
\bang (B \oplus B) 
	\ar[r, "\bang \nabla_B"] & 
\bang B , \\
\phantom{} \ar[ur, phantom, pos= 0.7, "\Two"] & 
 \bang A \otimes \bang A
  \ar[r, "\bang f \otimes \bang g"'] 
  \ar[u, "\seell_{A,A}"'] & 
  \bang B \otimes \bang B  
  	\ar[ur, bend right = 20, "\cocon_B"'] 
	\ar[u, "\seell_{B,B}"] & 
 \end{tikzcd}
 \]
 where the 2-cell on the left-hand triangle is given by \cref{thm:comm-comonoids-coincide}, the one in the central square is pseudonaturality  of $\seel$, and the right-hand triangle commutes strictly by the definition of $\cocon$ in \eqref{equ:def-cocon}. For part~(ii), unfolding the definition of $0_{A,A}$ from \eqref{equ:zero-map} 
 and $\coweak$ from \eqref{equ:def-coweak}, we have
 \[
 \begin{tikzcd}
 \bang A 
 	\ar[r, "\weak_A"] 
	\ar[dr, bend right = 20]  &
\unit
 	\ar[d, "\seeli"] \\
\ar[ur, phantom, description, pos = (0.7), "\Two"]  	& 
 \bang 0
 	\ar[r] & 
\bang A ,
\end{tikzcd}
 \]
 where the 2-cell is given by \cref{thm:comm-comonoids-coincide}.
\end{proof}

\section{Codereliction}
\label{sec:codereliction}

The aim of this section is to introduce the notion of a codereliction on $\catK$. We do so in \cref{thm:codereliction}, inspired by  
the definition of the notion of a creation map in~\cite[Definition~4.3]{FioreM:difsmi}. For this, we work with a symmetric Gray monoid $\catK$ with biproducts and a linear
exponential pseudocomonad $(\bang, \dig, \der)$ satisfying the additional assumptions in \cref{thm:hypothesis} below. The reason for making these
assumptions is that, in order to state \cref{thm:codereliction},  we make use of some 2-cells which do not seem to necessarily exist in the setting considered so far.

\begin{hypothesis} \label{thm:hypothesis}
 \leavevmode
\begin{enumerate}
\item For each $A, B \in \catK$, the zero map $0_{A,B} \co A \to B$, as defined in \eqref{equ:zero-map}, is an initial object of $\catK[A,B]$ and the convolution product $f + g \co A \to B$, as defined in \eqref{equ:convolution}, is a coproduct, with coprojections $\iota_1 \co f \to f + g$ and $\iota_2 \co g \to f + g$ given by the 2-cells
\[
\begin{tikzcd}
f \ar[r, "\cong"] & f + 0 \ar[r] & f + g \, , 
\end{tikzcd} \qquad
\begin{tikzcd}
g \ar[r, "\cong"] &  0 + g  \ar[r] & f + g \, . 
\end{tikzcd} 
\]
\item For every $A \in \catK$, the component of the counit $\der_A \co \bang A \to A$ of the pseudocomonad has a left adjoint  $\coder_A \co A \to \bang A$ in $\catK$,
so that we have a unit $\eta_A$ and a counit $\varepsilon_A$, as in
\begin{equation}
\label{equ:first-pseudocomonad-constraint}
\begin{tikzcd}
A 
	\ar[r, "\coder_A"] 
	\ar[dr,  bend right = 30, "\id_A"'] & 
\bang A 
	\ar[d, "\der_A"] \\
\phantom{} \ar[ur, phantom, {pos=.7},  "\overset{\eta_A}{\TwoHor}"]   & A ,
 \end{tikzcd}  \qquad
 \begin{tikzcd}
\bang A 
	\ar[r, "\der_A"] 
	\ar[dr, bend right =30, "\id_{\bang A}"'] & 
A 
	\ar[d, "\coder_A"] \\
\phantom{} \ar[ur, phantom, {pos=.7},  "\Two \varepsilon_A"]   & \bang A  .
 \end{tikzcd}
 \end{equation}
\end{enumerate}
\end{hypothesis}

From now until the end of the section we work with a fixed symmetric Gray monoid $\catK$ with biproducts and a linear
exponential pseudocomonad $(\bang, \dig, \der)$ under the assumptions of \cref{thm:hypothesis}.

We refer to the 2-cells $\eta_A \co \id_A \Rightarrow \der_A \coder_A$
in~\eqref{equ:first-pseudocomonad-constraint}, for $A \in \catK$, as the \myemph{first pseudocomonad constraint} for~$\coder$. 
We introduce a second pseudocomonad constraint in~\cref{thm:second-pseudocomonad-constraint} and a strength constraint in~\cref{thm:strength-constraint}. To do so, we construct some auxiliary 2-cells in~\cref{thm:auxiliary-2-cells} below.

\begin{lemma} 
\label{thm:auxiliary-2-cells}
 For every $A \in \catK$, there are 2-cells as follows: 
\begin{enumerate}
\item 
\[
\begin{tikzcd}[ampersand replacement=\&]
\bang A \otimes \unit 
	\ar[r, "\id"]
	\ar[dr, bend right = 20, "\id_{\bang A} \otimes \coweak_A"']  \&
\bang A \ar[d, "\con_A"] \\
 \phantom{} \ar[ur, phantom, pos = 0.7, "\TwoHor"]	\&
\bang A \otimes \bang A  ,
\end{tikzcd}
\]
\item
\[
\begin{tikzcd}[ampersand replacement=\&]
\bang A
	\ar[r, "\con_A"]
	\ar[dr, bend right = 20, "\id_{\bang A}"'] \&
\bang A \otimes \bang A  \ar[d, "\cocon_A"] \\
  \phantom{} \ar[ur, phantom, pos = 0.7, "\Two"]	\&
\bang A .
\end{tikzcd}
\]
\end{enumerate} 
\end{lemma}

\begin{proof} For part~(i), we begin by constructing a 2-cell
\begin{equation}
\label{equ:claim-5}
\begin{tikzcd}[ampersand replacement=\&]
A \oplus 0
	\ar[r, "\simeq"] 
	\ar[dr, bend right =20, "\id_A \oplus u"'] \&
A 
	\ar[d, "\Delta_A"] \\
\phantom{} \ar[ur, phantom, description, pos= (0.6), "\TwoHor"]  \&
A \oplus A .
\end{tikzcd}
\end{equation}
This is induced by the following evident 2-cells:
\[
\begin{tikzcd}
A \oplus 0 
	\ar[r, "\pi_1"] 
	\ar[dr, phantom, description, "\cong"] 
	\ar[d, "\id_A \oplus u"'] &
A 
	\ar[r, "\Delta_A"] 
	\ar[d, "\id_A"] 
	\ar[dr, phantom, description, "\cong"]  &
A \oplus A 
	\ar[d, "\pi_1"] \\
   A \oplus A 
  	\ar[r, "\pi_1"'] & 
 A \ar[r, "\id_A"'] &
A , 
\end{tikzcd} \qquad 
\begin{tikzcd}
A \oplus 0 \ar[r, "\pi_1"] \ar[d, "\id_A \oplus u"']  
\ar[dr, phantom, description, "\cong"] &
A \ar[r, "\id_A"] \ar[d, "0_A"]  
\ar[dr, phantom, description, "\TwoHor"] &
A \ar[r, "\Delta_A"] \ar[d, "\id_A"] 
\ar[dr, phantom, description, "\cong"] &
A \oplus A \ar[d, "\pi_2"] \\
A \oplus A \ar[r, "\pi_2"'] & 
A \ar[r, "\id_A"'] & 
A \ar[r, "\id_A"'] & 
A  .
\end{tikzcd}
\]
Given \eqref{equ:claim-5}, the desired 2-cell follows immediately via the Seely equivalences in~\eqref{equ:seely-with-biproducts}.

For part~(ii), observe that we have a 2-cell
\[
\begin{tikzcd}[ampersand replacement=\&]
A
	\ar[r, "\Delta_A"]
	\ar[dr, bend right = 20, "\id_A"'] \&
A \oplus A  \ar[d, "\nabla_A"] \\
 \phantom{} \ar[ur, phantom, description, pos= (0.7), "\Two"]	\&
A ,
\end{tikzcd}
\]
since $\nabla_A \Delta_A = \id_A + \id_A$ and $ \id_A + \id_A$ is a coproduct by \cref{thm:hypothesis}. As before, we obtain the required 2-cell via the Seely equivalences.
\end{proof}

\begin{proposition}
\label{thm:second-pseudocomonad-constraint} 
For every $A \in \catK$, there is a $2$-cell $\mu_{A,B}$, which we call the \myemph{second pseudocomonad constraint} for $\coder$, fitting in the diagram
\begin{equation}
\label{equ:definition-of-mu}
\begin{tikzcd}[column sep = large]
A \otimes I 
	\ar[r, "\id"] 
	\ar[d, "\coder_A \otimes \coweak_A"']  
	\ar[drr, phantom, description, "\overset{\mu_{A,B}}{\TwoHor}"] & 
A 
	\ar[r, "\coder_A"] & 
\bang A 
	\ar[d, "\dig_A"] \\
\bang A \otimes \bang A 
	\ar[r, "\coder_{\bang A}  \otimes \dig_A"']  & 	
\bbang A \otimes \bbang A 
	\ar[r,  "\cocon_{\bang A}"'] & 
\bbang A .
\end{tikzcd}
\end{equation}
 \end{proposition} 
 
 \begin{proof} The required 2-cell is given by the pasting diagram
 \begin{equation}
 \begin{gathered} 
 \label{equ:mu-diagram}
 \begin{tikzcd}[row sep = 1.5cm, scale = 0.7] 
 A \otimes \unit \ar[r, "\id"] 
 \ar[dr, "\coder_A \otimes \id"']  
 \ar[ddr, "\coder_A \otimes \coweak_A"', bend right =30]  &
 A \ar[dr, "\coder_A"] &
  &
  &
  &
  & 
  \\
  &
 \bang A \otimes \unit \ar[r, "\id"] \ar[d, "\id \otimes \coweak_A"'] \ar[dr, phantom, description, "\overset{\mu_1}{\TwoHor}"] &[-1em]
 \bang A \ar[r, "\dig_A"] 
 \ar[d, "\con_A"]  \ar[dr, phantom, description, "\overset{\mu_2}{\cong}"] &
 \bbang A \ar[d, "\con_{\bang A}"] \ar[ddrr, "\id", bend left = 20]
 &  \phantom{} &  \phantom{}
 \\
  \phantom{} \ar[uur, phantom, "\cong"] &
  \bang A \otimes \bang A \ar[r, "\id"] &
  \bang A \otimes \bang A \ar[r, "\dig_A \otimes \dig_A"] \ar[d, "\id \otimes \dig_A"'] 
  \ar[dr, phantom, description, pos= (0.7), "\overset{\mu_3 \otimes \id}{\cong}"] &
  \bbang A \otimes \bbang A \ar[d, "\der_{\bang A} \otimes \id"] \ar[dr, "\id", bend left = 20] & \phantom{} &  \phantom{}
 \\
 & & \bang A \otimes \bbang A \ar[r, "\id"'] \ar[ur, pos =(0.3), "\dig_A \otimes 1"] & 
 \bang A \otimes \bbang A \ar[r, "\coder_{\bang A} \otimes \id"'] \ar[ur, phantom, description, pos= (0.4), "\overset{\mu_4 \otimes \id}{\TwoHor}"] 
 \ar[uurr, phantom, description, pos= (0.5), "\overset{\mu_5}{\TwoHor}"] 
 &
 \bbang A \otimes \bbang A \ar[r, "\cocon_{\bang A}"'] &
 \bbang A  .
 \end{tikzcd}
 \end{gathered}
 \end{equation}
Here 
$\mu_1$ is the 2-cell in part~(i) of \cref{thm:auxiliary-2-cells}, 
$\mu_2$ is a pseudonaturality 2-cell for $\con$, 
$\mu_3$ is the left unitality constraint of the pseudocomonad,  
$\mu_4$ is the counit 
of the adjunction $\coder_{\bang A} \dashv \der_{\bang A}$ as
in~\eqref{equ:first-pseudocomonad-constraint}, 
and 
$\mu_5$ is the 2-cell in part~(ii) of \cref{thm:auxiliary-2-cells}.
\end{proof}

\begin{proposition} 
\label{thm:strength-constraint}  
For every $A, B \in \catK$, there is a 2-cell $\sigma_{A,B}$, which we call the \myemph{strength constraint} for $\coder$, fitting in the diagram 
\begin{equation*}
\begin{tikzcd}[column sep = large] 
A \otimes \bang B 
	\ar[r, "\coder_A \otimes \id_{\bang B}"] 
	\ar[d, "\id_A \otimes \der_B"'] 
	\ar[dr, phantom, description, "\overset{\sigma_{A,B}}{\TwoHor}"] & 
\bang A \otimes \bang B 
	\ar[d,  "\monn_{A,B}"] \\
A \otimes B 
	\ar[r, "\coder_{A \otimes B}"'] & 
\bang (A \otimes B) .
\end{tikzcd}
\end{equation*}
\end{proposition}

\begin{proof} The 2-cell $\sigma_{A,B}$ is constructed as the following pasting:
\[
\begin{tikzcd}[column sep = 1.5cm] 
A \otimes \bang B 
	\ar[r, "\coder_A \otimes \id_{\bang B}"] 
	\ar[dr, bend right= 20, "\id_A \otimes \der_B"']  & 
\bang A \otimes \bang B 
	\ar[r, "\monn_{A,B}"]
	\ar[d, "\der_A \otimes \der_B"] 
	\ar[dr, phantom, description, "\TwoHor"] &
\bang (A \otimes B)
	\ar[d, "\der_{A \otimes B}"] 
	\ar[dr, bend left = 20, "\id_{\bang(A \otimes B)}"]  & 
	\phantom{}
	\\
\phantom{} \ar[ur, phantom, pos=(0.7), description, "\overset{\eta_A \otimes \der_B}{\TwoHor}"]  & 
A \otimes B
	\ar[r, "\id_{A \otimes B}"'] &
A \otimes B 
	\ar[r, "\coder_{A \otimes B}"'] 
	\ar[ur, phantom, description, pos=(0.3), "\overset{\varepsilon_{A\otimes B}}{\TwoHor}"]  &
\bang ( A \otimes B)  ,
\end{tikzcd}
\]
where the 2-cell in the square is given by the structure of monoidal pseudonatural transformation of $\der$, as in \cref{thm:sym-lax-monoidal-pseudocomonad}.
\end{proof}

\cref{thm:codereliction} below introduces our 2-categorical counterpart of the notion of a codereliction, in the special setting being considered here.  This uses the 2-cells $\eta_A$
of~\eqref{equ:first-pseudocomonad-constraint}, $\mu_{A,B}$ of  \cref{thm:second-pseudocomonad-constraint}, and $\sigma_{A,B}$ of \cref{thm:strength-constraint}.  
 Its part~(i) is the counterpart of the first comonad axiom in~\cite[Definition~4.3]{FioreM:difsmi} and is  also in \cite[{p.} 1026]{EhrhardT:intdll} and~\cite[{p.} 186, dC.3]{BluteR:difcr}, where it is called the Linear Rule. Part~(ii) is the counterpart of the second comomad axiom in~\cite[Definition~4.3]{FioreM:difsmi}.
which is also in~\cite[{p.} 1021]{EhrhardT:intdll} and~\cite[{p.} 187, dC.4']{BluteR:difcr}, where it is called the Alternative Chain Rule. 
Part (iii) is the counterpart of the Strength Axiom in~\cite[Definition~4.3]{FioreM:difsmi} and is also in \cite[{p.}~1026]{EhrhardT:intdll} and~\cite[{p.}~190, dC.m]{BluteR:difcr}, where it is called
the Monoidal Rule.

\begin{definition} \label{thm:codereliction}  A \myemph{codereliction} on $\catK$ is a pseudonatural transformation with components on objects $\coder_A \co A \to \bang A$, for $A \in \catK$, 
which is adjoint to $\der_A \co \bang A \to A$ and satisfies the following axioms.
\begin{enumerate}
\item \emph{First pseudocomonad axiom.} The first pseudocomonad constraint 
\[
\begin{tikzcd}
A 
	\ar[r, "\coder_A"] 
	\ar[dr,  bend right = 30, "\id_A"'] & 
\bang A 
	\ar[d, "\der_A"] \\
\phantom{} \ar[ur, phantom, {pos=.7},  "\overset{\eta_A}{\TwoHor}"]   & A 
 \end{tikzcd} 
 \] 
 is invertible for every $A \in \catK$.
\item \emph{Second pseudocomonad axiom.} The second pseudocomonad constraint 
\[
\begin{tikzcd}[column sep = large]
A \otimes I 
	\ar[r, "\id"] 
	\ar[d, "\coder_A \otimes \coweak_A"']  
	\ar[drr, phantom, description, "\overset{\mu_{A,B}}{\TwoHor}"]  & 
A 
	\ar[r, "\coder_A"] & 
\bang A 
	\ar[d, "\dig_A"] \\
\bang A \otimes \bang A 
	\ar[r, "\coder_{\bang A}  \otimes \dig_A"']  & 	
\bbang A \otimes \bbang A 
	\ar[r,  "\cocon_{\bang A}"'] & 
\bbang A 
\end{tikzcd}
\]
is invertible for every $A, B \in \catK$.
\item \emph{Strength axiom.} The strength constraint
\begin{equation*}
\begin{tikzcd}[column sep = large] 
A \otimes \bang B 
	\ar[r, "\coder_A \otimes \id_{\bang B}"] 
	\ar[d, "\id_A \otimes \der_B"'] 
	\ar[dr, phantom, description, "\overset{\sigma_{A,B}}{\TwoHor}"] & 
\bang A \otimes \bang B 
	\ar[d,  "\monn_{A,B}"] \\
A \otimes B 
	\ar[r, "\coder_{A \otimes B}"'] & 
\bang (A \otimes B) 
\end{tikzcd}
\end{equation*}
 is invertible for every $A, B \in \catK$.
\end{enumerate}
 \end{definition}

For a codereliction as in~\cref{thm:codereliction}, the pseudomonad and strength constraint will satisfy coherence conditions imposed by the universal properties of the 2-cells involved in
their definitions. Such conditions, which we do not pursue spelling out here, can be used to guide the formulation of a codereliction in a more general setting than the one considered in this
section (\ie avoiding \cref{thm:hypothesis}).
 
 In \cref{thm:consequences-of-dereliction} below, we have versions of some equations for differential categories:
 part~(i) is a counterpart of a diagram in~\cite[Page~1021]{EhrhardT:intdll}
 and of~\cite[Definition~9, dC.1]{BluteR:difcr}; 
 part~(ii)  of a diagram in~\cite[Page~1021]{EhrhardT:intdll}
 and of~\cite[Definition~9, dC.2]{BluteR:difcr}; 
 part~(iii) of~\cite[Definition~9, dC.4]{BluteR:difcr}; 
 part~(iv) of~\cite[Proposition~4.3]{FioreM:difsmi}. 
 The proof is analogous to that in the one-dimensional setting~\cite{BluteR:difcr} and hence omitted.
 
 \begin{proposition} \label{thm:consequences-of-dereliction} 
Assume that $\coder$ is a codereliction pseudonatural transformation in $\catK$.
 Then we have invertible 2-cells as follows.
\begin{enumerate}
\item 
\label{item:constant-rule}
Constant rule:
\[
\begin{tikzcd}
 A 
 	\ar[r, "\coder_A"] 
	\ar[dr, bend right = 20, "0"'] &   
\bang A 
	\ar[d, "\weak_A"] \\
\phantom{} \ar[ur, phantom,  pos=(0.7), "\Two"]   & \unit .
\end{tikzcd} 
\]
\item Product rule:
\label{item:product-rule}
\[
 \begin{tikzcd}
 A 
 	\ar[r, "\coder_A"] 
 	\ar[dr, bend right = 20, "\coder_A \otimes \coweak_A   + \coweak_A \otimes \coder_A"']
  & \bang A  
  	\ar[d, "\con_A"]  \\
	\phantom{} \ar[ur, phantom,  pos=(0.7), "\Two"] 
 & \bang A \otimes \bang A   .
 \end{tikzcd} 
 \]
 \item Chain rule:
  \label{item:chain-rule}
 \[
\begin{tikzcd}[column sep = large]
A \otimes \bang A 
	\ar[r, "\coder_A \otimes \id_{\bang A}"]
	\ar[d, "\coder_A \otimes  \con_A"'] 
	\ar[drr, phantom, description, "\Two"]& 
\bang A \otimes \bang A 
	\ar[r, "\cocon_A"] & 
\bang A 
	\ar[d, "\dig_A"] \\
\bang A \otimes \bang A \otimes \bang A 
	\ar[r, "(\coder_{!\!A}\cocon_A) \otimes \dig_A"'] & 
\bbang A \otimes \bbang A 
	\ar[r, "\cocon_{\bang A}"'] & \bbang A  .
\end{tikzcd}
\]
 \item Monoidal rule:
 \label{item:additional-from-fiore}
\[
\begin{tikzcd}[column sep = large]
 A \otimes B 
 	\ar[dr, bend right = 20, "\coder_{A \otimes B}"'] 
	\ar[r, "\coder_A \otimes \coder_B"] & 
\bang A \otimes \bang B  
	\ar[d, "\monn_{A,B}"] \\
\phantom{} \ar[ur, phantom,  pos=(0.7), "\Two"] 
 &  \bang (A \otimes B) .
 \end{tikzcd}  
\] \qed
\end{enumerate}
\end{proposition}

This concludes the development of the general theory and we now turn to our application to profunctors and analytic functors.

 \section{Profunctors, categorical symmetric sequences, and analytic functors}
 \label{sec:prof}

The aim of this section is to show that the bicategory of profunctors $\Prof$, with its symmetric monoidal
structure given by products in $\Cat$ (recalled below), can be equipped with a linear 
exponential pseudocomonad in the sense of our \cref{thm:linear-exponential-comonad} and a codereliction in the sense of \cref{thm:codereliction}. We then use this to extend the differentiation operation for analytic endofunctors on sets defined by Joyal in~\cite{JoyalA:fonaes} to the analytic functors between presheaf categories 
introduced in~\cite{FioreM:carcbg}, showing that it satisfies counterparts of all the rules of many-variable differential calculus.

Before proceeding with our development, let us explain how the developments in \cref{sec:prelim,sec:mon-comon-bialg,sec:linear-exponential,sec:products,sec:biproducts,sec:codereliction}, which is generally concerned with symmetric Gray monoids, can be applied to the symmetric monoidal bicategory~$\Prof$. The key to do so is \cref{thm:strictification-symmetric-monoidal}, according to which $\Prof$ is biequivalent, as a symmetric monoidal bicategory, to a Gray monoid, say $\catK$. Thus, any structure or property of $\Prof$ that is
invariant under biequivalence (such as having biproducts) or symmetric monoidal biequivalence (such as being equipped with a linear exponential pseudocomonad), will be inherited by $\catK$. We will therefore
be able to apply our development to obtain results on $\catK$ (such as the presence of the pseudocomonad
constraints and the strength constraint of \cref{sec:codereliction}). These can then be transferred back to $\Prof$, exploiting again the symmetric monoidal biequivalence with $\catK$. For brevity, we will leave these `transfer of structure' arguments implicit, trusting that readers will be able to fill in the details.

 \subsection*{Profunctors}  We briefly recall the definition and key properties of the bicategory
 of profunctors~\cite[Chapter~7]{BorceuxF:hanca} to set notation. We write $\Set$ for the category of sets and functions.
For small categories $A$ and $B$, a \myemph{profunctor} $F \co A \to B$ is
a functor $F \co B^\op \times A  \to \Set$. Small categories, profunctors, and natural transformations between them form a bicategory, 
written $\Prof$. Explicitly, the objects of $\Prof$ are small categories. For small categories
$A$ and $B$, the hom-category of morphisms from $A$ to $B$ is defined by letting:
\[
\Prof[A,B] \defeq \CAT[B^\op \times A, \Set] ,
\]
where $\CAT$ denotes the 2-category of locally small categories, functors and natural transformations.
The composition of profunctors $M \co A \to B$ and $N \co B \to C$ is
the profunctor $N \circ M \co A \to C$ defined by the coend formula
\[
(N \circ M)(c,a) = \coend^{b \in B} N(c, b) \times M(b,a)  .
\]
For a small category $A$, the identity profunctor $\id_A \co A \to A$ is defined by letting
\[
\id_A(a',a) = A[a',a] .
\]
The coherence conditions for a bicategory can be proved by identifying $\Prof$ as the Kleisli bicategory for the presheaf relative pseudomonad~\cite{FioreM:klebrp}. 

We now recall some of the  structure of $\Prof$.
First of all, $\Prof$ is  a symmetric monoidal bicategory
(see~\cite[Corollary~6.6]{WesterHansenL:consmb}). For small categories $A$ and $B$, their tensor product in $\Prof$ is defined by
letting
\[
A \otimes B \defeq A \times B ,
\]
where $A \times B$ denotes the product of $A$ and $B$ in $\Cat$.
The unit of the tensor product is the category $\mathsf{1}$ with a single object (written $\ast$ below) and no non-identity maps. Furthermore, $\Prof$ is compact closed (see~\cite[page~194]{KellyGM:cohccc}
and~\cite[pages~757-758]{StayM:comcb}). For a small category $A$, the dual $A^\perp$
 is defined by 
\begin{equation}
\label{equ:dual-in-prof}
A^\perp \defeq A^\op .
\end{equation}
Indeed, there is an adjunction $A^\bot \dashv A$ in  $\Prof$, whose unit and counit are 
the profunctors $u_A \co 1 \to A \otimes A^\perp$ and $v_A \co A^\perp \otimes A \to 1$,
defined by~$u_A(  (a',a), \ast ) \defeq A[a, a']$ and~$v_A( \ast , (a', a) ) \defeq A[a',a]$, respectively.  
The internal  hom in $\Prof$ of two small categories $A$ and $B$ is therefore given 
by
\begin{equation}
\label{equ:lin-hom-in-prof}
A \linhom B \defeq A^\op \times B .. 
\end{equation} 
Secondly, the bicategory $\Prof$ also has finite biproducts in the sense of \cref{thm:biproducts}. Given small categories $A$ and $B$, their biproduct in~$\Prof$ is  defined by letting
\[
A \oplus B \defeq A + B ,
\]
where $A + B$ denotes the coproduct of $A$ and $B$ in $\Cat$. The zero object of $\Prof$ is the empty category $\mathsf{0}$.

\subsection*{The linear exponential pseudocomonad} 
Let us recall the definition of the 2-monad for symmetric strict monoidal categories. For $n \in \mathbb{N}$, we write $\mathfrak{S}_n$ for the $n$-th symmetric group. 
For a  small category $A$ and $n \in \mathbb{N}$, we define the category~$\mathfrak{S}_n \wr A $ as follows. Its objects are $n$-tuples $\alpha = \langle a_1, \ldots, a_n\rangle$
of objects of $A$. Given two objects $\alpha = \langle a_1, \ldots, a_n \rangle $ and $\alpha' = \langle a'_1, \ldots, a'_n \rangle $, a map  $(\sigma, \vec{f}) \co \alpha \to \alpha'$ between them
is a pair consisting of a permutation  $\sigma \in \mathfrak{S}_n$ and an $n$-tuple $\vec{f} = (f_1, \ldots,
f_n)$ of maps $f_i \co a_i \to a'_{\sigma(i)}$ in $A$. We then let $\wn A$ be the following coproduct in 
$\Cat$:
\begin{equation}
\label{equ:wn-in-prof}
\wn A \defeq \bigsqcup_{n \in \mathbb{N}} \big ( \mathfrak{S}_n \wr A \big) .
\end{equation}
The category $\wn A$ admits the structure of a symmetric strict monoidal category. The tensor product of objects is given by concatenation of sequences,
which we write $\alpha \sqcup \alpha'$, and its unit is the empty sequence~$\varepsilon$. The symmetry of the tensor
product is evident. 
This operation extends to a 2-functor $\wn(-) \co \Cat \to \Cat$ which is part of the 2-monad whose strict algebras are symmetric strict monoidal categories~\cite{BlackwellR:twodmt},
so that $\wn A$ is the free symmetric strict monoidal category on $A$.  The multiplication of the 2-monad sends a sequence of sequences~$\varphi = \langle \alpha_1, \ldots, \alpha_n \rangle \in \wn \wn A$ to 
the sequence $\sqcup \varphi \in \wn A$, which is defined as the concatenation
\[
\sqcup \varphi  = \alpha_1 \sqcup \ldots \sqcup \alpha_n . 
\]
The unit takes $a \in A$ to the singleton sequence $\langle a \rangle \in \wn A$. We let  $\PP \defeq \wn \mathsf{1}$ and identify it with the category of natural numbers and 
permutations, \cf \cite{KellyG:opemay}.

The 2-monad on $\Cat$ for symmetric strict monoidal categories defined above can be extended to a pseudomonad on $\Prof$ 
via a form of pseudodistributivity~\cite{FioreM:klebrp}. The duality 
that is available on $\Prof$ then allows us to turn the resulting pseudomonad into a pseudocomonad $\oc(-) \co \Prof \to \Prof$ as in \eqref{equ:wn-to-bang}. On objects, we have 
\[
\oc A \defeq \wn A , 
\]
where $\wn A$ is defined in~\eqref{equ:wn-in-prof}. The comultiplication $\dig_A \co \bang A \to \bbang A$ and the counit $\der_A \co \bang A \to A$ are defined by
\begin{equation}
\label{equ:dig-and-der-in-prof}
\begin{aligned}
\dig_A( \varphi, \alpha)  & \defeq \oc A [ \sqcup \varphi , \alpha  ] , \\
\der_A (\alpha, a)  & \defeq \oc A [ \alpha, \langle a \rangle ] .
\end{aligned}
\end{equation}

\begin{theorem} \label{thm:prof-degenerate-model} 
The pseudocomonad $\oc(-) \co \Prof \to \Prof$ admits the structure of a linear exponential pseudocomonad.
\end{theorem} 

\begin{proof} We apply \cref{thm:case-2} to construct a linear exponential pseudocomonad on $\Prof$ and then show that its underlying
pseudocomonad is exactly the pseudocomonad $\oc$ defined above. To begin, let us show that $\Prof$ admits the construction of free symmetric pseudomonoids, \ie that we have 
a biadjunction of the form
\begin{equation}
\label{equ:cmon-in-prof}
\begin{tikzcd}[column sep = large]
 \SymMon{\Prof}  \ar[r, shift right =2, "U"'] 
  \ar[r, description, phantom, "\scriptstyle \bot"] 
	&  \Prof ,
 \ar[l, shift right = 2, "F"'] 
\end{tikzcd}
\end{equation}
where $U$ is the evident forgetful pseudofunctor. In order to construct the required left biadjoint, recall that a symmetric pseudomonoid in $\Prof$ is a 
symmetric promonoidal category~\cite{DayB:clocf,DayB:embtcc}. For a small category~$A$, we then define~$FA$ to be the symmetric promonoidal category associated to the symmetric
monoidal category~$\wn A$ defined in~\eqref{equ:wn-in-prof}. The required adjointness follows once we show that we have a pseudonatural family of equivalences 
\[
\SymMon{\Prof}[ FA, B] \simeq \Prof[ A, B] ,
\]
for  a small category $A$ and  a small promonoidal category $B$. Applying the
correspondence between symmetric promonoidal categories and symmetric monoidally cocomplete categories~\cite{DayB:clocf,DayB:embtcc,DayB:kanepf} on the left-hand side, 
and the definition of $\Prof$ on the right hand side, it suffices to show that we have a pseudonatural family of equivalences 
\[
\SMonCoc[ \Psh(FA), \Psh(B)]  \simeq \CAT[A, \Psh(B)] \, . 
\]
where $\SMonCoc$ denotes the 2-category of symmetric monoidally cocomplete categories, symmetric strong monoidal cocontinuous functors, and monoidal
transformations. But these equivalences now follow immediately by the universal property of the Day convolution~\cite{DayB:clocf,ImG:unipcm}.  

Furthermore, the biadjunction in~\eqref{equ:cmon-in-prof} is monadic. Indeed, pseudoalgebras for the pseudomonad~$\wn(-)$ on $\Prof$ are exactly symmetric unbiased promonoidal categories, in the sense of~\cite{LeinsterT:higohc}, which are equivalent to (biased) symmetric promomonoidal categories via a routine
calculation which we leave to the reader.

To complete the proof, it is sufficient to observe that the pseudomonad associated to the biadjunction in~\eqref{equ:cmon-in-prof} is the pseudomonad $\wn(-) \co \Prof \to \Prof$ 
defined above.
 \end{proof}
 
 Just as we defined explicitly the comultiplication and counit in~\eqref{equ:dig-and-der-in-prof}, it is possible to give explicit definitions for all the relevant parts of the linear exponential
 pseudocomonad of \cref{thm:prof-degenerate-model}, very much in analogy with what happens in the relational model of Linear Logic. While we do not unfold them explicitly here for brevity, these
 will be used in the proofs of \cref{thm:lemma1,thm:lemma2,thm:lemma3} below.

 As a final step before turning to the codereliction, we check that the additional assumptions of \cref{thm:hypothesis} hold in $\Prof$. 

\begin{proposition}  \label{thm:prof-satisfies-hyp} \leavevmode
\begin{enumerate}
\item In $\Prof$, the convolution structure induced by the biproducts is cocartesian.
\item In $\Prof$, the component  of the counit of the linear exponential pseudocomonad $\der_A \co \bang A \to A$  has a left adjoint for every $A  \in \Prof$.
\end{enumerate}
\end{proposition} 

\begin{proof} For part~(i), let $A \in \Prof$. The diagonal $\Delta_A \co A \to A \oplus A$ is the profunctor defined by letting, for $(x,a) \in (A \oplus A)^\op \times A$,
\[
\Delta_A(x, a) = (A \oplus A)[ x, \iota_1(a)] + (A \oplus A)[x, \iota_2(a)] .
\]
For $B \in \Prof$, the codiagonal $\nabla_B \co B \oplus B \to B$ is defined by letting, for $(b,y) \in  B^\op \times (B \oplus B)$,
\[
\nabla_B(b,y) = (B \oplus B)[ \iota_1(b), y] + (B\oplus B)[\iota_2(b), y] .
\]
Also, for $F, G \co A \to B$ in $\Prof$, $F \oplus G \co A \oplus A \to B \oplus B$ is given by letting
\begin{multline*} 
(F \oplus G)(y,x) = \coend^{b, a} (B \oplus B)[y, \iota_1(b)]  \times F(b,a) \times (A\oplus A)[\iota_1(a), x] \ +  \\
  \coend^{b, a} (B \oplus B)[y, \iota_2(b)]  \times G(b,a) \times (A\oplus A)[\iota_2(a), x] .
\end{multline*}
A direct calculation then shows that, for $F, G \co A \to B$ in $\Prof$, $\nabla_B \circ (F \oplus G) \circ \Delta_A \cong F + G$, where $F + G \co A \to B$ is
defined by 
\[
(F +  G)(b,a) \defeq F(b,a) + G(b,a) . 
\]
 Similarly, the composite of $A \to 0$ with $0 \to B$ is the initial object of $\Prof[A,B]$, \ie the profunctor $0_{A,B}(b,a) = \varnothing$.
 
Part~(ii) is evident, since it is well-known that every map in $\Prof$, in
particular the unit $A^\perp\to \wn(A^\perp)$, has a right adjoint and hence
its dual has a left adjoint.
\end{proof}

Let us also point out that the composition functors of $\Prof$ are `bilinear' and strict, in the sense that we have natural isomorphisms
\begin{gather*}
N \circ (M + M') \cong (N \circ M) + (N \circ M') , \quad
(N + N') \circ M \cong (N \circ M) + (N' \circ M) , \\ 
0
\circ M \cong 0 
, \quad
N \circ 0 
\cong 0 .
\end{gather*}

\subsection*{Codereliction for categorical symmetric sequences}  \cref{thm:prof-degenerate-model} and \cref{thm:prof-satisfies-hyp} allow us to apply the definitions and results  in~\cref{sec:products,sec:biproducts} to $\Prof$. We now verify that $\Prof$ admits a codereliction operation in the sense of in \cref{thm:codereliction}. 
For $A \in \Prof$, we define $\coder_A \co A \to \bang A$ to be the left adjoint of~$\der_A \co \oc A \to A$ in $\Prof$, namely
\begin{equation}
\label{equ:codereliction-def}
\coder_A(\alpha, a) \defeq \oc A [ \alpha,  \langle a \rangle]  . 
\end{equation}
For convenience of the readers, we organise the proof of the invertibility of the three 2-cells $\eta_A$, $\mu_A$, and~$\sigma_{A,B}$ in three lemmas, each dealing with one 2-cell. In these lemmas, we construct natural isomorphisms $\eta'_A$,~$\mu'_A$, and~$\sigma'_{A,B}$, define explicitly the natural transformations $\eta_A$, $\mu_A$, and $\sigma_{A,B}$, and conclude by observing that~$\eta_A = \eta'_A$, $\mu_A = \mu'_A$, and $\sigma_{A,B} = \sigma'_{A,B}$.

\begin{lemma} 
\label{thm:lemma1}
For every $A \in \Prof$, the first pseudocomonad constraint, in
\[
\begin{tikzcd}
A 
	\ar[r, "\coder_A"] 
	\ar[dr,  bend right = 30, "\id_A"'] & 
\bang A 
	\ar[d, "\der_A"] \\
\phantom{} \ar[ur, phantom, {pos=.7},  "\overset{\eta_A}{\TwoHor}"]   & A ,
 \end{tikzcd} 
 \] 
is invertible.
\end{lemma} 

\begin{proof} Let us define a natural transformation $\eta'_A$, parallel to $\eta_A$, whose component for $a' \in A$ and $a \in A$ is defined as the evident composite
\[
A[a', a] \Rightarrow \bang A[ \langle a' \rangle, \langle a \rangle] \cong \int^{\alpha \in \bang A} \bang A[ \langle a' \rangle, \alpha] \times \bang A[ \alpha, \langle a \rangle]  . 
\]
This is clearly invertible and  $\eta_A = \eta'_A$, so that $\eta_A$ is also invertible.
\end{proof}

\begin{lemma} 
\label{thm:lemma2}
For every $A \in \Prof$, the second pseudocomonad constraint $\mu_A$, in
\[
\begin{tikzcd}[column sep = large]
A \otimes I 
	\ar[r, "\id"] 
	\ar[d, "\coder_A \otimes \coweak_A"']  
	\ar[drr, phantom, description, "\overset{\mu_{A,B}}{\TwoHor}"]  & 
A 
	\ar[r, "\coder_A"] & 
\bang A 
	\ar[d, "\dig_A"] \\
\bang A \otimes \bang A 
	\ar[r, "\coder_{\bang A}  \otimes \dig_A"']  & 	
\bbang A \otimes \bbang A 
	\ar[r,  "\cocon_{\bang A}"'] & 
\bbang A ,
\end{tikzcd}
\]
 is invertible.
\end{lemma}

\begin{proof} Let us define the natural transformation $\mu'_A$, parallel to $\mu_A$, by letting its component for $\varphi' \in \bbang A$ and $a \in A$ be the composite
\begin{align*}
\int^{\varphi \in \bbang A} 
  \bbang A \big[\varphi', \langle \langle a \rangle \rangle \sqcup \varphi \big]  \times 
  \bang A [\sqcup \varphi, \varepsilon] 
& \ \Rightarrow\ 
\int^{\varphi \in \bbang A} 
  \bang A [  \sqcup \varphi', 
             \sqcup (  \langle \langle a \rangle \rangle \sqcup \varphi ) ] 
  \times 
  \bang A [  \sqcup \varphi, \varepsilon] 
\\
& \ \cong\ 
\int^{\varphi \in \bbang A} 
  \bang A [  \sqcup \varphi', \langle a \rangle \sqcup (  \sqcup \varphi ) ] 
  \times 
  \bang A [  \sqcup \varphi, \varepsilon] 
\\ 
& \ \Rightarrow\ 
\bang A [  \sqcup \varphi', \langle a \rangle \sqcup \varepsilon] 
\\
& \ \cong\ 
\bang A [ \sqcup \varphi', \langle a \rangle] 
\end{align*}
 Since this is invertible, it suffices to show $\mu_A = \mu'_A$. 
 
 We begin by unfolding explicitly the natural transformations that appear in the definition of $\mu_{A,B}$ in~\eqref{equ:mu-diagram}. 
 The component of the natural transformation $\mu_1$ for $(\alpha_1, \alpha_2) \in \bang A \otimes \bang A$ and $\alpha \in \bang A$  is
 \begin{align*} 
 \bang A[ \alpha_1, \alpha] \times \bang A [\alpha_2, \varepsilon] & \Rightarrow \bang A[ \alpha_1 \sqcup \alpha_2, \alpha \sqcup \varepsilon] \\
 & \cong \bang A[ \alpha_1 \sqcup \alpha_2, \alpha]  .
 \end{align*}
The component of the natural isomorphism $\mu_2$ for $(\varphi_1, \varphi_2) \in \bbang A \otimes \bbang A$ and $\alpha \in \bang A$ is
 \begin{align*}
 &  \int^{\alpha_1, \alpha_2 \in \bang A} \bang A [\sqcup \varphi_1, \alpha_1] \times \bang A[ \sqcup \varphi_2, \alpha_2] \times \bang A[ \alpha_1 \sqcup \alpha_2, \alpha] \\
   & \quad \cong \bang A [ (\sqcup \varphi_1) \sqcup (\sqcup \varphi_2), \alpha] \\
   & \quad  \cong \bang A[ \sqcup (\varphi_1 \sqcup \varphi_2), \alpha] \\
   & \quad  \cong \int^{\varphi \in \bbang A} \bbang A [ \varphi_1 \sqcup \varphi_2, \varphi] \times \bang A[ \sqcup \varphi, \alpha] . 
  \end{align*}
 The component of the natural isomorphism $\mu_3$ for $\alpha' \in \bang A$ and $\alpha \in \bang A$ is
  \begin{align*} 
  \bang A [ \alpha', \alpha] & \cong \bang A[ \sqcup \langle \alpha' \rangle, \alpha] \\
  & \cong \int^{\varphi \in \bbang A} \bbang A [ \langle \alpha' \rangle, \varphi] \times \bang A[ \sqcup \varphi, \alpha] . 
  \end{align*}
  The component of $\mu_4$ for $\varphi' \in \bbang A$ and $\varphi \in \bang A$ is 
  \[
  \int^{\alpha \in \bang A} \bbang A [ \varphi', \langle \alpha \rangle ]  \times \bbang A [ \langle \alpha \rangle, \varphi] \Rightarrow \bbang A [\varphi', \varphi]  . 
  \]
 Finally, the component of $\mu_5$ for $\varphi' \in \bbang A$ and $\varphi \in \bbang A$ is
 \[
 \int^{\varphi_1, \varphi_2} \bbang A[\varphi', \varphi_1 \sqcup \varphi_2] \times \bbang A [\varphi_1 \sqcup \varphi_2, \varphi ]
	\Rightarrow \bbang A [ \varphi' , \varphi]  . 
\]
 
 We can now calculate explicitly the value of some of the natural transformations obtained by pasting together regions of the diagram in~\eqref{equ:mu-diagram}. 
 The component of the one obtained by pasting the three natural transformations in the top left hand side of the diagram for $(\alpha_1, \alpha_2) \in \bang A \otimes \bang A$ and $\alpha \in \bang A$ is
 \begin{align*} 
 \bang A [ \alpha_1, \langle a \rangle] \times \bang A[ \alpha_2, \varepsilon] 
 	& \Rightarrow  \big(  \bang A[ \alpha_1, \langle a \rangle] \times \bang A[ \alpha_2, \varepsilon] \big) +  \big( \bang A[ \alpha_1, \varepsilon ] \times \bang A[ \alpha_2, \langle  a \rangle]  \big) \\
 &	\cong \bang A [ \alpha_1 \sqcup \alpha_2, \langle a \rangle] , 
\end{align*}
where the first map is a coproduct injection.
 The component of the natural transformation obtained by pasting the above with the three natural transformations in the central rectangle of~\eqref{equ:mu-diagram}, for $(\alpha, \varphi) \in \bang A \otimes \bbang A$ and $a \in A$, is
 \begin{align*} 
 \bang A [ \alpha, \langle  a \rangle ] \times \bang A[ \sqcup \varphi, \varepsilon] & \Rightarrow
 \big(  \bang A[ \alpha,  \langle  a \rangle ] \times \bang A[ \sqcup \varphi, \varepsilon]  \big) + \big( \bang A[ \alpha, \varepsilon ] \times \bang A[ \sqcup \varphi,  \langle  a \rangle] \big)  \\
  &   \cong \bang A [ \alpha \sqcup ( \sqcup \varphi),  \langle  a \rangle ] . 
  \end{align*}
  Whiskering the above with $\cocon_{\bang A}$ and $\coder_{\bang A} \otimes \id$ results in the natural transformation whose component at $\varphi' \in \bbang A$ and $a \in A$ is
  \begin{equation}
  \label{equ:useful-chain-of-coends}
  \begin{aligned}
 &  \int^{\varphi \in \bbang A} \bbang A [\varphi', \langle \langle a \rangle \rangle \sqcup \varphi] \times \bang A [ \sqcup \varphi, \varepsilon]  \\
 & \qquad \Rightarrow 
 	\int^{\varphi \in \bbang A} \bbang A [\varphi', \langle \langle a \rangle \rangle \sqcup \varphi] \times \bang A [ \sqcup \varphi, \varepsilon] 
		+ \int^{\varphi \in \bbang A} \bbang A [\varphi', \langle \varepsilon \rangle \sqcup \varphi] \times \bang A [ \sqcup \varphi, \langle a \rangle] \\
 & \qquad \cong 
 	\int^{\varphi \in \bbang A} \int^{\alpha \in \bang A} \bbang A [\varphi', \langle \alpha \rangle \sqcup \varphi] \times \bang A [ \alpha \sqcup ( \sqcup \varphi), \langle a \rangle] \\
 & \qquad \cong 
 	\int^{\varphi \in \bbang A} \int^{\alpha \in \bang A} \bbang A [\varphi', \langle \alpha \rangle \sqcup \varphi] \times \bang A [ \sqcup  ( \langle \alpha \rangle \sqcup  \varphi ), \langle a \rangle] . 
  \end{aligned}
  \end{equation}
  
We now consider the natural transformation obtained by pasting the 2-cells $\mu_4\otimes \id$ and $\mu_5$ in~\eqref{equ:mu-diagram}. Its component for $\varphi' \in \bbang A$ and $\varphi \in \bbang A$ is
\begin{align*}
 & \int^{\alpha \in \bang A} \int^{\varphi_1, \varphi'_1, \varphi'_2 \in \bbang A} \bbang A [ \varphi', \varphi'_1 \sqcup \varphi'_2] \times \bbang A [ \varphi'_1, \langle \alpha \rangle] \times \bbang A [ \langle \alpha \rangle, \varphi_1] 
 	\times \bbang A [\varphi_1 \sqcup \varphi'_2, \varphi] \\
& \qquad \Rightarrow 
	\int^{\varphi_1, \varphi'_1, \varphi'_2 \in \bbang A} \bbang A [ \varphi', \varphi'_1 \sqcup \varphi'_2] \times \bbang A [ \varphi'_1, \varphi_1]  	\times \bbang A [\varphi_1 \sqcup \varphi'_2, \varphi] \\
& \qquad \cong
	\int^{\varphi'_1, \varphi'_2 \in \bbang A} \bbang A [ \varphi', \varphi'_1 \sqcup \varphi'_2] \times \bbang A [\varphi'_1 \sqcup \varphi'_2, \varphi] \\
& \qquad \Rightarrow
	\bbang A [ \varphi', \varphi] . 
\end{align*}
Therefore, the precomposition of this natural transformation with $\dig_A \coder_A$ (\cf the top arrow in~\eqref{equ:mu-diagram}) gives the natural transformation whose component at $\varphi' \in \bbang A$ and $a \in A$ is 
 \begin{align*}
  & 
  \int^{\alpha \in \bang A} \int^{\varphi_1, \varphi'_1, \varphi'_2, \varphi \in \bbang A} 
   \bbang A [ \varphi', \varphi'_1 \sqcup \varphi'_2] \times \bbang A [\varphi'_1, \langle \alpha \rangle] \times \\
    & \qquad \qquad \qquad \qquad  \qquad \qquad \qquad   \bbang A[ \langle \alpha \rangle, \varphi_1]  \times   \bbang A [\varphi_1 \sqcup \varphi'_2, \varphi] \times \bang A[ \sqcup \varphi, \langle a \rangle] \\
    & \qquad \Rightarrow 
     \int^{\varphi_1, \varphi'_1, \varphi'_2, \varphi \in \bbang A}  \bbang A [ \varphi', \varphi'_1 \sqcup \varphi'_2] \times \bbang A [\varphi'_1, \varphi_1] \times  \bbang A [\varphi_1 \sqcup \varphi'_2, \varphi] \times \bang A[ \sqcup \varphi, \langle a \rangle] \\
     & \qquad \cong 
      \int^{\varphi'_1, \varphi'_2, \varphi \in \bbang A}  \bbang A [ \varphi', \varphi'_1 \sqcup \varphi'_2]  \times \bbang A [\varphi'_1 \sqcup \varphi'_2, \varphi] \times \bang A[ \sqcup \varphi, \langle a \rangle] \\
      & \qquad \Rightarrow 
      \int^{\varphi \in \bbang A} \bbang A [ \varphi', \varphi] \times \bang A[ \sqcup \varphi,  \langle  a \rangle] \\
      & \qquad \cong 
      \bang A[ \sqcup \varphi',  \langle  a \rangle]   . 
   \end{align*}
 This equals the following composite: 
 \begin{align*}
  & 
  \int^{\alpha \in \bang A} \int^{\varphi_1, \varphi'_1, \varphi'_2, \varphi \in \bbang A} 
   \bbang A [ \varphi', \varphi'_1 \sqcup \varphi'_2] \times \bbang A [\varphi'_1, \langle \alpha \rangle] \times \\
    &  \qquad \qquad \qquad \qquad  \qquad \qquad \qquad     
    \bbang A[ \langle \alpha \rangle, \varphi_1]  \times   \bbang A [\varphi_1 \sqcup \varphi'_2, \varphi] \times \bang A[ \sqcup \varphi, \langle a \rangle] \\
  & \qquad \cong 
   \int^{\alpha \in \bang A} \int^{\varphi'_2, \varphi  \in \bbang A} \bbang A [ \varphi', \langle \alpha \rangle \sqcup \varphi'_2] \times \bbang A [ \langle \alpha \rangle \sqcup \varphi'_2, \varphi] \times \bang A[ \sqcup \varphi,  \langle  a \rangle] \\
   & \qquad \Rightarrow 
   \int^{\varphi \in \bbang A} \bbang A [ \varphi', \varphi] \times \bang A[ \sqcup \varphi,  \langle  a \rangle] \\
   & \qquad \cong \bang A [ \sqcup \varphi',  \langle  a \rangle] ,
   \end{align*}
which, in turn, equals the composite
 \begin{align*}
  &  \int^{\alpha \in \bang A} \int^{\varphi_1, \varphi'_1, \varphi'_2, \varphi \in \bbang A}  \bbang A [ \varphi', \varphi'_1 \sqcup \varphi'_2] \times \bbang A [\varphi'_1, \langle \alpha \rangle] \times \\
  &  \qquad \qquad \qquad \qquad  \qquad \qquad \qquad   \bbang A[ \langle \alpha \rangle, \varphi_1]  \times   
  				\bbang A [\varphi_1 \sqcup \varphi'_2, \varphi] \times \bang A[ \sqcup \varphi, \langle a \rangle] \\
& \qquad \cong  \int^{\alpha \in \bang A} \int^{\varphi'_2, \varphi \bbang A}  \bbang A [\varphi', \langle \alpha \rangle \sqcup \varphi'_2] \times \bbang A [ \langle \alpha \rangle \sqcup \varphi'_2, \varphi] \times \bang A [ \sqcup \varphi,  \langle  a \rangle] \\
& \qquad \cong \int^{\alpha \in \bang A} \int^{\varphi'_2 \in \bbang A} \bbang A [\varphi', \langle \alpha \rangle \sqcup \varphi'_2] \times \bang A [ \sqcup \big( \langle \alpha \rangle \sqcup \varphi'_2),  \langle  a \rangle]  \\
& \qquad \Rightarrow \int^{\alpha \in \bang A} \int^{\varphi'_2 \in \bbang A} \bang A [\sqcup \varphi', \sqcup ( \langle \alpha \rangle \sqcup \varphi'_2 )] \times \bang A [ \sqcup ( \langle \alpha \rangle \sqcup \varphi'_2),  \langle  a \rangle]  \\
& \qquad \Rightarrow \bang A [ \sqcup \varphi',  \langle  \alpha \rangle ]  . 
\end{align*}

From this and~\eqref{equ:useful-chain-of-coends}, the component of $\mu_A$ for $\varphi' \in \bbang A$ and $a \in A$, is
 \begin{align*}
  & \int^{\varphi \in \bbang A} \bbang A [ \varphi', \langle \langle a \rangle \rangle \sqcup \varphi ] \times \bang A[ \sqcup \varphi, \varepsilon] \\
   & \qquad \Rightarrow \int^{\varphi \in \bbang A}  \bbang A [ \varphi', \langle \langle a \rangle \rangle \sqcup \varphi ] \times \bang A[ \sqcup \varphi, \varepsilon]  + \int^{\varphi \in \bbang A}  \bbang A [ \varphi', 
   \langle \varepsilon \rangle \sqcup \varphi ] \times \bang A[ \sqcup \varphi,  \langle  a \rangle]  \\ 
   & \qquad \cong \int^{\varphi \in \bbang A} \int^{\alpha \in \bang A} \bbang A [ \varphi', \langle \alpha \rangle \sqcup \varphi ] \times \bang A[ \alpha \sqcup (\sqcup \varphi),  \langle  a \rangle]  \\
   & \qquad \cong \int^{\varphi \in \bbang A} \int^{\alpha \in \bang A} \bbang A [ \varphi', \langle \alpha \rangle \sqcup \varphi ] \times \bang A[ \sqcup ( \langle \alpha \rangle  \sqcup \varphi) ,  \langle  a \rangle]  \\ 
   & \qquad \Rightarrow \int^{\varphi \in \bbang A} \int^{\alpha \in \bang A} \bang A [ \sqcup \varphi',  \sqcup \big( \langle \alpha \rangle \sqcup \varphi \big) ] \times \bang A[ \sqcup ( \langle \alpha \rangle  \sqcup \varphi) ,  \langle  a \rangle] \\
   & \qquad \Rightarrow \bang A[ \sqcup \varphi',  \langle  a \rangle] .
   \end{align*} 
This equals 
 \begin{align*}
  & \int^{\varphi \in \bbang A} \bbang A [ \varphi', \langle \langle a \rangle \rangle \sqcup \varphi ] \times \bang A[ \sqcup \varphi, \varepsilon] \\
  & \qquad \Rightarrow \int^{\varphi \in \bbang A} \bang A[ \sqcup \varphi', \sqcup \big( \langle \langle a \rangle \rangle \sqcup \varphi \big) ] \times \bang A[ \sqcup \varphi, \varepsilon] \\
  & \qquad \cong  \int^{\varphi \in \bbang A} \bang A[ \sqcup \varphi', \langle  a  \rangle \sqcup ( \sqcup \varphi ) ] \times \bang A[ \sqcup \varphi, \varepsilon] \\
  & \qquad \Rightarrow \bang A[ \sqcup \varphi', \langle  a  \rangle \sqcup \varepsilon] \\
  & \qquad \cong \bang A [ \sqcup \varphi', \langle  a  \rangle] ,
\end{align*} 
which is invertible.
\end{proof}

\begin{lemma}
\label{thm:lemma3}
For every $A, B \in \Prof$, the strength constraint $\sigma_{A,B}$, in
\[
\begin{tikzcd}[column sep = large] 
A \otimes \bang B 
	\ar[r, "\coder_A \otimes \id_{\bang B}"] 
	\ar[d, "\id_A \otimes \der_B"'] 
	\ar[dr, phantom, description, "\overset{\sigma_{A,B}}{\TwoHor}"]  & 
\bang A \otimes \bang B 
	\ar[d,  "\monn_{A,B}"] \\
A \otimes B 
	\ar[r, "\coder_{A \otimes B}"'] & 
\bang (A \otimes B) ,
\end{tikzcd}
\] is invertible.
\end{lemma}

\begin{proof} Recall that $A \otimes B$ in $\Prof$ is given by the cartesian product $A \times B$ in $\Cat$ and so we have functors $\pi_1 \co A \otimes B \to A$ and $\pi_2 \co A \otimes B \to B$.
We define the natural transformation~$\sigma'_{A,B}$, parallel to $\sigma_{A,B}$, by letting its component for $\gamma \in \bang (A \otimes B)$ and $(a, \beta) \in A \otimes \bang B$ be the composite
  \begin{align*}
 &  \int^{b \in B}  \bang B [ \langle b \rangle, \beta] \times \bang (A \otimes B)[ \gamma, \langle (a,b) \rangle] 
 	\\
 &	\quad  \Rightarrow \int^{b \in B} \bang B[ \langle b \rangle, \beta] \times \bang A [ \bang \pi_1 (\gamma), \bang \pi_1 \langle ( a,b) \rangle ] \times \bang B[ \bang \pi_2 (\gamma), \bang \pi_2 \langle (a,b) \rangle] \\
	&\quad  \cong   \int^{b \in B} \bang B[ \langle b \rangle, \beta] \times \bang A [ \bang \pi_1 (\gamma), \langle  a  \rangle ] \times \bang B[ \bang \pi_2 (\gamma), \langle b \rangle] \\ 
	& \quad \Rightarrow \bang A[ \bang \pi_1 (\gamma), \langle  a  \rangle] \times \bang B [ \bang \pi_2 (\gamma), \beta] .
\end{align*}
This is invertible. Furthermore, an inspection of the construction of the strength constraint in the proof of \cref{thm:strength-constraint} reveals that the component  of $\sigma_{A,B}$ for  $\gamma \in \bang (A \otimes B)$ and~$(a, \beta) \in A \otimes \bang B$ is
\begin{align*}
& \int^{b \in B} \bang B[ \langle b \rangle, \beta] \times \bang (A \otimes B)[ \gamma, \langle (a,b) \rangle]  \\
 & \qquad \cong 
 \int^{(a', b) \in A \otimes B} A[a', a] \times \bang B[ \langle b \rangle, \beta] \times \bang (A \otimes B) [\gamma, \langle (a', b) \rangle] \\
	& \qquad \cong \int^{(a',b) \in A \otimes B} \bang A[ \langle  a'  \rangle, \langle  a  \rangle] \times \bang B [ \langle  b  \rangle, \beta] \times \bang (A \otimes B) [\gamma, \langle  (a', b) \rangle] \\
	& \qquad \cong \int^{(a',b) \in A \otimes B} \int^{\alpha \in \bang A}  \bang A[ \langle  a'  \rangle, \alpha] \times  \bang A[  \alpha, \langle  a  \rangle] \times \bang B [ \langle  b  \rangle, \beta] \times \bang (A \otimes B) [\gamma, \langle (a', b) \rangle] \\
	&\qquad \cong \int^{(a',b) \in A \otimes B} \int^{\alpha \in \bang A} \int^{\gamma' \in \bang (A \otimes B)} \bang A[  \alpha, \langle  a \rangle] \times \bang (A \otimes B)[ \langle (a',b) \rangle, \gamma' ] \ \times \\ 
	& \qquad\qquad\qquad\qquad\qquad\qquad   \bang A[ \bang \pi_1 (\gamma'), \alpha] \times 
		\bang B [ \bang \pi_2 (\gamma'), \beta] \times  \bang (A \otimes B) [\gamma, \langle (a', b) \rangle] \\
	& \qquad \Rightarrow \int^{\alpha \in \bang A} \int^{\gamma' \in \bang (A \otimes B)}  \bang A[  \alpha, \langle  a  \rangle]  \times \bang (A \otimes B)[ \gamma, \gamma' ] \times \bang A[ \bang \pi_1 (\gamma'), \alpha]  \times  \bang B [ \bang \pi_2 (\gamma'), \beta] \\
	& \qquad \cong \bang A [ \bang \pi_1 (\gamma), \langle  a  \rangle] \times \bang B[ \bang \pi_2( \gamma), \beta] . 
\end{align*}
This equals the 2-cell $\sigma'_{A,B}$ constructed above and therefore $\sigma_{A,B}$ is invertible.
\end{proof}

\begin{theorem} \label{thm:sym-has-differentiation}
The pseudonatural transformation $\coder$ is a codereliction operation in $\Prof$.
\end{theorem}  

\begin{proof} The claim follows by \cref{thm:lemma1}, \cref{thm:lemma2}, and \cref{thm:lemma3}.
\end{proof}

\subsection*{Categorical symmetric sequences and analytic functors}  Let us continue to consider
$\Prof$ equipped with the linear exponential pseudocomonad~$(\bang, \dig, \der)$ of \cref{thm:prof-degenerate-model}. By definition, a \myemph{categorical symmetric sequence} $F \co A \to B$ 
is a Kleisli map from $A$ to $B$ for this pseudocomonad, \ie a profunctor $F \co \bang A \to B$.
For $A = B = \mathsf{1}$, we obtain a functor $F \co \PP \to \Set$, \ie a symmetric sequence
in the usual sense~\cite{KellyG:opemay}. The bicategory of categorical symmetric sequences $\Sym$ is defined to be the Kleisli bicategory of the pseudocomonad:
\[
\Sym \defeq \Prof_{\oc} .
\]
Explicitly, the objects of $\Sym$ are small categories and, for small categories $A$ and $B$, we have
\[
\Sym[A,B] \defeq \Prof[\bang A, B]  .
\]

\cref{thm:coKleisli-cartesian-closed} provides an abstract proof of the main result in~\cite{FioreM:carcbg}, namely that the bicategory~$\Sym$ is cartesian closed. The product $A \with B$ of 
$A, B \in \Sym$ is defined as their biproduct 
 in~$\Prof$, \ie $A \with~B \defeq~A \oplus B$, the terminal object $\term$ of $\Sym$ is the zero object 
 of~$\Prof$, \ie $\term \defeq 0$, and the exponential $A \Rightarrow B$ of  $A, B \in \Sym$ is given by~$A \Rightarrow B \defeq \bang A^\op \otimes B$ by \cref{thm:compact-closed-hom} and~\eqref{equ:dual-in-prof}.  All the relevant structure can be given explicitly, see~\cite{FioreM:carcbg} for details.

The codereliction operation in~\eqref{equ:codereliction-def} allows us to define explicitly the derivative of a categorical symmetric sequence. Given $F \co A \to B$ in $\Sym$,  its derivative 
is the profunctor~$\mathrm{d} F \co A \times \oc A \to B$ defined by
\[
\begin{tikzcd}[column sep = large]
A \times \oc A  \ar[r, "\coder_A \times \id_A"] &
\oc A \times \oc A \ar[r, "\cocon_A"] &
\oc A \ar[r, "F"] & B .
\end{tikzcd}
\]
Unfolding the relevant definitions, we obtain
\begin{equation}
\label{equ:derivative-analytic}
\mathrm{d}F (b, (a, \alpha)) \defeq F(b,  \alpha \sqcup \langle a \rangle) .
\end{equation}

We conclude by applying our results to analytic functors~\cite{JoyalA:fonaes}. 
This is of independent interest and also, as suggested in  \cref{sec:intro}, makes more compelling the view of a categorical symmetric sequence $F$ as a non-linear map with $\mathrm{d}F$ as its derivative.
Recall from~\cite{FioreM:carcbg} that the  analytic functor 
 associated to a symmetric sequence $F \co A \to B$ is defined by letting
\[
FX(b) = \coend^{\alpha \in \oc A} F(b,\alpha) \times X^{\alpha} ,
\]
where $X^{\alpha} \defeq X(a_1) \times \ldots \times X(a_n)$, for $\alpha= \langle a_1, \ldots, a_n\rangle$. When
$A = B = \mathsf{1}$ and the symmetric sequences become functors $F \co \PP \to \Set$, we obtain
exactly the analytic functors introduced in~\cite{JoyalA:fonaes}. Analytic functors
between presheaf categories can be regarded as many-variable analytic functions, while analytic functors
on $\Set$ can be seen as single-variable analytic functions.

Taking the derivative as in~\eqref{equ:derivative-analytic} and using the internal hom of the compact
closed structure of $\Prof$ in~\eqref{equ:lin-hom-in-prof}, the profunctor $\mathrm{d}F \co A \otimes \bang A \to B$ may be regarded as a categorical symmetric
 sequence $\mathrm{d} F \co A \to (A \linhom B)$, and thus we can consider its associated analytic functor $\mathrm{d}F \co  \pshA \to \psh(A \linhom B)$, which is given by
 \[
\mathrm{d}FX (b,a) = \coend^{\alpha \in \oc A} F(b,\alpha \sqcup  \langle a \rangle) \times X^{\alpha} .
 \]
This reduces to Joyal's formula for the derivative of an analytic functor~\cite{JoyalA:fonaes} when $A = B = \mathsf{1}$.

\cref{thm:sym-has-differentiation} then implies that the counterparts of the rules for many-variable 
differential calculus, including the Leibniz and chain rules, hold for analytic functors between
presheaf categories, just as the rules for one-variable calculus hold for analytic functors
on $\Set$~\cite{JoyalA:fonaes}. Indeed, the axioms for a codereliction in
\cref{thm:codereliction} and their consequences in
\cref{thm:consequences-of-dereliction} can be stated explicitly and be
identified with familiar rules of differential calculus, \cf \cite{SeelyEtAl,FioreM:difsmi}.

\subsubsection*{Acknowledgements} 
Nicola Gambino wishes to thank Zeinab Galal, Adrian Miranda, and Federico Olimpieri for helpful discussions. We are grateful to Nathanael Arkor and the anonymous referee for helpful comments on the paper.

 Marcelo Fiore acknowledges that this material is based upon work supported by EPSRC via grant EP/V002309/1.
Nicola Gambino acknowledges that this material is based upon work supported by the US Air Force Office for Scientific Research under award number FA9550-21-1-0007,  EPSRC via grant EP/V002325/2, and ARIA via grant MSAI-PR01-P12.


\end{document}